%% file: main.tex
\documentclass[12pt]{article} 

\usepackage[utf8]{inputenc}
\usepackage[T1]{fontenc}
\usepackage{newunicodechar}
\newunicodechar{̦}{}
\usepackage{appendix}
\usepackage{longtable}
\usepackage{amsmath}
\usepackage{amsthm}
\usepackage{amsfonts}
\usepackage{amssymb}
\usepackage{breqn}
\usepackage{xcolor}
\usepackage{xifthen}
\usepackage{parskip}
\usepackage{verbatim}
\usepackage[style=numeric, url=false]{biblatex}
\usepackage[left=2.5cm, right=2.5cm,bottom=4.5cm]{geometry}
\allowdisplaybreaks

\input{alphabets}

\input{environments}

\let\[\relax
\let\]\relax

\newcommand{\pushright}[1]{\ifmeasuring@#1\else\omit\hfill$\displaystyle#1$\fi\ignorespaces}

\let\d\relax
\DeclareMathOperator{\d}{d\!}
\newcommand{\D}[1]{\,\textnormal{D}_{#1}}
\newcommand{\ind}{\mathbf{1}}

\renewcommand{\l}{\left}
\renewcommand{\r}{\right}

\newcommand\numberthis{\addtocounter{equation}{1}\tag{\theequation}}

\renewcommand{\theta}{\vartheta}
\renewcommand{\phi}{\varphi}
\renewcommand{\epsilon}{\varepsilon}

\DeclareMathOperator{\E}{\hat{\Ebb}}
\newcommand{\M}{\textnormal{M}_{\ast}}

\renewcommand{\H}{\textnormal{H}_{\ast}}
\newcommand{\F}{\textnormal{B}_b}
\renewcommand{\L}{\textnormal{L}_{\ast}}
\renewcommand{\S}{\textnormal{M}_b}
\newcommand{\C}{\textnormal{C}}

\newcommand{\Lip}{\textnormal{Lip}}

\begin{document}
    \begin{center}        
        {\LARGE Regularity of Solutions of Mean-Field $G$-SDEs}\vspace{.6cm}
        
        {\large Karl-Wilhelm Georg Bollweg, Thilo Meyer-Brandis}\vspace{.3cm}
        \section*{Abstract}
    	\begin{minipage}{12cm}
                We study regularity properties of the unique solution of a mean-field $G$-SDE. More precisely, we consider a mean-field $G$-SDE with square-integrable random initial condition and establish its first and second order Fréchet differentiability in the random initial condition and specify the $G$-SDEs of the respective Fréchet derivatives.
    	\end{minipage}
    \end{center}

    \section{Introduction}
    Mean-field stochastic differential equations have emerged as a powerful mathematical framework for modeling the dynamics of large populations of interacting agents subject to random perturbations. Their significance lies in their ability to capture both the individual stochastic behavior of agents and the macroscopic effects of collective interactions, making them essential tools in fields such as physics, biology, economics, and quantitative finance. In particular, mean-field SDEs serve as the probabilistic counterparts of mean-field control problems and mean-field games, where the system's evolution depends not only on the individual state but also on the distribution of the population. The pioneering work of Kac \cite{kac_foundations_1956} introduced the mean-field approach in the context of kinetic theory, while McKean \cite{mckean_class_1966} first formalized nonlinear Markov processes whose dynamics depend on their own law. Since then, mean-field SDEs have been extensively studied and generalized, with foundational contributions by Sznitman \cite{sznitman_topics_1991} on propagation of chaos and Lasry and Lions \cite{lasry_jeux_2006, lasry_jeux_2006-1} and Carmona and Delarue \cite{carmona_probabilistic_2018, carmona_probabilistic_2017} on mean-field games and controls. These equations also underpin numerous modern applications, from systemic risk modeling in finance to synchronisation phenomena in neuroscience, underscoring their broad relevance and mathematical richness.
    
    In the 2000s, Shige Peng introduced the theory of sublinear expectations and, as special case, the $G$-setting as framework to study Knightian uncertainty, cf. \cite{peng_multi-dimensional_2008, peng_g-expectation_2006, peng_new_2008, benth_g-expectation_2007}. There have been significant advancements in the theory of sublinear expectations and the $G$-setting in recent years. For instance, \cite{nutz_constructing_2013}, \cite{cohen_quasi-sure_2012}, \cite{hollender_levy-type_2016}, \cite{peng_filtration_2004} study the construction of sublinear expectations and their properties, and \cite{hu_g-levy_2021}, \cite{fadina_affine_2019}, \cite{neufeld_nonlinear_2016}, \cite{biagini_non-linear_2023} study different classes of stochastic processes in a sublinear expectation framework. 
    A sublinear expectation can be thought of the "worst" outcome within a class of models. The $G$-setting is used to quantify volatility uncertainty and consists of the so called $G$-Brownian motion and the $G$-expectation.
    Besides the probabilistic interpretation of quantifying Knightian uncertainty, there is a strong connection between sublinear expectations and fully non-linear partial differential equations. 
    This has been extensively studied in e.g. \cite{peng_bsde_2016}, \cite{hu_bsdes_2022}, \cite{li_doubly_2025}, \cite{hu_backward_2014}.
    
    Recently, the extension of mean-field theory to the $G$-expectation framework has received increased attention.
    First attempts in that direction can be found in \cite{sun_mean-field_2020} and \cite{sun_distribution_2023}. 
    In \cite{sun_mean-field_2020}, the author considers a SDE of the form
    \begin{align*}
        \d X_t &=  \E\l[ b\!\l(t,x,X_t\r) \r] \Big|_{x=X_t} \!\!\! \d t + \E\l[ h\!\l(t,x,X_t\r) \r] \Big|_{x=X_t} \!\!\! \d \l<B\r>_t +  \E\l[ h\!\l(t,x,X_t\r) \r] \Big|_{x=X_t} \!\!\! \d B_t, & 0\leq t\leq T, \\
        X_0 &= x,
    \end{align*}
    where $b, h, g:\,[0,T]\times \Rbb \times \Rbb \rightarrow \Rbb$, $B$ denotes a one-dimensional $G$-Brownian motion and $\E$ denotes the corresponding $G$-expectation.
    More details on the $G$-setting are provided in Section~\ref{sec:setting} or can be found in \cite{peng_nonlinear_2019}. 
    Let $\L^{2,d}$ denote the space of all $\Rbb^d$-valued random vectors $\xi$ with finite sublinear second moment $\E\l[ \l\|\xi\r\|^2 \r]<\infty$.
    For $\xi\in \L^{2,d}$, the functional $F_\xi$ defined by
    \begin{align*}
        F_\xi:&\, \textnormal{Lip}(\Rbb^d) \rightarrow \Rbb, & \phi \mapsto F_{\xi}(\phi)&:=\E\l[ \phi(\xi)\r]
    \end{align*}
    can be interpreted as the "sublinear distribution" of $\xi$. 
    In \cite{sun_distribution_2023}, the approach from \cite{sun_mean-field_2020} is extended to higher dimensions and to coefficients that depend on the sublinear distribution $F_{X_t}$ of the $d$-dimensional solution process $X_t$. 
    That is, the authors consider a SDE of the form
    \begin{align*}
        \d X_t &=  b\!\l(t,X_t,F_{X_t}\r) \d t + h\!\l(t,X_t,F_{X_t}\r) \d \l<B\r>_t +  g\!\l(t,X_t,F_{X_t}\r) \d B_t, & 0\leq t\leq T, \\
        X_0 &= x. 
    \end{align*}
    In \cite{sun_distribution_2023}, the authors define a space containing all sublinear distributions and endow it with a metric allowing them to define continuity conditions on the coefficients. However, the space of sublinear distributions is not a vector space and, thus, it does not have a natural notion of differentiability which limits the study of regularity properties of the solution. 
    
    In \cite{bollweg_mean-field_2025}, a novel formulation of a mean-field $G$-SDEs was introduced in which the coefficients depend on the solution process as random variable.
    More precisely, the authors consider a $G$-SDE of the form
    \begin{align*}
        \d X^{t,\xi}_s&= b\!\l(s,\omega, x, X^{t,\xi}_s\r)\Big|_{x=X^{t,\xi}_s}  \d s + h\!\l(s,\omega, x, X^{t,\xi}_s\r)\Big|_{x=X^{t,\xi}_s}  \d \l<B\r>_s
        \\&\quad 
        + g\!\l(s,\omega, x, X^{t,\xi}_s\r)\Big|_{x=X^{t,\xi}_s} \d B_s, & t\leq s\leq T, \\
        X^{t,\xi}_t&=\xi
        \numberthis\label{eq:MF-SDE-0}
    \end{align*}
    with coefficients defined on $[0,T]\times \Omega\times \Rbb^d\times \L^{2,d}$ and initial data $\xi\in\L^{2,d}$.
    This formulation generalises the formulations introduced in \cite{sun_mean-field_2020}, \cite{sun_distribution_2023} where the coefficients depend on the sublinear distribution.
    Moreover, $\L^{2,d}$ is a Banach space and, thus, the formulation in \cite{bollweg_mean-field_2025} comes with standard notions of differentiability.
    
    In this paper, we are interested in regularity properties of the solution of a mean-field SDE driven by $G$-Brownian motion. While the formulation in \cite{sun_distribution_2023} is closer to the classical formulation as it depends on the (sublinear) distribution of the solution process, we work with the formulation introduced in \cite{bollweg_mean-field_2025} since it allows us to consider Fréchet differentiable coefficients and study the Fréchet differentiability of the solution $X^{t,\xi}$ of \eqref{eq:MF-SDE-0} with respect to the random initial condition $\xi$.
    The Fréchet derivatives of $X^{t,\xi}$ capture how perturbations of the initial data propagate through the stochastic system and, thus, they are crucial for studying the sensitivity of the solution process with respect to changes in the initial data.
    This sensitivity analysis is a central tool for a wide range of applications.
    For instance, the Fréchet derivatives can be used to derive optimality conditions for stochastic control problems or establish recursive formulae for conditional expectations using the dynamic programming principle.     
    Further, the Fréchet derivatives of $X^{t,\xi}$ can be used in numerical approximations of $X^{t,\xi}$ as well as for (sub)gradient methods for optimisation problems.    
    
    For simplicity and conciseness, we use the following notation.
    \begin{notation}
        For a function $f$ on $[0,T]\times\Omega\times\Rbb^d\times\L^{2,d}$, define
        \begin{equation*}
            f\!\l(s,\omega,\eta,\xi\r):=f\!\l(s,\omega,\eta(\omega),\xi\r)=f\!\l(s,\omega,x,\xi\r)\Big|_{x=\eta(\omega)}
        \end{equation*}
        for any $0\leq s\leq T$, $\omega\in\Omega$ and $\xi,\eta\in\L^{2,d}$. 
        Often, we suppress the explicit dependence on $\omega$ and write $f\!\l(s,\eta,\xi \r)$ instead of $f\!\l(s,\omega,\eta,\xi\r)$.
    \end{notation}
    Thus, \eqref{eq:MF-SDE-0} can be written as
    \begin{align*}
        \d X^{t,\xi}_s
        &=
        b\!\l(s,X^{t,\xi}_s,X^{t,\xi}_s \r) \d s
        + h\!\l(s,X^{t,\xi}_s,X^{t,\xi}_s \r) \d \l<B\r>_s
        + g\!\l(s,X^{t,\xi}_s,X^{t,\xi}_s \r) \d B_s, 
        & t\leq s\leq T, \\ 
        X^{t,\xi}_t &= \xi. \numberthis\label{eq:MF-GSDE-xi-0}
    \end{align*}
    Under mild assumptions on the coefficients, it is shown in \cite{bollweg_mean-field_2025} that \eqref{eq:MF-GSDE-xi} admits a unique solution $X^{t,\xi}$, cf. Theorem 3.12 in \cite{bollweg_mean-field_2025}. 
    For $x\in \Rbb^d$, we associate to $X^{t,\xi}$ the $G$-SDE
    \begin{align*}
        \d X^{t,x,\xi}_s 
        &= 
        b\!\l(s,X^{t,x,\xi}_s,X^{t,\xi}_s \r) \d s
        + h\!\l(s,X^{t,x,\xi}_s,X^{t,\xi}_s \r) \d \l<B\r>_s
        + g\!\l(s,X^{t,x,\xi}_s,X^{t,\xi}_s \r) \d B_s, 
        & t\leq s\leq T, \\
        X^{t,x,\xi}_t &= x \numberthis\label{eq:MF-GSDE-x-0}
    \end{align*}
    with deterministic initial condition $x\in \Rbb^d$.
    The $G$-SDEs \eqref{eq:MF-GSDE-xi-0} and \eqref{eq:MF-GSDE-x-0} are closely connected.
    More precisely, if \eqref{eq:MF-GSDE-xi-0} and \eqref{eq:MF-GSDE-x-0} admit each a unique solution, then the process $X^{t,\xi}$ can be obtained from $X^{t,x,\xi}$ by evaluating at $x=\xi$ as formalised in Lemma~\ref{lem:concatenation-identity}. 
    This allows us to infer properties of $X^{t,\xi}$ from properties of $X^{t,x,\xi}$ using the aggregation property of the conditional sublinear expectation.
    Thus, many of our auxiliary results are formulated in terms of conditional sublinear expectations.
    
    Our main contribution is the derivation of first and second order Fréchet derivatives of the solution process as formalised in Propositions~\ref{prp:x-differentiability}, \ref{prp:xi-differentiability-xi}, \ref{prp:xi-differentiability-x} and \ref{prp:Dx2-differentiability}.
    For coefficients with Lipschitz and bounded Fréchet derivative, we establish the Fréchet differentiability of of $X^{t,x,\xi}$ and $X^{t,\xi}$. 
    Moreover, we characterise each of the Fréchet derivatives of $X^{t,x,\xi}$ and $X^{t,\xi}$ as the unique solution of a $G$-SDE.
    These results are in line with the results on classical mean-field SDEs, cf. \cite{buckdahn_mean-field_2014}.

    This paper is structured as follows.
    In Section~\ref{sec:setting}, we recall the $G$-framework before establishing preliminary results such as continuity and growth properties of the solution map $(x,\xi)\mapsto (X^{t,\xi},X^{t,x,\xi})$ in Section~\ref{sec:prelim}.
    Section~\ref{sec:derivative-1} is devoted to the first order Fréchet derivatives of the solution map in $x$ and $\xi$ while the second order derivatives are studied in Section~\ref{sec:derivative-2}.
    Finally, in Section~\ref{sec:comparison}, we show how the formulation in \cite{sun_distribution_2023} can be embedded into the formulation in \cite{bollweg_mean-field_2025} and develop a notion of differentiability for maps on the space of sublinear distributions.
    
    \begin{notation}
        Most of our results are obtained via approximations and the Grönwall inequality. For the sake of conciseness and readability, we use the symbol $\lesssim$ to denote proportionality in the following sense. 
        
        For two maps $f,g:\,\Theta \rightarrow \Rbb$ with domain $\Theta$, we define 
        \begin{equation*}
            f(\theta) \lesssim g(\theta) \qquad :\Longleftrightarrow \qquad \exists C\geq 1:\,\forall \theta\in \Theta:\, f(\theta)\leq C\,g(\theta).
        \end{equation*}
    \end{notation}

    \section{Setting}\label{sec:setting}
    In this section, we recall the generalized $G$-framework as introduced in Chapter~8 in \cite{peng_nonlinear_2019}. 
    Fix $n\geq 1$ and let $\Omega:=\textnormal{C}_0(\Rbb_+,\Rbb^n)$ denote the space of all continuous $\Rbb^n$-valued paths starting at the origin equipped with the topology of uniform convergence.
    Let $\CF$ denote the corresponding Borel $\sigma$-algebra.
    Moreover, let $\Fbb=(\CF_t)_{t\geq 0}$ denote the natural filtration generated by the coordinate mapping process $B:\,\Rbb_+\times\Omega\rightarrow\Rbb^n$ given by $B_t(\omega)=\omega(t)$.    
    
    Fix a convex and compact set $\Sigma\subseteq \mathbb{S}_+^n$ of symmetric non-negative definite $n\times n$-matrices and set
    \begin{equation*}
        \mathcal{A}^\Sigma:=\Big\{ \theta=(\theta_t)_{t\geq 0}\,:\,\theta \text{ is $\Sigma$-valued and $\mathbb{F}$-progressively measurable}\Big\}.
    \end{equation*}
    Let $P_0$ denote the Wiener measure on $(\Omega,\CF)$, and define
    \begin{equation*}
        \CP:=\l\{ P_0 \circ \l( \theta \bullet B\r)^{-1}\,:\, \theta \in\CA^{\Sigma}\r\},
    \end{equation*}
    where $\theta\bullet B:=\int_0^\cdot \theta_s \d B_s$ denotes the It\^o integral with respect to the stochastic basis $(\Omega,\CF,\Fbb,P_0)$.

    For $d\geq 1$ and a $\sigma$-algebra $\mathcal G\subseteq \CF$, let $\F^d(\mathcal G)$ denote the space of all bounded $\mathcal G$-measurable maps $\xi:\,\Omega\rightarrow \Rbb^d$.
    The set of probability measures $\CP$ induces an upper expectation on $\F(\CF):=\F^1(\CF)$, namely
    \begin{equation*}
        \E:\,\F(\CF) \rightarrow \Rbb,\qquad \xi\mapsto\E\!\l[ \xi \r]:=\sup_{P\in\CP} E_P\!\l[\xi\r],
    \end{equation*}
    where $E_P$ denotes the linear expectation with respect to $P$.
    The process $B$ is a $G$-Brownian motion with respect to $\E$ and $(\Omega,\F(\CF),\E)$ is a sublinear expectation space.
    For $p\geq 1$, define the norm
    \begin{equation*}
        \|\cdot\|_{\L^{p}}:\,\F^d(\CF) \rightarrow\Rbb_+,\qquad \xi\mapsto\|\xi\|_{\L^{p}}:=\E\!\l[\l\| \xi\r\|^p\r]^{\frac{1}{p}},
    \end{equation*}
    where $\l\|\cdot \r\|$ denotes the Euclidean norm on $\Rbb^d$ and let $\L^{p,d}(t)$ and $\L^{p,d}$ denote the completion of $\F^d(\CF_t)$ and $\F^d(\CF)$ with respect to $\|\cdot\|_{\L^{p}}$ for $t\geq 0$. We set $\L^p(t):=\L^{p,1}(t)$ and $\L^p:=\L^{p,1}$.   

    For $d\geq 1$ and $T>0$, let $\S^d(0,T)$ denote the space of all maps $X:\,[0,T]\times\Omega\rightarrow \Rbb^d$ of the form
    \begin{equation*}
        X=\sum_{k=0}^{m-1} \xi_k \ind_{\l[t_k,t_{k+1}\r)}
    \end{equation*}
    with $m\in\mathbb{N}$, $0=t_0<\ldots<t_m=T$, and $\xi_k\in\F^d(\CF_{t_k})$ for all $0\leq k\leq m-1$.
    For $p\geq 1$, define the norms
    \begin{align*}
        \|\cdot\|_{\M^{p}}:\,\S^d(0,T)\rightarrow\Rbb_+, \qquad \|X\|_{\M^{p}}&:= \l(\int_0^T \E\!\l[ \l\| X_s\r\|^p  \r] \d s\r)^{\frac{1}{p}}, \\
        \|\cdot\|_{\H^{p}}:\,\S^d(0,T)\rightarrow\Rbb_+,\qquad \|X\|_{\H^{p}}&:= \E\!\l[ \sup_{0\leq s\leq T} \l\| X_s\r\|^p  \r] ^{\frac{1}{p}},
    \end{align*}
    and let $\M^{p,d}(0,T)$ and $\H^{p,d}(0,T)$ denote the completion of $\S^d(0,T)$ with respect to $\|\cdot\|_{\M^{p}}$ and $\|\cdot\|_{\H^{p}}$ respectively.
    Clearly, $\H^{p,d}(0,T)\subseteq \M^{p,d}(0,T)$, and we set $\M^p(0,T):=\M^{p,1}(0,T)$, $\H^p(0,T):=\H^{p,1}(0,T)$.

    Set $\S(0,T):=\S^1(0,T)$ and let $B^i$ denote the $i$-th component of $B$ for $1\leq i\leq n$. Define the map $\mathcal I_i:\,\S(0,T) \rightarrow \L^2(T)$ by
    \begin{equation*}
        \mathcal I_i(X):= \int_0^T X_s \d B^i_s := \sum_{k=0}^{m-1} \xi_k \l( B^i_{t_{k+1}} - B^i_{t_k}\r)
    \end{equation*}
    for each
    \begin{equation*}
        X = \sum_{k=0}^{m-1} \xi_k \ind_{[t_k,t_{k+1})}.
    \end{equation*}
    The map $\mathcal I_i$ is linear and continuous with respect to $\l\|\cdot\r\|_{\M^2}$ and, thus, can be uniquely continuously extended to $\M^2(0,T)$.
    For $0\leq t\leq s\leq T$ and $X\in \M^2(0,T)$, define
    \begin{equation*}
        \int_t^s X_u \d B^i_u := \mathcal I_i(X\ind_{[t,s)}).
    \end{equation*}
    The quadratic variation of $B$ is a map $\l<B\r>:\,\Rbb_+\times \Omega \rightarrow \mathbb S^n_+$ defined componentwise by
    \begin{equation*}
        \l<B^i,B^j\r>_t := B^i_t B^j_t - \int_0^t B^i \d B^j_s - \int_0^t B^j \d B^i_s, \qquad t\geq 0
    \end{equation*}
    for $1\leq i,j\leq n$. For $1\leq i,j\leq n$, define the map $\mathcal Q_{ij}:\,\S(0,T)\rightarrow \L^1(T)$ by
    \begin{equation*}
        \mathcal Q_{ij}(X):= \int_0^T X_s \d \l<B^i,B^j\r>_s := \sum_{k=0}^{m-1} \xi_k \l( \l<B^i,B^j\r>_{t_{k+1}} - \l<B^i,B^j\r>_{t_k}\r)
    \end{equation*}
    for each
    \begin{equation*}
        X = \sum_{k=0}^{m-1} \xi_k \ind_{[t_k,t_{k+1})}.
    \end{equation*}
    The map $\mathcal Q_{ij}$ is linear and continuous with respect to $\l\| \cdot \r\|_{\M^1}$ and, thus, can be uniquely continuously extended to $\M^1(0,T)$.
    For $0\leq t\leq s\leq T$ and $X\in \M^1(0,T)$, define
    \begin{equation*}
        \int_t^s X_u \d \l<B^i,B^j\r>_u := \mathcal Q_{ij}(X\ind_{[t,s)}).
    \end{equation*}

    Let $\beta_k, \alpha_{kij}\in \M^1(0,T)$ and $\gamma_{ki} \in \M^2(0,T)$ for $1\leq k\leq d$, $1\leq i,j \leq n$. We say that $X\in \M^{1,d}(0,T)$ satisfies
    \begin{align*}
        \d X_s &= \beta(s) \d s + \alpha(s) \d \l<B\r>_s + \gamma(s) \d B_s, \qquad t\leq s\leq T,
    \end{align*}
    if the components $X^k$, $1\leq k\leq d$, satisfy 
    \begin{equation*}
        X^k_s - X^k_t = \int_t^s \beta_k(u) \d u + \sum_{i,j=1}^n \int_t^s \alpha_{kij}(u) \d\l<B^i,B^j\r>_u + \sum_{i=1}^n \int_t^s \gamma_{ki}(u) \d B^i_u
    \end{equation*}
    quasi-surely for all $t\leq s\leq T$. 
    
    For a $G$-SDE with initial condition $X_t=\xi$, we are not interested in the behavior on $0\leq s< t$. Thus, we reduce our attention to the space
    \begin{align*}
        \H^{p,d}(t,T)&:=\l\{ X\in \H^{p,d}(0,T)\,:\, \E\l[ \sup_{0\leq s < t}\l\|X_s\r\|^p\r]=0\r\}.
    \end{align*}    
    We say that the $G$-SDE
    \begin{align}
        \d X_s &= b\!\l(s,X_s\r) \d s + h\!\l(s,X_s\r) \d \l<B\r>_s + g\!\l(s,X_s\r) \d B_s, & t\leq s\leq T, \label{eq:GSDE}\\
        X_t &= \xi_t \label{eq:initial}
    \end{align}
    admits a unique solution $X\in\H^{2,d}(t,T)$ if there exists a $X\in\H^{2,d}(t,T)$ that satisfies \eqref{eq:GSDE} with $X_t=\xi$ quasi-surely and, for any $X,Y\in \H^{2,d}(t,T)$ that satisfy \eqref{eq:GSDE} with $X_t=Y_t=\xi$ quasi-surely, we have $\l\|X-Y\r\|_{\H^2}=0$.

    \section{Preliminary Results}\label{sec:prelim}
    
    In this section, we establish growth and continuity properties of the solution map under the following assumptions on the coefficients.
    \begin{ass}\label{ass:1-lipschitz}
        Let $b:\, [0,T]\times\Omega\times\Rbb^d\times\L^{2,d}\rightarrow\Rbb^d$, $h:\, [0,T]\times\Omega\times\Rbb^d\times\L^{2,d}\rightarrow\Rbb^{d\times n\times n}$, and $g:\, [0,T]\times\Omega\times\Rbb^d\times\L^{2,d}\rightarrow\Rbb^{d\times n}$ be such that the following holds for all components $f=b_k,h_{kij},g_{ki}$, $1\leq i,j\leq n$, $1\leq k\leq d$.

        \begin{enumerate}
            \item 
            We have $f(\cdot,x,\xi)\ind{[s,T]}\in\M^2(0,T)$ for all $x\in\Rbb^d$, $\xi\in\F^d(\CF_s)$ and $0\leq s\leq T$.
            
            \item 
            There exists a $q_0$-integrable $\alpha_0:\,[0,T]\rightarrow [1,\infty)$ with $q_0\geq 2$ such that
            \begin{align*}
                \l|f\!\l(s,\omega, x,\xi\r) - f\!\l(s,\omega, y,\eta\r)\r|
                &\leq \alpha_0\!\l(s\r) \l( \l\|x-y\r\| +\l\|\xi-\eta\r\|_{\L^{2}}\r)
            \end{align*}
            for all $x,y\in\Rbb^d$, $\xi,\eta\in\L^{2,d}$, $0\leq s\leq T$ and $\omega\in\Omega$.
        \end{enumerate}
    \end{ass}
    For convenience, let us define the set of coefficients
    \begin{align*}
        F:=\l\{ b_k, h_{kij}, g_{ki}\,:\,1\leq k\leq d,\; 1\leq i,j\leq n\r\}.
    \end{align*}
    
    \begin{cor}\label{cor:growth}
        If Assumption~\ref{ass:1-lipschitz} is satisfied, then the following holds for all components $f=b_k,h_{kij},g_{ki}$, $1\leq i,j\leq n$, $1\leq k\leq d$. 
        
        There exists an integrable $\kappa:\,[0,T] \rightarrow [1,\infty)$ and a process $K\in \M^1(0,T)$ such that
        \begin{equation*}
            \l| f\l(s,\omega,x,\xi\r) \r|^2 
            \leq \kappa\!\l(s\r)\l( \l\|x\r\|^2 + \l\|\xi\r\|_{\L^{2}}^2\r) + K_s(\omega)
        \end{equation*}
        for all $x\in\Rbb^d$, $\xi\in\L^{2,d}$, $0\leq s\leq T$ and $\omega\in\Omega$.
    \end{cor}
    \begin{proof}
        The continuity condition in Assumption~\ref{ass:1-lipschitz}.2 implies
        \begin{align*}
            \l| f\!\l(s,\omega,x,\xi\r)\r|^2
            %&= 
            %\l| f\!\l(s,\omega,x,\xi\r) - f\!\l(s,\omega, 0,0\r) + f\!\l(s,\omega, 0,0\r)\r|^2
            % Jensen's inequality
            %\\
            &\leq 
            2 \l|f\!\l(s,\omega, x,\xi\r) - f\!\l(s,\omega, 0,0\r)\r|^2 + 2 \l|f\!\l(s,\omega, 0,0\r)\r|^2 
            % ass:1 
            %\\&\leq 
            %2 \alpha_0\!\l(s\r)^2 \l( \l\|x\r\| +\l\|\xi\r\|_{\L^{2}}\r)^2 + 2 \l|f\!\l(s,\omega, 0,0\r)\r|^2
            % Jensen's inequality
            \\&\leq 
            4 \alpha_0\!\l(s\r)^2 \l( \l\|x\r\|^2 +\l\|\xi\r\|_{\L^{2}}^2\r) + 2 \l|f\!\l(s,\omega, 0,0\r)\r|^2,
        \end{align*}
        and, clearly, $\kappa:=4\alpha_0^2\geq 1$ is integrable.
        
        Finally, Assumption~\ref{ass:1-lipschitz}.1 implies that $K:=\l| f\!\l(\cdot,0,0\r)\r|^2\in\M^1(0,T)$, where $0$ denotes the origin in $\Rbb^d\subseteq \L^{2,d}(0)$.
    \end{proof}

    Thus, we conclude that Assumption~\ref{ass:1-lipschitz} is stronger than Assumption~3.1 in \cite{bollweg_mean-field_2025} and, thus, Theorem~3.12 in \cite{bollweg_mean-field_2025} immediately yields the existence of unique solutions.

    \begin{prp}
        If Assumption~\ref{ass:1-lipschitz} is satisfied, then the $G$-SDEs
        \begin{align*}
            \d X^{t,\xi}_s
            &=
            b\!\l(s,X^{t,\xi}_s,X^{t,\xi}_s \r) \d s
            + h\!\l(s,X^{t,\xi}_s,X^{t,\xi}_s \r) \d \l<B\r>_s
            + g\!\l(s,X^{t,\xi}_s,X^{t,\xi}_s \r) \d B_s, 
            & t\leq s\leq T, \\ 
            X^{t,\xi}_t &= \xi. \numberthis\label{eq:MF-GSDE-xi}
            \\
            \d X^{t,x,\xi}_s 
            &= 
            b\!\l(s,X^{t,x,\xi}_s,X^{t,\xi}_s \r) \d s
            + h\!\l(s,X^{t,x,\xi}_s,X^{t,\xi}_s \r) \d \l<B\r>_s
            + g\!\l(s,X^{t,x,\xi}_s,X^{t,\xi}_s \r) \d B_s, 
            & t\leq s\leq T, \\
            X^{t,x,\xi}_t &= x \numberthis\label{eq:MF-GSDE-x}
        \end{align*}
        admit unique solutions $X^{t,\xi}, X^{t,x,\xi} \in\H^{2,d}(t,T)$.
    \end{prp}
    
    In particular, we deduce that the solution map
    \begin{equation*}
        \Rbb^d\times \L^{2,d}(t) \rightarrow \H^{2,d}(t,T) \times \H^{2,d}(t,T),
        \qquad
        (x,\xi) \mapsto (X^{t,x,\xi},X^{t,\xi})
    \end{equation*}
    is well-defined. Further, Corollary~\ref{cor:growth} implies that the solution map is of linear growth. More precisely, we have the following growth properties.

    \begin{lem}\label{lem:H2-bound-xi}
        If Assumption~\ref{ass:1-lipschitz} is satisfied, then we have
        \begin{equation*}
            \E\l[ \sup_{t\leq w\leq T} \l\| X^{t,\xi}_w\r\|^2 \r]  \lesssim 1+\l\| \xi \r\|_{\L^2}^2
        \end{equation*}
        for all $0\leq t\leq T$ and $\xi\in\L^{2,d}(t)$.
    \end{lem}
    \begin{proof}
        By Lemma~\ref{lem:conditional} and Corollary~\ref{cor:growth}, we have for all $t\leq s\leq T$
        \begin{flalign*}
            \E\l[ \sup_{t\leq w\leq s}\l\| X^{t,\xi}_w\r\|^2 \r]
            % lem:conditional
            &\lesssim
            \l\| \xi\r\|_{\L^2}^2 + \sum_{f\in F}\int_t^s \E\l[ \l| f\!\l(u,X^{t,\xi}_u,X^{t,\xi}_u\r)\r|^2\r] \d u
            % cor:growth 
            \\&\leq 
            \l\| \xi\r\|_{\L^2}^2 + \int_t^s \E\l[\kappa\!\l(u\r)  \l(\l\|X^{t,\xi}_u\r\|^2 + \l\|X^{t,\xi}_u\r\|_{\L^2}^2\r) +K_u\r] \d u
            % sublinearity E
            %\\&\leq 
            %\int_t^s \E\l[ K_u\r] \d u + \l\| \xi\r\|_{\L^2}^2 + 2\int_t^s \kappa\!\l(u\r)  \l\|X^{t,\xi}_u\r\|_{\L^2}^2 \d u
            % M1-norm, supremum
            %\\&\leq 
            %\l\|K\r\|_{\M^1} + \l\| \xi\r\|_{\L^2}^2 + 2 \int_t^s \kappa\!\l(u\r) \E\l[ \sup_{t\leq w\leq u} \l\|X^{t,\xi}_w\r\|^2 \r] \d u
            \\&\lesssim
            1 + \l\| \xi\r\|_{\L^2}^2 + \int_t^s \kappa\!\l(u\r) \E\l[ \l\|X^{t,\xi}_u\r\|^2 \r] \d u
            ,
        \end{flalign*}
        and Grönwall's inequality yields the desired result.
    \end{proof}

    \begin{lem}\label{lem:H2-bound}
        If Assumption~\ref{ass:1-lipschitz} is satisfied, then there exists a $K\in\M^1(0,T)$ such that
        \begin{equation*}
            \E\l[ \sup_{t\leq w\leq s} \l\| X^{t,x,\xi}\r\|^2 \,\Big|\,\CF_t\r]  \lesssim \int_t^s \E\l[ K_u\,\big|\,\CF_t\r] \d u +\l\| x \r\|^2 +\l\|\xi\r\|_{\L^2}^2
        \end{equation*}
        for all $0\leq t\leq s\leq T$, $x\in\Rbb^d$ and $\xi\in\L^{2,d}(t)$.
    \end{lem}
    \begin{proof}
        By Lemma~\ref{lem:conditional} and Corollary~\ref{cor:growth}, we have
        \begin{flalign*}
            \quad&
            \E\l[ \sup_{t\leq w\leq s}\l\| X^{t,x,\xi}_w\r\|^2 \,\Big|\,\CF_t\r]
            &&\\
            % lem:conditional
            &\lesssim
            \l\| x \r\|^2 + \sum_{f\in F} \int_t^s \E\l[ \l| f\!\l(u,X^{t,x,\xi}_u,X^{t,\xi}_u\r)\r|^2\,\Big|\,\CF_t\r] \d u
            % cor:growth 
            \\&\leq 
            \l\| x \r\|^2 + \int_t^s \E\l[ K_u\,\big|\,\CF_t\r] \d u + \int_t^s \kappa\!\l(u\r) \l( \E\l[ \l\|X^{t,x,\xi}_u\r\|^2 \,\Big|\,\CF_t\r] + \l\|X^{t,\xi}_u\r\|_{\L^2}^2 \r)\d u 
            % L2-H2 inequality
            %\\&\leq
            %\int_t^s \E\l[ K_u\,\big|\,\CF_t\r] \d u + \l\| x \r\|^2 + \l\|X^{t,\xi}\r\|_{\H^2}^2 \int_t^s\kappa\!\l(u\r) \d u  + \int_t^s \kappa\!\l(u\r) \E\l[ \l\| X^{t,x,\xi}_u \r\|^2 \r] \d u
            %\\&\lesssim
            %\l\| x \r\|^2 + \int_t^s \E\l[ K_u\,\big|\,\CF_t\r] \d u + \l\|X^{t,\xi}\r\|_{\H^2}^2  + \int_t^s \kappa\!\l(u\r) \E\l[ \l\| X^{t,x,\xi}_u \r\|^2 \r] \d u
            % lem:H2-bound-xi
            \\&\lesssim
            \l\|x\r\|^2 + \int_t^s \E\l[ K_u\,\big|\,\CF_t\r] \d u + \l\|\xi \r\|_{\L^2}^2 
            + \int_t^s \kappa\!\l(u\r) \E\l[ \l\| X^{t,x,\xi}_u \r\|^2 ~\Big|~\CF_t\r] \d u
            ,
        \end{flalign*}
        where we used Lemma~\ref{lem:H2-bound-xi} in the last step.
        Finally, Grönwall's inequality yields the desired result.
    \end{proof}

    \begin{rmk}
        By taking the sublinear expectation, Lemma~\ref{lem:H2-bound} immediately yields
        \begin{align*}
            \E\l[ \sup_{t\leq w\leq T} \l\| X^{t,x,\xi}\r\|^2 \r]  \lesssim \l\| x \r\|^2 +\l\|\xi\r\|_{\L^2}^2,
        \end{align*}
        which is analogous to the result in Lemma~\ref{lem:H2-bound-xi}.
        Many of the results for $X^{t,x,\xi}$ are stated in a conditional form so that we apply them to the concatenation $X^{t,x,\xi}\big|_{x=\xi}$ which, as we show in Lemma~\ref{lem:concatenation-identity}, is indifferent from $X^{t,\xi}$.
    \end{rmk}

    \begin{lem}\label{lem:xi-eta-2-bound} 
        If Assumption~\ref{ass:1-lipschitz} is satisfied, then 
        \begin{equation*}
            \E \l[ \sup_{t\leq s\leq T} \l\| X^{t,\xi}_s - X^{t,\eta}_s\r\|^2\r] 
            \lesssim \l\|\xi-\eta\r\|_{\L^2}^2
        \end{equation*}
        for all $0\leq t\leq T$ and $\xi,\eta\in\L^{2,d}(t)$.
    \end{lem}
    \begin{proof}
        By Lemma~\ref{lem:conditional}, we have for all $t\leq s\leq T$
        \begin{flalign*}
            \quad&
            \E\l[ \sup_{t\leq w\leq s} \l\| X^{t,\xi}_w - X^{t,\eta}_w\r\|^2\r]
            &&
            % lem:conditional
            \\&\lesssim
            \l\|\xi-\eta\r\|_{\L^2}^2 + \sum_{f\in F} \int_t^s \E\l[ \l| f\!\l(u,X^{t,\xi}_u,X^{t,\xi}_u\r) - f\!\l(u,X^{t,\eta}_u,X^{t,\eta}_u\r) \r|^2 \r] \d u
            % f Lipschitz
            %\\&\leq
            %\l\|\xi-\eta\r\|_{\L^2}^2 + \int_t^s \alpha_0\!\l(u\r)^2 \l( \E\l[ \l\| X^{t,\xi}_u - X^{t,\eta}_u \r\|^2 \r] + \l\|X^{t,\xi}_u - X^{t,\eta}_u\r\|_{\L^2}^2 \r) \d u
            \\&\lesssim
            %\l\|\xi-\eta\r\|_{\L^2}^2 + \int_t^s \alpha_0\!\l(u\r)^2  \E\l[ \l\| X^{t,\xi}_u - X^{t,\eta}_u \r\|^2 \r] \d u
            % supremum
            %\\&\leq 
            \l\|\xi-\eta\r\|_{\L^2}^2 + \int_t^s \alpha_0\!\l(u\r)^2  \E\l[ \sup_{t\leq w\leq u}\l\| X^{t,\xi}_w - X^{t,\eta}_w \r\|^2 \r] \d u
            .
        \end{flalign*}
        Finally, Grönwall's inequality yields the desired result.
    \end{proof}

    \begin{lem}\label{lem:xxi-yeta-p-bound} 
        Let $1\leq p\leq q_0$. If Assumption~\ref{ass:1-lipschitz} is satisfied, then
        \begin{equation*}
            \E\l[ \sup_{t\leq s\leq T}\l\| X^{t,x,\xi}_s - X^{t,y,\eta}_s \r\|^p \,\Big|\,\CF_t \r] \lesssim \l\| x- y\r\|^p + \l\|\xi-\eta \r\|_{\L^2}^p
        \end{equation*}
        for all $0\leq t\leq T$, $\xi,\eta\in \L^{2,d}(t)$ and $x,y\in\Rbb^d$.
    \end{lem}
    \begin{proof}
        By Lemma~\ref{lem:conditional}, we have for all $t\leq s\leq T$
        \begin{flalign*}
            \quad&
            \E\l[ \sup_{t\leq w\leq s}\l\| X^{t,x,\xi}_w - X^{t,y,\eta}_w \r\|^p \,\Big|\,\CF_t\r]
            &&
            % lem: conditional
            \\&\lesssim
            \l\|x-y\r\|^p + \sum_{f\in F} \int_t^s \E\l[ \l| f\!\l(u,X^{t,x,\xi}_u,X^{t,\xi}_u\r) - f\!\l(u,X^{t,y,\eta}_u,X^{t,\eta}_u\r) \r|^p \,\Big|\,\CF_t\r] \d u
            % ass: 1
            \\&\leq
            \l\|x-y\r\|^p + \int_t^s \alpha_0\!\l(u\r)^p \l( \E\l[ \l\| X^{t,x,\xi}_u - X^{t,y,\eta}_u \r\|^p \,\Big|\,\CF_t\r] + \l\|X^{t,\xi}_u - X^{t,\eta}_u\r\|_{\L^2}^p \r) \d u 
            % L2-H2 norm inequality
            %\\&\lesssim
            %\l\|x-y\r\|^p + \l\|X^{t,\xi} - X^{t,\eta}\r\|_{\H^2}^p + \int_t^s \alpha_0\!\l(u\r)^p \E\l[ \l\| X^{t,x,\xi}_u - X^{t,y,\eta}_u \r\|^p \,\Big|\,\CF_t\r] \d u 
            % lem:xi-eta-2-bound
            \\&\lesssim
            \l\|x-y\r\|^p + \l\| \xi - \eta\r\|_{\L^2}^p + \int_t^s \alpha_0\!\l(u\r)^p \E\l[  \l\| X^{t,x,\xi}_u - X^{t,y,\eta}_u \r\|^2 \,\Big|\,\CF_t\r] \d u 
            ,
        \end{flalign*}
        where the last step follows from Lemma~\ref{lem:xi-eta-2-bound}.
        Finally, Grönwall's inequality yields the desired result.
    \end{proof}

    For $\eta\in\L^{1,d}$, we can define the concatenation
    \begin{equation*}
        X^{t,\eta,\xi}:\, [0,T]\times \Omega \rightarrow \Rbb^d, \qquad
        (s,\omega)\mapsto X^{t,\eta,\xi}_s(\omega):=X^{t,x,\xi}_s(\omega)\Big|_{x=\eta(\omega)}.
    \end{equation*}
    
    \begin{lem}\label{lem:concatenation-H2}
        If Assumption~\ref{ass:1-lipschitz} is satisfied, then $X^{t,\eta,\xi}\in \H^{2,d}(t,T)$ for all $0\leq t\leq T$ and $\xi,\eta\in\L^{2,d}(t)$.
    \end{lem}
    \begin{proof}
        Lemma~\ref{lem:xxi-yeta-p-bound} implies $\l(X^{t,x,\xi}-X^{t,y,\xi}\r)\in \H^{2,d}(0,T)\subseteq \M^{2,d}(0,T)$ and, thus, we immediately get $X^{t,\eta,\xi}\in \M^{2,d}(t,T)$ due to Lemma~A.4 in \cite{bollweg_mean-field_2025}.

        Moreover, Lemma~\ref{lem:H2-bound} yields
        \begin{align*}
            \E\l[ \sup_{t\leq w\leq T} \l\|X^{t,\eta,\xi}_w\r\|^2 \r] 
            &= 
            \E\l[ \E\l[ \sup_{t\leq w\leq T}\l\| X^{t,x,\xi}_w \r\|^2 \,\Big|\,\CF_t \r]\bigg|_{x=\eta} \r]
            % lem:H2-bound
            \lesssim
            \l\|K\r\|_{\M^1} + \l\|\eta\r\|_{\L^2}^2 + \l\|\xi\r\|_{\L^2}^2
            < \infty
            .
        \end{align*}
    \end{proof}
    
    \begin{lem}\label{lem:concatenation-identity} 
        If Assumption~\ref{ass:1-lipschitz} is satisfied, then
        \begin{equation*}
            \l\| X^{t,\xi,\xi} - X^{t,\xi} \r\|_{\H^2}=0
        \end{equation*}
        for all $0\leq t\leq T$ and $\xi\in \L^{2,d}(t)$.
    \end{lem}
    \begin{proof}
        By Lemma~\ref{lem:conditional}, we have for all $t\leq s\leq T$
        \begin{flalign*}
            \quad&
            \E\l[ \sup_{t\leq w\leq s}\l\| X^{t,x,\xi}_w - X^{t,\xi}_w \r\|^2 \,\Big|\,\CF_t \r] 
            &&\\
            % lem: conditional, xi F_t measurable
            & \lesssim \l\|x-\xi\r\|^2 + \sum_{f\in F} \int_t^s \E\l[ \l| f\!\l(u,X^{t,x,\xi}_u,X^{t,\xi}_u\r) - f\!\l(u,X^{t,\xi}_u,X^{t,\xi}_u\r) \r|^2\,\Big|\,\CF_t\r] \d u
            % assumption-1
            %\\&\leq 
            %\l\|x-\xi\r\|^2 + \int_t^s \alpha_0\!\l(u\r)^2 \E\l[ \l\| X^{t,x,\xi}_u - X^{t,\xi}_u \r\|^2 \,\Big|\,\CF_t\r] \d u
            % supremum
            \\&\leq 
            \l\|x-\xi\r\|^2 + \int_t^s \alpha_0\!\l(u\r)^2 \E\l[ \l\| X^{t,x,\xi}_u - X^{t,\xi}_w \r\|^2 \,\Big|\, \CF_t\r] \d u
        \end{flalign*}
        and Grönwall's inequality yields
        \begin{equation*}
            \E\l[ \sup_{t\leq w\leq T}\l\| X^{t,x,\xi}_w - X^{t,\xi}_w \r\|^2 \,\Big|\,\CF_t \r] \lesssim \l\|x-\xi\r\|^2.
        \end{equation*}
        Finally, the aggregation property implies
        \begin{align*}
            \l\|X^{t,\xi,\xi} - X^{t,\xi}\r\|_{\H^2} 
            &= 
            \E\l[ \E\l[ \sup_{t\leq w\leq T} \l\|X^{t,x,\xi}_w - X^{t,\xi}_w \r\|^2\,\Big|\,\CF_t\r] \bigg|_{x=\xi} \r]
            % lem:dif-Xxxi-Xxi
            %\lesssim
            %\E\l[ \l\|x-\xi\r\|^2 \Big|_{x=\xi} \r]
            =0.
        \end{align*}
    \end{proof}

    \section{First Order Derivatives}\label{sec:derivative-1}
    In this section, we show that the solution map $(x,\xi)\mapsto X^{t,x,\xi}$ is Fréchet differentiable for Fréchet differentiable coefficients with Lipschitz and bounded Fréchet derivatives.
    Before we turn to the differentiability results, let us agree on some definitions and recall the fundamental theorem of calculus.

    \begin{dfn}
        Let $V$ and $W$ be normed real vector spaces with norms $\l\|\cdot\r\|_V$ and $\l\|\cdot\r\|_W$ respectively. A map $f:\,V \rightarrow W$ is called Fréchet differentiable if, for every $v_0\in V$, there exists a continuous linear operator $\D{} f(v_0):\, V \rightarrow W $ such that
        \begin{equation*}
            \lim_{\l\|v\r\|_V \rightarrow 0} \frac{\l\| f\!\l(v_0 +v\r) - f\!\l(v_0\r) - \D{}f\!\l(v_0\r)v \r\|_W}{\l\|v\r\|_V} = 0,
        \end{equation*}
        and the map
        \begin{equation*}
            \D{}f:\, V \rightarrow B(V,W), \qquad v\mapsto \D{}f(v)
        \end{equation*}
        is called the Fréchet derivative of $f$, where $B(V,W)$ denotes the space of all bounded linear operators $L:\,V\rightarrow W$.

        A Fréchet differentiable map $f:\,V \rightarrow W$ is called continuously Fréchet differentiable if the Fréchet derivative $v\mapsto \D {} f(v)$ is continuous with respect to the operator norm. Let $\C^{1}(V)$ denote the space of all continuously Fréchet differentiable maps $f:\, V\rightarrow \Rbb$.
    \end{dfn}

    In Section~\ref{sec:derivative-2}, we repeatedly use the following version of the fundamental theorem of calculus.
    \begin{lem}
        Let $V$ and $W$ be normed real vector spaces.
        If $f:\,V \rightarrow W$ is continuously Fréchet differentiable, then
        \begin{equation*}
            f\!\l(v_0 + v\r) - f\!\l(v\r) = \int_0^1 \D{}f\!\l(v_0+\lambda v\r) v \d \lambda 
        \end{equation*}
        for all $v,v_0\in V$.
    \end{lem}
    
    \begin{ass}\label{ass:2-first-derivative}
        Let $b:\, [0,T]\times \Omega \times\Rbb^d\times\L^{2,d}\rightarrow\Rbb^d$, $h:\, [0,T]\times\Omega\times \Rbb^d\times\L^{2,d}\rightarrow\Rbb^{d\times n\times n}$, and $g:\, [0,T]\times \Omega\times\Rbb^d\times\L^{2,d}\rightarrow\Rbb^{d\times n}$ be such that the following holds for all components $f=b_k,h_{kij},g_{ki}$ with $1\leq i,j\leq n$, $1\leq k\leq d$.

        \begin{enumerate}
            \item 
            We have $f(s,\omega,x,\cdot)\in\C^{1}(\L^{2,d})$ and $f(s,\omega,\cdot,\xi)\in\C^{1}(\Rbb^d)$ for all $0\leq s\leq T$, $\omega\in\Omega$, $x\in\Rbb^d$ and $\xi\in\L^{2,d}$. 

            \item
            There exists a $q_1$-integrable $\alpha_1:\,[0,T]\rightarrow [1,\infty)$ with $q_1\geq 2$ such that
            \begin{align*}
                \l| \D x f\!\l(s,\omega,x,\xi\r)z - \D x f\!\l(s,\omega,y,\eta\r)z \r|
                & \leq \alpha_1\!\l(s\r) \l\|z\r\| \l( \l\|x-y\r\|  + \l\|\xi-\eta\r\|_{\L^2} \r)
                , %\numberthis \label{ineq:Dxf-Lipschitz} 
               \\
                \l| \D \xi f\!\l(s,\omega,x,\xi\r)\zeta - \D \xi f\!\l(s,\omega,y,\eta\r)\zeta \r|
                & \leq \alpha_1\!\l(s\r) \l\|\zeta\r\|_{\L^2} \l( \l\|x-y\r\|  + \l\|\xi-\eta\r\|_{\L^2} \r),
                %\numberthis \label{Dxif-Lipschitz}
                \\
                \l| \D \xi f\!\l(s,\omega,x,\xi \r) \eta \r| &\leq \alpha_1\!\l(s\r)\l\|\eta\r\|_{\L^1}
                % \numberthis \label{Dxif-1-bound}
            \end{align*}
            for all $x,y,z\in\Rbb^d$, $\xi,\eta,\zeta\in\L^{2,d}$, $0\leq s\leq T$ and $\omega\in\Omega$, where $\D x f\!\l(s,\omega,x,\xi\r)$ and $\D \xi f\!\l(s,\omega,x,\xi \r)$ denote the Fréchet derivatives of $f$ with respect to $x$ and $\xi$ respectively.
        \end{enumerate}
    \end{ass}

    \begin{rmk}
        Note that Assumption~\ref{ass:1-lipschitz} yields bounds for $\D x f$ and $\D\xi f$ which are uniform in $(\omega,x,\xi)$ and $q_0$-integrable in $s$.
        To be specific, we have the following bounds for all components $f=b_k,h_{kij},g_{ki}$, $1\leq i,j\leq n$, $1\leq k\leq d$,
        \begin{align*}
            \l| \D x f\!\l(s,\omega,x,\xi\r) y \r| &\leq \alpha_0(s) \l\|y\r\|
            &
            \l| \D \xi f\!\l(s,\omega,x,\xi\r) \eta \r| &\leq \alpha_0(s) \l\|\eta\r\|_{\L^2}
            .
            \numberthis\label{ineq:1-derivative-bound}
        \end{align*} 
        for all $x,y\in \Rbb^d$, $\xi,\eta\in \L^{2,d}$, $0\leq s\leq T$ and $\omega\in \Omega$.
    \end{rmk}

    Moreover, Assumption~\ref{ass:2-first-derivative} implies that the Fréchet derivatives of the coefficients are in $\M^2(0,T)$. More precisely, we have the following results.

    \begin{lem}\label{lem:coefficients-differentiable-1}
        If Assumption~\ref{ass:1-lipschitz} and \ref{ass:2-first-derivative} are satisfied, then the the following holds for for all components $f=b_k,h_{kij},g_{ki}$ with $1\leq i,j\leq n$, $1\leq k\leq d$. 
        The map
        \begin{equation*}
            \H^{2,d}(0,T)\times \H^{2,d}(0,T)\rightarrow \M^1(0,T), 
            \qquad (X,Y)\mapsto f\!\l(\cdot,X,Y\r)
        \end{equation*}
        is Fréchet differentiable in each argument with Fréchet derivatives $\D x f\!\l(\cdot,X,Y\r)$ and $\D \xi f\!\l(\cdot,X,Y\r)$ at $(X,Y)$ respectively.
    \end{lem}
    \begin{proof}
        Assumption~\ref{ass:1-lipschitz} implies that $f(\cdot,X,Y)\in\M^1(0,T)$ for all $X,Y\in\H^{2,d}(0,T)$, cf. Corollary~3.4 in \cite{bollweg_mean-field_2025}. 
        Thus, the map $(X,Y)\mapsto f(\cdot,X,Y)$ is well-defined.
        
        Let $X,Y,Z\in\H^{2,d}(0,T)$.
        Since $f(s,\omega,\cdot,\xi)\in\C^{1}(\Rbb^d)$ for all $0\leq s\leq T$, $\omega\in\Omega$ and $\xi\in\L^{2,d}$, we have
        \begin{flalign*}
            \quad
            &\l\| f\!\l(\cdot,X+Z,Y\r) - f\!\l(\cdot,X,Y\r) - \D x f\!\l(\cdot,X,Y\r) Z \r\|_{\M^1}
            &&\\
            &=
            \int_0^T \E\l[ \l| f\!\l(s,X_s+Z_s,Y_s\r) - f\!\l(s,X_s,Y_s\r) - \D x f\!\l(s,X_s,Y_s\r) Z_s \r| \r] \d s
            % continuously differentiable
            %\\&=
            %\int_0^T \E\l[ \l| \int_0^1 \D x f\!\l(s,X_s+\lambda Z_s,Y_s\r) Z_s \d \lambda - \D x f\!\l(s,X_s,Y_s\r) Z_s \r|^2 \r] \d s
            % Jensen's inequality
            \\&\leq 
            \int_0^T \E\l[ \int_0^1 \l|  \D x f\!\l(s,X_s+\lambda Z_s,Y_s\r) Z_s - \D x f\!\l(s,X_s,Y_s\r) Z_s \r| \d \lambda \r] \d s
            % ass:2
            %\\&\leq
            %\int_0^T \E\l[ \int_0^1 \alpha_1\!\l(s\r) \lambda \l\|Z_s\r\|^2 \d \lambda \r] \d s
            % lambda < 1 
            \\&\leq 
            \int_0^T \alpha_1\!\l(s\r) \E\l[\l\|Z_s\r\|^2 \r] \ ds
            % L2-H2 inequality
            \\&\leq 
            \l\|Z\r\|_{\H^2}^2 \int_0^T \alpha_1\!\l(s\r) \d s.
        \end{flalign*}
        Analogously, since $f(s,\omega,x,\cdot)\in\C^{1}(\L^{2,d})$ for all $0\leq s\leq T$, $\omega\in\Omega$ and $x\in\Rbb^d$, we have
        \begin{flalign*}
            \quad 
            &\l\| f\!\l(\cdot,X,Y+Z\r) - f\!\l(\cdot,X,Y\r) - \D \xi f\!\l(\cdot,X,Y\r) Z \r\|_{\M^1}
            &&\\
            &=
            \int_0^T \E\l[ \l| f\!\l(s,X_s,Y_s+Z_s\r) - f\!\l(s,X_s,Y_s\r) - \D \xi f\!\l(s,X_s,Y_s\r) Z_s \r| \r] \d s
            % continuously differentiable
            %\\&=
            %\int_0^T \E\l[ \l| \int_0^1 \D \xi f\!\l(s,X_s,Y_s+\lambda Z_s\r) Z_s \d \lambda - \D \xi f\!\l(s,X_s,Y_s\r) Z_s \r| \r] \d s
            % Jensen's inequality
            \\&\leq 
            \int_0^T \E\l[ \int_0^1 \l| \D \xi f\!\l(s,X_s,Y_s+\lambda Z_s\r) Z_s - \D \xi f\!\l(s,X_s,Y_s\r) Z_s \r| \d \lambda \r] \d s
            % ass:2
            %\\&\leq 
            %\int_0^T \int_0^1 \alpha_1\!\l(s\r) \lambda \l\| Z_s\r\|_{\L^2}^2 \d \lambda \d s
            % lambda < 1
            \\&\leq 
            \int_0^T \alpha_1\!\l(s\r) \l\| Z_s\r\|_{\L^2}^2 \d s
            % L2-H2 ineq
            \\&\leq 
            \l\|Z\r\|_{\H^2}^2 \int_0^T \alpha_1\!\l(s\r) \d s
            .
        \end{flalign*}
        The integrability of $\alpha_1$ implies
        \begin{align*}
            \lim_{\l\|Z\r\|_{\H^2} \rightarrow 0} \frac{\l\| f\!\l(\cdot,X+Z,Y\r) - f\!\l(\cdot,X,Y\r) - \D x f\!\l(\cdot,X,Y\r) Z \r\|_{\M^1}}{\l\|Z\r\|_{\H^2}} &= 0, \\
            \lim_{\l\|Z\r\|_{\H^2} \rightarrow 0} \frac{\l\| f\!\l(\cdot,X,Y+Z\r) - f\!\l(\cdot,X,Y\r) - \D \xi f\!\l(\cdot,X,Y\r) Z \r\|_{\M^1}}{\l\|Z\r\|_{\H^2}} &= 0.
        \end{align*}
        That is, the map $(X,Y)\mapsto f(\cdot,X,Y)$ is Fréchet differentiable in each argument.
    \end{proof}

    \begin{lem}\label{lem:coefficients-1}
        If Assumption~\ref{ass:1-lipschitz} and \ref{ass:2-first-derivative} are satisfied, then $\D xf\!\l(\cdot,X,Y\r) Z, \D \xi f\!\l(\cdot,X,Y\r) Z\in \M^2(0,T)$
        for all components $f=b_k,h_{kij},g_{ki}$, $1\leq i,j\leq n$, $1\leq k\leq d$ and $X,Y,Z\in\H^{2,d}(0,T)$.
    \end{lem}
    \begin{proof}
        Lemma~\ref{lem:coefficients-differentiable-1} implies $\D xf\!\l(\cdot,X,Y\r) Z, \D \xi f\!\l(\cdot,X,Y\r) Z\in \M^1(0,T)$ for all $X,Y,Z\in\H^{2,d}(0,T)$.
        Moreover, the bound in \eqref{ineq:1-derivative-bound} yields
        \begin{align*}
            \int_0^T \E\l[ \l| \D x f\!\l(s,X_s,Y_s\r) Z_s \r|^2 \r] \d s 
            &\leq 
            \int_0^T \alpha_0\!\l(s\r)^2 \E\l[ \l\| Z_s \r\|^2\r ] \d s
            %\\&
            \lesssim 
            \l\|Z\r\|_{\H^2}^2
            < \infty,
        \end{align*}
        and
        \begin{align*}
            \int_0^T \E\l[ \l| \D \xi f\!\l(s,X_s,Y_s\r) Z_s \r|^2 \r] \d s 
            &\leq 
            \int_0^T \alpha_0\!\l(s\r)^2 \l\|Z_s\r\|_{\L^2}^2 \d s
            %\\&
            \lesssim 
            \l\|Z\r\|_{\H^2}^2
            < \infty
        \end{align*}
        since $\alpha_0$ is square-integrable and $Z\in\H^{2,d}(0,T)$.
        Hence, $\D xf\!\l(\cdot,X,Y\r) Z, \D \xi f\!\l(\cdot,X,Y\r) Z\in \M^2(0,T)$ for all $X,Y,Z\in\H^{2,d}(0,T)$.
    \end{proof}
    
    \begin{lem}\label{lem:A-linear}
        If Assumptions~\ref{ass:1-lipschitz} and \ref{ass:2-first-derivative} are satisfied, then the $G$-SDE
        \begin{align*}
            \d A^{t,x,\xi,y}_s 
            &= \D x b\!\l(s,X^{t,x,\xi}_s,X^{t,\xi}_s\r) A^{t,x,\xi,y}_s \d s  + \D x 
            h\!\l(s,X^{t,x,\xi}_s,X^{t,\xi}_s\r) A^{t,x,\xi,y}_s \d \l<B\r>_s  
            \\&\quad 
            + \D x g\!\l(s,X^{t,x,\xi}_s,X^{t,\xi}_s\r) A^{t,x,\xi,y}_s \d B_s, \qquad t\leq s\leq T, \\
            A^{t,x,\xi,y}_t &= y.
            \numberthis\label{eq:SDE-x-derivative}
        \end{align*}        
        admits a unique solution $A^{t,x,\xi,y}\in \H^{2,d}(t,T)$
        for all $0\leq t\leq T$, $x,y\in\Rbb^d$ and $\xi\in\L^{2,d}(t)$.
        Moreover, the map
        \begin{equation*}
            \Rbb^d \rightarrow \H^{2,d}(t,T),\qquad y\mapsto A^{t,x,\xi,y}
        \end{equation*}
        is linear.
    \end{lem}
    \begin{proof}
        By Lemma~\ref{lem:coefficients-1}, the coefficients in \eqref{eq:SDE-x-derivative} are in $\M^2(0,T)$. Moreover, they are Lipschitz continuous and, thus \eqref{eq:SDE-x-derivative}
        admits a unique solution $A^{t,x,\xi,y}\in \H^{2,d}(t,T)$ for all $0\leq t\leq T$, $x,y\in \Rbb^d$ and $\xi\in\L^{2,d}(t)$.
        In particular, we deduce that the map $y\mapsto A^{t,x,\xi,y}$ is well-defined.

        Let $\lambda\in\Rbb$.
        By Lemma~\ref{lem:conditional}, we have for all $t\leq s\leq T$
        \begin{flalign*}
            \quad 
            \E&\l[ \sup_{t\leq w\leq s} \l\| A^{t,x,\xi,y+\lambda z}_w - A^{t,x,\xi,y}_w - \lambda A^{t,x,\xi,z}_w \r\|^2 \r]
            &&
            % lem:conditional
            \\&\lesssim
            \sum_{f\in F} \int_t^s \E\l[ \l| \D x f\!\l(u,X^{t,x,\xi}_u,X^{t,\xi}_u\r) \l( A^{t,x,\xi,y+\lambda z}_u - A^{t,x,\xi,y}_u - \lambda A^{t,x,\xi,z}_u \r) \r|^2 \r] \d u
            % Dx f bound
            \\&\leq
            \int_t^s \alpha_0\!\l(u\r)^2 \E\l[ \sup_{t\leq w\leq u}\l\| A^{t,x,\xi,y+\lambda z}_w - A^{t,x,\xi,y}_w - \lambda A^{t,x,\xi,z}_w \r\|^2 \r] \d u.
        \end{flalign*}
        Finally, Grönwall's inequality yields 
        \begin{equation*}
            \l\| A^{t,x,\xi,y+\lambda z}_u - A^{t,x,\xi,y}_u - \lambda A^{t,x,\xi,z}_u \r\|_{\H^2}=0.
        \end{equation*}
        Since $\lambda\in\Rbb$ and $y,z\in\Rbb^d$ were arbitrary, we deduce that $y\mapsto A^{t,x,\xi,y}$ is linear.
    \end{proof}

    \begin{lem}\label{lem:Dx-p-bound}
        Let $2\leq p\leq q_0$. 
        If Assumptions~\ref{ass:1-lipschitz} and \ref{ass:2-first-derivative} are satisfied, then
        \begin{equation*}
            \E\l[ \sup_{t\leq s\leq T} \l\| A^{t,x,\xi,y}_s \r\|^p \,\Big|\, \CF_t \r] \lesssim \l\|y\r\|^p
        \end{equation*}
        for all $0\leq t\leq T$, $x,y\in\Rbb^d$ and $\xi\in\L^{2,d}(t)$.
    \end{lem}
    \begin{proof}
        By Lemma~\ref{lem:conditional}, we have for all $t\leq s\leq T$ that
        \begin{flalign*}
            \quad
            \E&\l[ \sup_{t\leq w\leq s} \l\| A^{t,x,\xi,y}_w y\r\|^p \,\Big|\,\CF_t \r]
            &&
            % lem:conditional
            \\&\lesssim 
            \l\|y\r\|^p +  \sum_{f\in F} \int_t^s \E\l[ \l|  \D x f\!\l(u,X^{t,x,\xi}_u,X^{t,\xi}_u\r) A^{t,x,\xi,y}_u \r|^p \,\Big|\,\CF_t \r] \d u
            % bounded Dx f
            \\&\lesssim
            %\l\|y\r\|^p + \int_t^s \alpha_0\!\l(u\r)^p \E\l[ \l\| A^{t,x,\xi,y}_u\r\|^p \,\Big|\,\CF_t \r] \d u
            % supremum
            %\\&\leq 
            \l\|y\r\|^p + \int_t^s \alpha_0\!\l(u\r)^p \E\l[ \sup_{t\leq w\leq u}\l\| A^{t,x,\xi,y}_w\r\|^p \,\Big|\,\CF_t \r] \d u
            .
        \end{flalign*}
        Grönwall's inequality yields the desired result.
    \end{proof}
    
    \begin{prp}\label{prp:x-differentiability}
        Let $0\leq t\leq T$ and $\xi \in\L^{2,d}(t)$. If Assumptions~\ref{ass:1-lipschitz} and \ref{ass:2-first-derivative} are satisfied, then
        the map
        \begin{equation*}
            \Rbb^d\rightarrow\H^{2,d}(t,T), \qquad x \mapsto X^{t,x,\xi}
        \end{equation*}
        is Fréchet differentiable with Fréchet derivative
        \begin{equation*}
            \D x X^{t,x,\xi}:\,
            \Rbb^d \rightarrow \H^{2,d}(t,T)
            ,\qquad
            y\mapsto \D x X^{t,x,\xi}y:=A^{t,x,\xi,y}
        \end{equation*}
        at $x\in \Rbb^d$.
    \end{prp}
    \begin{proof}
        By Lemma~\ref{lem:A-linear}, the map $\D x X^{t,x,\xi}:\,y \mapsto A^{t,x,\xi,y}$ is linear.        
        Set $Y:=X^{t,x+y,\xi}-X^{t,x,\xi}$, then
        \begin{equation}
            \E\l[ \sup_{t\leq s\leq T} \l\|Y_s\r\|^4 \r]
            =
            \E\l[ \sup_{t\leq s\leq T} \l\|X^{t,x+y,\xi}_s-X^{t,x,\xi}_s\r\|^4 \r]  
            \lesssim 
            \l\|y\r\|^4 \label{ineq:Y-bound}
        \end{equation}
        due to Lemma~\ref{lem:xxi-yeta-p-bound}.
        By Lemma~\ref{lem:conditional}, we have for all $t\leq s\leq T$ that
        \begin{flalign*}
            \quad
            \E&\l[ \sup_{t\leq w\leq s} \l\| X^{t,x+y,\xi}_w - X^{t,x,\xi}_w - A^y_w\r\|^2\r]
            &&
            % lem:conditional
            \\&\lesssim
            \sum_{f\in F} \int_t^s \E\l[ \l| f\!\l(u,X^{t,x+y,\xi}_u,X^{t,\xi}_u\r) - f\!\l(u,X^{t,x,\xi}_u,X^{t,\xi}_u\r) - \D x f\!\l(u,X^{t,x,\xi}_u,X^{t,\xi}_u\r) A^{t,x,\xi,y}_u \r|^2 \r] \d u
            % continuously differentiable
            \\& =
            \sum_{f\in F} \int_t^s \E\l[ \l| \int_0^1 \D x f\!\l(u,X^{t,x,\xi}_u+\lambda Y_u,X^{t,\xi}_u\r) Y_u \d \lambda - \D x f\!\l(u,X^{t,x,\xi}_u,X^{t,\xi}_u\r) A^{t,x,\xi,y}_u \r|^2 \r] \d u
            % Jensen's inequality
            \\&\lesssim
            \sum_{f\in F} \int_t^s \int_0^1 \E\l[ \l| \D x f\!\l(u,X^{t,x,\xi}_u+\lambda Y_u,X^{t,\xi}_u\r) Y_u - \D x f\!\l(u,X^{t,x,\xi}_u,X^{t,\xi}_u\r) Y_u \r|^2 \r] \d \lambda \d u 
            \\&\quad +
            \sum_{f\in F} \int_t^s \E\l[ \l| \D x f\!\l(u,X^{t,x,\xi}_u,X^{t,\xi}_u\r) \!\l( Y_u - A^{t,x,\xi,y}_u \r) \r|^2 \r] \d u
            % ass:2, Dx f bound
            %\\&\leq
            %\int_t^s \int_0^1 \alpha_1\!\l(u\r)^2 \lambda^2 \E\l[ \l\| Y_u \r\|^4 \r] \d \lambda \du + \alpha_0\!\l(u\r)^2 \E\l[ \l\| Y_u - A^{t,x,\xi,y}_u \r\|^2 \r] \d u
            % lambda < 1, supremum
            \\&\leq 
            \int_t^s \alpha_1\!\l(u\r)^2 \E\l[ \l\| Y_u \r\|^4 \r] + \alpha_0\!\l(u\r)^2 
            \E\l[ \l\| Y_u - A^{t,x,\xi,y}_u \r\|^2 \r] \d u
            \\&\lesssim 
            \l\|y\r\|^4
            + \int_t^s \alpha_0\!\l(u\r)^2 \E\l[ \sup_{t\leq w\leq u}\l\| X^{t,x+y,\xi}_w-X^{t,x,\xi}_w - A^{t,x,\xi,y}_w \r\|^2 \r] \d u
            ,
        \end{flalign*}
        where the last step follows from \eqref{ineq:Y-bound}. 
        Finally, Grönwall's inequality yields
        \begin{equation*}
            \l\| X^{t,x+y,\xi} - X^{t,x,\xi} - A^{t,x,\xi,y} \r\|_{\H^2}^2
            \lesssim \l\|y\r\|^4.
        \end{equation*}
        Thus,
        \begin{equation*}
            \lim_{\|y\|\rightarrow 0} \frac{\l\|X^{t,x+y,\xi} - X^{t,x,\xi} - A^{t,x,\xi,y}\r\|_{\H^2}}{\|y\|}
            = 0,
        \end{equation*}
        i.e., $\D x X^{t,x,\xi}:\,y\mapsto A^{t,x,\xi,y}$ is the Fréchet derivative of $x\mapsto X^{t,x,\xi}$ at $x\in\Rbb^d$.
    \end{proof}

    Next, we show that the map $x\mapsto X^{t,x,\xi}$ is continuously Fréchet differentiable.

    \begin{lem}\label{lem:Dx-Dx}
        Let $2\leq p\leq \l(q_1\wedge \frac{q_0}2\r)$. 
        If Assumptions~\ref{ass:1-lipschitz} and \ref{ass:2-first-derivative} are satisfied with $q_0\geq 4$, then
        \begin{equation*}
            \E\l[ \sup_{t\leq s\leq T} \l\| \D x X^{t,x,\xi}_sz  - \D x X^{t,y,\eta}_sz \r\|^p \,\Big| \,\CF_t\r] \lesssim \l\|z\r\|^p \l( \l\| x-y\r\|^p + \l\|\xi-\eta\r\|_{\L^2}^p\r)
        \end{equation*}
        for all $0\leq t\leq T$, $x,y,z\in\Rbb^d$ and $\xi,\eta\in\L^{2,d}(t)$.
    \end{lem}
    \begin{proof}
        By Lemma~\ref{lem:conditional}, we have for all $t\leq s\leq T$ that
        \begin{flalign*}
            \quad
            \E&\l[ \sup_{t\leq w\leq s} \l\| \D x X^{t,x,\xi}_w z - \D x X^{t,y,\eta}_w z \r\|^p \,\Big| \,\CF_t \r]
            &&\\
            % lem:conditional
            &\lesssim
            \sum_{f\in F} \int_t^s \E\l[ \l| \D x f\!\l(u,X^{t,x,\xi}_u,X^{t,\xi}_u\r) \D x X^{t,x,\xi}_u z - \D x f\!\l(u,X^{t,y,\eta}_u,X^{t,\eta}_u\r) \D x X^{t,y,\eta}_u z \r|^p \,\Big| \,\CF_t\r] \d u
            % expand + Jensen's inequality
            \\& \lesssim
            \sum_{f\in F} \int_t^s \E\l[ \l| \D x f\!\l(u,X^{t,x,\xi}_u,X^{t,\xi}\r) \D x X^{t,x,\xi}_u z- \D x f\!\l(u,X^{t,y,\eta}_u,X^{t,\eta}\r) \D x X^{t,x,\xi}_u z\r|^p \,\Big| \,\CF_t\r] \d u
            \\&\qquad + 
            \sum_{f\in F} \int_t^s \E\l[ \l| \D x f\!\l(u,X^{t,y,\eta}_u,X^{t,\eta}\r) \l(\D x X^{t,x,\xi}_u z - \D x X^{t,y,\eta}_u z\r)\r|^p \,\Big| \,\CF_t \r] \d u
            % ass:2, Dx f bound
            \\&\leq 
            \int_t^s \alpha_1\!\l(u\r)^p \E\l[ \l\| \D x X^{t,x,\xi}_u z\r\|^p \l\| X^{t,x,\xi}_u - X^{t,y,\eta}_u\r\|^p \,\Big|\,\CF_t\r] \d u 
            %\\&\quad + 
            %\int_t^s \alpha_1\!\l(u\r)^p \l\| X^{t,\xi}_u - X^{t,\eta}_u \r\|_{\L^2}^p \E\l[ \l\| \D x X^{t,x,\xi}_u z\r\|^p \,\Big|\,\CF_t\r] \d u
            %\\&\quad +
            %\int_t^s \alpha_0\!\l(u\r)^p \E\l[ \l\| \D x X^{t,x,\xi}_u z - \D x X^{t,y,\eta}_u z\r\|^p \,\Big| \,\CF_t \r] \d u
            % Cauchy-Schwarz
            \\&\leq
            \int_t^s \alpha_1\!\l(u\r)^p \E\l[ \l\| \D x X^{t,x,\xi}_u z\r\|^{2p} \,\Big|\,\CF_t\r]^{\frac12} \E\l[ \l\| X^{t,x,\xi}_u - X^{t,y,\eta}_u\r\|^{2p} \,\Big|\,\CF_t\r]^{\frac12} \d u 
            \\&\quad + 
            \int_t^s \alpha_1\!\l(u\r)^p \l\| X^{t,\xi}_u - X^{t,\eta}_u \r\|_{\L^2}^p \E\l[ \l\| \D x X^{t,x,\xi}_u z\r\|^p \,\Big|\,\CF_t\r] \d u
            \\&\quad +
            \int_t^s \alpha_0\!\l(u\r)^p \E\l[ \l\| \D x X^{t,x,\xi}_u z - \D x X^{t,y,\eta}_u z\r\|^p \,\Big| \,\CF_t \r] \d u
            % lem: xxi-yeta-p-bound, Dx-p-bound, xi-eta-2-bound
            %\\&\lesssim
            %\int_t^s \alpha_1\!\l(u\r)^p \l\|z\r\|^p \l( \l\|x-y\r\|^p + \l\|\xi-\eta\r\|_{\L^2}^p \r) \d u 
            %\\&\quad + 
            %\int_t^s \alpha_1\!\l(u\r)^p \l\|z\r\|^p \l\|\xi-\eta\r\|_{\L^2}^p  \d u
            %\\&\quad +
            %\int_t^s \alpha_0\!\l(u\r)^p \E\l[ \l\| \D x X^{t,x,\xi}_u z - \D x X^{t,y,\eta}_u z\r\|^p \,\Big| \,\CF_t \r] \d u
            \\&\lesssim
            \l\|z\r\|^p \l( \l\|x-y\r\|^p + \l\|\xi-\eta\r\|_{\L^2}^p \r) + \int_t^s \alpha_0\!\l(u\r)^p \E\l[ \l\| \D x X^{t,x,\xi}_u z - \D x X^{t,y,\eta}_u z\r\|^p \,\Big| \,\CF_t \r] \d u
            ,
        \end{flalign*}
        where the last step follows from Lemmas~\ref{lem:xi-eta-2-bound}, \ref{lem:xxi-yeta-p-bound} and \ref{lem:Dx-p-bound}.
        Finally, Grönwall's inequality yields the desired result.
    \end{proof}
    
    \begin{cor}\label{cor:Dx-continuous}
        Let $0\leq t\leq T$, $\xi\in\L^{2,d}(t)$. If Assumptions~\ref{ass:1-lipschitz} and \ref{ass:2-first-derivative} are satisfied with $q_0\geq 4$, then the map
        \begin{equation*}
            \Rbb^d \rightarrow \H^{2,d}(t,T),\qquad x\mapsto X^{t,x,\xi}
        \end{equation*}
        is continuously Fréchet differentiable.
    \end{cor}
    \begin{proof}
        Lemma~\ref{lem:Dx-Dx} implies that
        \begin{equation*}
            \sup_{0\neq z\in \Rbb^d}\frac{\l\| \D x X^{t,x,\xi}z - \D x X^{t,y,\xi}z\r\|_{\H^2}}{\l\|z\r\|}
            \lesssim \l\| x-y\r\|,
        \end{equation*}
        i.e., $x\mapsto \D x X^{t,x,\xi}$ is continuous with respect to the operator norm.
    \end{proof}

    \begin{lem}\label{lem:Dx-concatenation}
        Let $0\leq t\leq T$ and $\xi,\eta,\zeta\in\L^2(t)$.
        If Assumptions~\ref{ass:1-lipschitz} and \ref{ass:2-first-derivative} are satisfied with $q_0\geq 4$, then $\D x X^{t,\eta,\xi}\zeta\in\H^{2,d}(t,T)$ with
        \begin{equation*}
            \E\l[ \sup_{t\leq w\leq T}\l\| \D x X^{t,\eta,\xi}_w\zeta \r\|^2 \r] \lesssim \l\|\zeta\r\|_{\L^2}^2,
        \end{equation*}
        where $\D x X^{t,\eta,\xi}\zeta$ denotes the map
        \begin{equation*}
            [0,T]\times \Omega \rightarrow \Rbb^{d},
            \qquad
            (s,\omega)\mapsto \D x X^{t,\eta,\xi}_s\zeta(\omega):= A^{t,x,\xi,y}_s(\omega)\Big|_{x=\eta(\omega), y=\zeta(\omega)}.
        \end{equation*}
    \end{lem}
    \begin{proof}
        We have $X^{t,\eta,\xi}\in\H^{2,d}(t,T)$ due to Corollary~\ref{lem:concatenation-H2}.
        Moreover, the SDE
        \begin{align*}
            \d Y_s
            &= \D x b\!\l(s,X^{t,\eta,\xi}_s,X^{t,\xi}_s\r)Y_s \d s  + \D x 
            h\!\l(s,X^{t,\eta,\xi}_s,X^{t,\xi}_s\r)Y_s \d \l<B\r>_s  
            \\&\quad 
            + \D x g\!\l(s,X^{t,\eta,\xi}_s,X^{t,\xi}_s\r)Y_s \d B_s, \qquad t\leq s\leq T, \\
            Y_t &= \zeta.
        \end{align*}
        has a unique solution $Y\in\H^{2,d}(t,T)$ since the coefficients are Lipschitz continuous and in $\M^2(0,T)$.

        By Lemma~\ref{lem:conditional}, we have for all $t\leq s\leq T$
        \begin{flalign*}
            \quad
            \E&\l[ \sup_{t\leq w\leq s}\l\|A^{t,x,\xi,y}_w - Y_w\r\|^2\,\Big|\,\CF_t\r]
            &&
            % lem:conditional
            \\&\lesssim
            \l\| y-\zeta\r\|^2 +
            \sum_{f\in F} \int_t^s \E\l[ \l| \D x f\!\l(u,X^{t,x,\xi}_u,X^{t,\xi}_u\r)A^{t,x,\xi,y}_u - \D x f\!\l(u,X^{t,\eta,\xi}_u,X^{t,\xi}_u\r)Y_u \r|^2\,\Big|\,\CF_t \r]\d u
            \\&\lesssim
            %\l\| y-\zeta\r\|^2 +
            %\sum_{f\in F} \int_t^s \E\l[ \l| \D x f\!\l(u,X^{t,x,\xi}_u,X^{t,\xi}_u\r)A^{t,x,\xi,y}_u - \D x f\!\l(u,X^{t,\eta,\xi}_u,X^{t,\xi}_u\r)A^{t,x,\xi,y}_u \r|^2 \,\Big|\,\CF_t \r]\d u
            %\\&\quad
            %+\sum_{f\in F} \int_t^s \E\l[ \l| \D x f\!\l(u,X^{t,x,\xi}_u,X^{t,\xi}_u\r)\!\l(A^{t,x,\xi,y}_u - Y_u\r) \r|^2 \,\Big|\,\CF_t \r]\d u
            % Dx f Lipschitz, bounded
            %\\&\leq
            \l\| y-\zeta\r\|^2 +
            \int_t^s \alpha_1\!\l(u\r)^2 \E\l[ \l\| X^{t,x,\xi}_u -X^{t,\eta,\xi}_u\r\|^2 \l\| A^{t,x,\xi,y}_u \r\|^2 \,\Big|\,\CF_t \r]\d u
            \\&\quad
            +\int_t^s \alpha_0\!\l(u\r)^2 \E\l[ \l\| A^{t,x,\xi,y}_u - Y_u \r\|^2 \,\Big|\,\CF_t \r]\d u
            % Cauchy-Schwarz
            \\&\leq 
            \l\| y-\zeta\r\|^2 +
            \int_t^s \alpha_1\!\l(u\r)^2 \E\l[ \l\| X^{t,x,\xi}_u -X^{t,z,\xi}_u\r\|^4\,\Big|\,\CF_t\r]^{\frac12}\Big|_{z=\eta} \E\l[\l\| A^{t,x,\xi,y}_u \r\|^4 \,\Big|\,\CF_t \r]^{\frac12}\d u
            \\&\quad
            +\int_t^s \alpha_0\!\l(u\r)^2 \E\l[ \l\| A^{t,x,\xi,y}_u - Y_u \r\|^2 \,\Big|\,\CF_t \r]\d u
            % lem:xxi-yeta-p-bound, Dx-p-bound
            \\&\lesssim 
            \l\| y-\zeta\r\|^2 + \l\|x-\eta\r\|^2 \l\|y\r\|^2 
            +\int_t^s \alpha_0\!\l(u\r)^2 \E\l[ \sup_{t\leq w\leq u}\l\| A^{t,x,\xi,y}_w - Y_w \r\|^2 \,\Big|\,\CF_t \r]\d u
        \end{flalign*}
        due to Lemmas~\ref{lem:xxi-yeta-p-bound} and \ref{lem:Dx-p-bound}.
        Grönwall's inequality implies
        \begin{equation*}
            \E\l[ \sup_{t\leq w\leq T}\l\|A^{t,x,\xi,y}_w - Y_w\r\|^2\,\Big|\,\CF_t\r] \lesssim \l\| y-\zeta\r\|^2 + \l\|x-\eta\r\|^2 \l\|y\r\|^2
        \end{equation*}
        and, thus,
        \begin{equation*}
            \l\| \D x X^{t,\eta,\xi}\zeta - Y \r\|_{\H^2}^2
            % tower property
            =\E\l[ \E\l[ \sup_{t\leq w\leq T}\l\|A^{t,x,\xi,y}_w - Y_w\r\|^2\,\Big|\,\CF_t \r] \bigg|_{x=\eta,\,y=\zeta}\r]
            % Grönwall's inequality
            %\lesssim 
            %\E\l[ \l\|\zeta-\zeta\r\|^2 + \l\|\eta-\eta\r\|^2 \l\|\zeta\r\|^2 \r]
            =
            0.
        \end{equation*}
        That is, $\D x X^{t,\eta,\xi}\zeta = Y\in \H^{2,d}(t,T)$. Finally, we have
        \begin{align*}
            \E\l[ \sup_{t\leq w\leq T}\l\| \D x X^{t,\eta,\xi}\zeta \r\|^2 \r]
            &= \E\l[ \E\l[ \sup_{t\leq w\leq T}\l\| \D x X^{t,x,\xi}z \r\|^2 \,\Big|\,\CF_t\r]\bigg|_{x=\eta,\,z=\zeta} \r]
            \lesssim \E\l[ \l\| \zeta \r\|^2 \r]
        \end{align*}
        due to Lemma~\ref{lem:Dx-p-bound}.
    \end{proof}

    \begin{cor}\label{cor:Dx-Dx-L1}
        If Assumptions~\ref{ass:1-lipschitz} and \ref{ass:2-first-derivative} are satisfied with $q_0\geq 4$, then
        \begin{equation*}
            \E\l[ \sup_{t\leq w\leq T}\l\| \D x X^{t,\eta,\xi}_w \zeta - \D x X^{t,\nu,\chi}_w \zeta\r\| \r] \lesssim \l\|\zeta\r\|_{\L^2}\l( \l\| \eta-\nu\r\|_{\L^2} + \l\| \xi-\chi \r\|_{\L^2}\r)
        \end{equation*}
        for all $0\leq t\leq T$ and $\xi,\eta,\zeta,\nu,\chi\in\L^{2,d}(t)$.
    \end{cor}
    \begin{proof}
        Lemma~\ref{lem:Dx-concatenation} together with the aggregation property yield
        \begin{flalign*}
            \quad &
            \E\l[ \sup_{t\leq w\leq T}\l\| \D x X^{t,\eta,\xi}_w \zeta - \D x X^{t,\nu,\chi}_w \zeta\r\| \r]
            &&
            %\\&=
            %\E\l[ \E\l[ \sup_{t\leq w\leq T}\l\| \D x X^{t,x,\xi}_w z - \D x X^{t,y,\chi}_w z\r\| \,\Big|\, \CF_t \r]\bigg|_{x=\eta,\,y=\nu,\,z=\zeta}\r]
            % Jensen
            \\&\leq \E\l[ \E\l[ \sup_{t\leq w\leq T}\l\| \D x X^{t,x,\xi}_w z - \D x X^{t,y,\chi}_w z\r\|^2 \,\Big|\, \CF_t \r]^{\frac12} \bigg|_{x=\eta,\,y=\nu,\,z=\zeta}\r]
            % lem:Dx-Dx
            \\&\lesssim
            \E\l[ \l\|\zeta\r\| \l( \l\|\eta-\nu\r\| + \l\|\xi-\chi\r\|_{\L^2}\r)\r]
            % Cauchy-Schwarz
            \\&\lesssim
            \l\|\zeta\r\|_{\L^2} \l( \l\| \eta-\nu\r\|_{\L^2} + \l\| \xi-\chi \r\|_{\L^2}\r).
        \end{flalign*}
    \end{proof}
    
    \begin{lem}\label{lem:H1-diff}
        Let $0\leq t\leq T$ and $\xi\in\L^2(t)$. 
        If Assumptions~\ref{ass:1-lipschitz} and \ref{ass:2-first-derivative} are satisfied with $q_0\geq 4$, then
        \begin{equation*}
            \lim_{\l\|\eta\r\|_{\L^2} \rightarrow 0} \frac{\l\| X^{t,\xi+\eta,\xi+\eta} - X^{t,\xi,\xi+\eta} - \D x X^{t,\xi,\xi}\eta \r\|_{\H^1}}{\l\|\eta\r\|_{\L^2}}=0,
        \end{equation*}
        where the limit is taken over $\eta\in \L^{2,d}(t)$.
    \end{lem}
    \begin{proof}
        Due to Corollary~\ref{cor:Dx-continuous}, the map $x\mapsto X^{t,x,\xi+\eta}$ is continuously differentiable.
        In particular, we have
        \begin{equation*}
            X^{t,x+y,\xi+\eta}_s - X^{t,x,\xi+\eta}_s = \int_0^1 \D x X^{t,x+\lambda y,\xi+\eta}_s \eta \d \lambda
        \end{equation*}
        q.s. for all $t\leq s\leq T$.
        Thus, Corollary~\ref{cor:Dx-Dx-L1} yields
        \begin{flalign*}
            \quad
            \E&\l[ \sup_{t\leq s\leq T} \l\| X^{t,\xi+\eta,\xi+\eta}_s - X^{t,\xi,\xi+\eta}_s - \D x X^{t,\xi,\xi}_s\eta \r\| \r]
            &&
            % continuously differentiable
            %\\&=
            %\E\l[ \sup_{t\leq s\leq T} \l\| \int_0^1 \D x X^{t,\xi+\lambda \eta,\xi+\eta}_s\eta - \D x X^{t,\xi,\xi}_s\eta \d \lambda \r\| \r]
            %Jensen's inequality
            %\\&\leq
            %\E\l[ \sup_{t\leq s\leq T} \int_0^1 \l\|  \D x X^{t,\xi+\lambda \eta,\xi+\eta}_s\eta - \D x X^{t,\xi,\xi}_s\eta  \r\| \d \lambda \r]
            % sublinearity
            \\&\leq 
            \int_0^1 \E\l[ \sup_{t\leq s\leq T}  \l\|  \D x X^{t,\xi+\lambda \eta,\xi+\eta}_s\eta - \D x X^{t,\xi,\xi}_s\eta  \r\| \r] \d \lambda
            % Jensen
            %\\&\leq 
            %\int_0^1 \E\l[ \E\l[ \sup_{t\leq s\leq T}  \l\|  \D x X^{t,x+\lambda y,\xi+\eta}_s y - \D x X^{t,x,\xi}_s y  \r\|^2 \,\Big|\,\CF_t\r]^{\frac12} \bigg|_{x=\xi,\,y=\eta}\r] \d \lambda
            % lem:Dx-Dx
            %\\&\lesssim
            %\int_0^1 \E\l[ \l\|\eta\r\|\l( \l\|\eta\r\| + \l\|\eta\r\|_{\L^2}\r) \r] \d \lambda
            % lambda < 1
            \\&\lesssim
            \l\|\eta\r\|_{\L^2}^2
        \end{flalign*}
        which implies the desired result.
    \end{proof}

    \begin{lem}\label{lem:Y-SDE}
        If Assumptions~\ref{ass:1-lipschitz} and \ref{ass:2-first-derivative} are satisfied with $q_0\geq 4$, then there $G$-SDEs
        \begin{align*}
            \d Y^{t,\xi,\eta}_s 
            &= 
            \l[ \D x b\!\l(s,X^{t,\xi}_s, X^{t,\xi}_s\r) Y^{t,\xi,\eta}_s
            + 
            \D \xi b\!\l(s,X^{t,\xi}_s, X^{t,\xi}_s\r)\!\l(\!\D x X^{t,\xi,\xi}_s\eta + Y^{t,\xi,\eta}_s\r) \r]\! \d s
            \\ & \quad +
            \l[ \D x h\!\l(s,X^{t,\xi}_s, X^{t,\xi}_s\r) Y^{t,\xi,\eta}_s 
            + 
            \D \xi h\!\l(s,X^{t,\xi}_s, X^{t,\xi}_s\r)\!\l(\!\D x X^{t,\xi,\xi}_s\eta + Y^{t,\xi,\eta}_s\r) \r]\! \d \l<B\r>_s
            \\ & \quad + 
            \l[ \D x g\!\l(s,X^{t,\xi}_s, X^{t,\xi}_s\r) Y^{t,\xi,\eta}_s
            +  
            \D \xi g\!\l(s,X^{t,\xi}_s, X^{t,\xi}_s\r)\!\l(\!\D x X^{t,\xi,\xi}_s\eta + Y^{t,\xi,\eta}_s\r) \r]\! \d B_s, \\
            &\hspace{12.75cm} t\leq s\leq T,
            \\
            Y^{t,\xi,\eta}_t&=\eta,
            \numberthis\label{eq:SDE-xi-derivative-xi}
            \\
            \d Y^{t,x,\xi,\eta}_s 
            &= 
            \l[ \D x b\!\l(s,X^{t,x,\xi}_s, X^{t,\xi}_s\r) Y^{t,x,\xi,\eta}_s
            + 
            \D \xi b\!\l(s,X^{t,x,\xi}_s, X^{t,\xi}_s\r)\!\l(\!\D x X^{t,\xi,\xi}_s\eta + Y^{t,\xi,\eta}_s\r) \r]\! \d s
            \\ & \quad +
            \l[ \D x h\!\l(s,X^{t,x,\xi}_s, X^{t,\xi}_s\r) Y^{t,x,\xi,\eta}_s 
            + 
            \D \xi h\!\l(s,X^{t,x,\xi}_s, X^{t,\xi}_s\r)\!\l(\!\D x X^{t,\xi,\xi}_s\eta + Y^{t,\xi,\eta}_s\r) \r]\! \d \l<B\r>_s
            \\ & \quad + 
            \l[ \D x g\!\l(s,X^{t,x,\xi}_s, X^{t,\xi}_s\r) Y^{t,x,\xi,\eta}_s
            +  
            \D \xi g\!\l(s,X^{t,x,\xi}_s, X^{t,\xi}_s\r)\!\l(\!\D x X^{t,\xi,\xi}_s\eta + Y^{t,\xi,\eta}_s\r) \r]\! \d B_s, \\
            &\hspace{12.75cm} t\leq s\leq T,
            \\
            Y^x_t&=\eta.
            \numberthis\label{eq:SDE-xi-derivative-x}
        \end{align*}
        admit unique solutions $Y^{t,\xi,\eta},Y^{t,x,\xi,\eta} \in\H^{2,d}(t,T)$ for all $0\leq t\leq T$, $x\in\Rbb^d$ and $\xi,\eta\in\L^{2,d}(t)$.
        Moreover, the map
        \begin{equation*}
            \L^{2,d}(t)\rightarrow\H^{2,d}(t,T),\qquad \eta \mapsto Y^{t,x,\xi,\eta}
        \end{equation*}
        is linear.
    \end{lem}
    \begin{proof}
        We have $\D x X^{t,\xi,\xi} \eta\in\H^{2,d}(t,T)$ due to Lemma~\ref{lem:Dx-concatenation}. Thus, Lemma~\ref{lem:coefficients-differentiable-1} implies that the coefficients in \eqref{eq:SDE-xi-derivative-xi} are in $\M^2(0,T)$. Since they are Lipschitz continuous, \eqref{eq:SDE-xi-derivative-xi} admits a unique solution $Y^{t,\xi,\eta}\in\H^{2,d}(t,T)$.
        
        Similarly, since $Y^{t,\xi,\eta}\in \H^{2,d}(t,T)$, the coefficients in \eqref{eq:SDE-xi-derivative-x} are in $\M^2(0,T)$ and Lipschitz continuous and, thus, \eqref{eq:SDE-xi-derivative-x} admits a unique solution $Y^{t,x,\xi,\eta}\in\H^{2,d}(t,T)$.

        Let $\eta,\zeta\in\L^{2,d}(t)$ and $\lambda\in\Rbb$.
        Lemma~\ref{lem:conditional} yields for all $t\leq s\leq T$
        \begin{flalign*}
            \quad
            \E&\l[ \sup_{t\leq w\leq s}\l\| Y^{t,x,\xi,\eta+\lambda\zeta}_w - Y^{t,x,\xi,\eta}_w - \lambda Y^{t,x,\xi,\zeta}_w \r\|^2\r]
            &&
            % lem:conditional
            \\&\lesssim
            \sum_{f\in F} \int_t^s \E\l[ \l| \D x f\!\l(u,X^{t,x,\xi}_u,X^{t,\xi}_u\r)\!\l(Y^{t,x,\xi,\eta+\lambda\zeta}_u - Y^{t,x,\xi,\eta}_u - \lambda Y^{t,x,\xi,\zeta}_u\r) \r|^2 \r] \d u
            \\&\quad
            + \sum_{f\in F} \int_t^s \E\l[ \l| \D \xi f\!\l(u,X^{t,x,\xi}_u,X^{t,\xi}_u\r)\!\l(Y^{t,x,\xi,\eta+\lambda\zeta}_u - Y^{t,x,\xi,\eta}_u - \lambda Y^{t,x,\xi,\zeta}_u\r) \r|^2 \r] \d u
            % Dx f, Dxi f bound
            \\&\lesssim
            \int_t^s \alpha_0\!\l(u\r)^2 \E\l[ \l\| Y^{t,x,\xi,\eta+\lambda\zeta}_u - Y^{t,x,\xi,\eta}_u - \lambda Y^{t,x,\xi,\zeta}_u \r\|^2 \r] \d u,
        \end{flalign*}
       and Grönwall's inequality yields $\l\| Y^{t,x,\xi,\eta+\lambda\zeta} - Y^{t,x,\xi,\eta} - \lambda Y^{t,x,\xi,\zeta}\r\|_{\H^2}=0$.
    \end{proof}

    \begin{lem}\label{lem:Y-2-bound}
        If Assumptions~\ref{ass:1-lipschitz} and \ref{ass:2-first-derivative} are satisfied with $q_0\geq 4$, then 
        \begin{equation*}
            \E\l[ \sup_{t\leq w\leq T}\l\| Y^{t,\xi,\eta}_w \r\|^2 \r] \lesssim \l\|\eta\r\|_{\L^2}^2
        \end{equation*}
        for all $0\leq t\leq T$ and $\xi,\eta\in\L^{2,d}(t)$.
    \end{lem}
    \begin{proof}
        By Lemma~\ref{lem:conditional}, we have for all $t\leq s\leq T$
        \begin{flalign*}
            \quad
            \E&\l[\sup_{t\leq w\leq s}\l\| Y^{t,\xi,\eta}_w\r\|^2 \r]
            &&
            % lem:conditional
            \\&\lesssim
            \l\|\eta\r\|_{\L^2}^2 +
            \sum_{f\in F} \int_t^s \E\l[ \l| \D x f\!\l(u,X^{t,\xi}_u,X^{t,\xi}_u\r)Y^{t,\xi,\eta}_u \r|^2 \r] \d u
            \\&\quad 
            + \sum_{f\in F} \int_t^s \E\l[ \l| \D \xi f\!\l(u,X^{t,\xi}_u,X^{t,\xi}_u\r)\!\l(\D x X^{t,\xi,\xi}_u\eta + Y^{t,\xi,\eta}_u\r) \r|^2\r] \d u
            % Dx f, D\xi f bound
            %\\&\lesssim
            %\l\|\eta\r\|_{\L^2}^2 +
            %\int_t^s \alpha_0\!\l(u\r)^2 \l( \E\l[\l\|Y^{t,\xi,\eta}_u\r\|^2\r] + \l\|\D x X^{t,\xi,\xi}_u \eta + Y^{t,\xi,\eta}_u\r\|_{\L^2}^2 \r) \d u
            % Jensen's inequality
            \\&\lesssim
            \l\|\eta\r\|_{\L^2}^2 +
            \int_t^s \alpha_0\!\l(u\r)^2 \l( \E\l[\l\|Y^{t,\xi,\eta}_u\r\|^2\r] + \l\|\D x X^{t,\xi,\xi}_u \eta \r\|_{\L^2}^2 \r) \d u
            % lem:Dx-concatenation
            \\&\lesssim
            \l\|\eta\r\|_{\L^2}^2 +
            \int_t^s \alpha_0\!\l(u\r)^2 \E\l[\l\|Y^{t,\xi,\eta}_u\r\|^2\r] \d u
        \end{flalign*}
        due to Lemma~\ref{lem:Dx-concatenation}. Finally, Grönwall's inequality yields the desired result.
    \end{proof}

    \begin{lem}\label{lem:Y-p-bound}
        Let $2\leq p\leq q_0$. 
        If Assumptions~\ref{ass:1-lipschitz} and \ref{ass:2-first-derivative} are satisfied with $q_0\geq 4$, then 
        \begin{align*}
            \E\l[ \sup_{t\leq w\leq T}\l\| Y^{t,\xi,\eta}_w \r\|^p \,\Big|\,\CF_t\r] &\lesssim \l\|\eta\r\|^p + \l\|\eta\r\|_{\L^2}^p, \\
            \E\l[ \sup_{t\leq w\leq T}\l\| Y^{t,x,\xi,\eta}_w \r\|^p \,\Big|\,\CF_t\r] &\lesssim \l\|\eta\r\|^p + \l\|\eta\r\|_{\L^2}^p
        \end{align*}
        for all $0\leq t\leq T$, $x\in\Rbb^d$ and $\xi,\eta\in\L^{2,d}(t)$.
    \end{lem}
    \begin{proof}
        By Lemma~\ref{lem:conditional}, we have for all $t\leq s\leq T$
        \begin{flalign*}
            \quad
            \E&\l[\sup_{t\leq w\leq s}\l\| Y^{t,\xi,\eta}_w\r\|^p \,\Big|\,\CF_t\r]
            &&
            % lem:conditional
            \\&\lesssim
            \l\|\eta\r\|^p +
            \sum_{f\in F} \int_t^s \E\l[ \l| \D x f\!\l(u,X^{t,\xi}_u,X^{t,\xi}_u\r)Y^{t,\xi,\eta}_u \r|^p \,\Big|\,\CF_t \r] \d u
            \\&\quad 
            + \sum_{f\in F} \int_t^s \E\l[ \l| \D \xi f\!\l(u,X^{t,\xi}_u,X^{t,\xi}_u\r)\!\l(\D x X^{t,\xi,\xi}_u\eta + Y^{t,\xi,\eta}_u\r) \r|^p \,\Big|\,\CF_t\r] \d u
            % Dx f, D\xi f bound
            \\&\lesssim
            \l\|\eta\r\|^p +
            \int_t^s \alpha_0\!\l(u\r)^p \l( \E\l[\l\|Y^{t,\xi,\eta}_u\r\|^p\,\Big|\,\CF_t\r] + \l\|\D x X^{t,\xi,\xi}_u \eta \r\|_{\L^2}^p + \l\|Y^{t,\xi,\eta}_u \r\|_{\L^2}^p \r) \d u
            % lem:Y-2-bound, lem:Dx-concatenation
            \\&\lesssim
            \l\|\eta\r\|^p + \l\|\eta\r\|_{\L^2}^p + 
            \int_t^s \alpha_0\!\l(u\r)^p \E\l[\l\|Y^{t,\xi,\eta}_u\r\|^p\,\Big|\,\CF_t\r] \d u
            ,
        \end{flalign*}
        and Grönwall's inequality yields the desired result for $Y^{t,\xi,\eta}$.
        
        Analogously, we have for $Y^{t,x,\xi,\eta}$ that
        \begin{flalign*}
            \quad
            \E&\l[\sup_{t\leq w\leq s}\l\| Y^{t,x,\xi,\eta}_w\r\|^p\,\Big|\,\CF_t \r]
            &&
            % lem:conditional
            \\&\lesssim
            \l\|\eta\r\|^p +
            \sum_{f\in F} \int_t^s \E\l[ \l| \D x f\!\l(u,X^{t,x,\xi}_u,X^{t,\xi}_u\r)Y^{t,x,\xi,\eta}_u \r|^p \,\Big|\,\CF_t  \r] \d u 
            \\&\quad 
            + \sum_{f\in F} \int_t^s \E\l[ \l| \D \xi f\!\l(u,X^{t,x,\xi}_u,X^{t,\xi}_u\r)\!\l(\D x X^{t,\xi,\xi}_u\eta + Y^{t,\xi,\eta}_u\r) \r|^p \,\Big|\,\CF_t \r] \d u
            % Dx f, D\xi f bound
            \\&\lesssim
            \l\|\eta\r\|^p +
            \int_t^s \alpha_0\!\l(u\r)^p \l( \E\l[\l\|Y^{t,x,\xi,\eta}_u\r\|^p \,\Big|\,\CF_t  \r] + \l\|\D x X^{t,\xi,\xi}_u \eta \r\|_{\L^2}^p + \l\|Y^{t,\xi,\eta}_u\r\|_{\L^2}^p\r) \d u
            % Dx-p-bound
            \\&\lesssim
            \l\|\eta\r\|^p +
            \l\|\eta\r\|_{\L^2}^p + \int_t^s \alpha_0\!\l(u\r)^p \E\l[\l\|Y^{t,x,\xi,\eta}_u\r\|^p \,\Big|\,\CF_t  \r] \d u,
        \end{flalign*}
        and Grönwall's inequality yields the desired result for $Y^{t,x,\xi,\eta}$.
    \end{proof}

    \begin{lem}\label{lem:Y-concatenation}
        Let $0\leq t\leq T$ and $\xi,\eta\in\L^{2,d}(t)$.
        If Assumptions~\ref{ass:1-lipschitz} and \ref{ass:2-first-derivative} are satisfied with $q_0\geq 4$, then
        \begin{equation*}
            \l\| Y^{t,\xi,\eta} - Y^{t,\xi,\xi,\eta}\r\|_{\H^2}=0,
        \end{equation*}
        where $ Y^{t,\xi,\xi,\eta}$ denotes the map
        \begin{equation*}
            [0,T]\times \Omega \rightarrow \Rbb^d, \qquad (s,\omega)\mapsto Y^{t,\xi,\xi,\eta}_s(\omega):=Y^{t,x,\xi,\eta}_s(\omega)|_{x=\xi(\omega)}. 
        \end{equation*}
    \end{lem}
    \begin{proof}
        Set $Z:=\D x X^{t,\xi,\xi}\eta + Y^{t,\xi,\eta}$, then $\l\|Z\r\|_{\H^2} \lesssim \l\|\eta\r\|_{\L^2}$ due to Lemmas~\ref{lem:Dx-concatenation} and \ref{lem:Y-2-bound}.
        By Lemma~\ref{lem:conditional}, we have for all $t\leq s\leq T$
        \begin{flalign*}
            \quad
            \E&\l[\sup_{t\leq w\leq s}\l\| Y^{t,\xi,\eta}_w - Y^{t,x,\xi,\eta}_w\r\|^2 \,\Big|\,\CF_t \r]
            &&
            % lem:conditional
            %\\&\lesssim
            %\sum_{f\in F} \int_t^s \E\l[\l| \D x f\!\l(u,X^{t,\xi}_u,X^{t,\xi}_u\r)Y^{t,\xi,\eta}_u - \D x f\!\l(u,X^{t,x,\xi}_u,X^{t,\xi}_u\r)Y^{t,x,\xi,\eta}_u \r|^2 \,\Big|\,\CF_t\r] \d u
            %\\&\quad
            %+ \sum_{f\in F} \int_t^s \E\l[\l| \D \xi f\!\l(u,X^{t,\xi}_u,X^{t,\xi}_u\r) Z_u - \D \xi f\!\l(u,X^{t,x,\xi}_u,X^{t,\xi}_u\r) Z_u \r|^2 \,\Big|\,\CF_t\r] \d u
            \\&\lesssim
            \sum_{f\in F} \int_t^s \E\l[\l| \D x f\!\l(u,X^{t,\xi}_u,X^{t,\xi}_u\r)Y^{t,\xi,\eta}_u - \D x f\!\l(u,X^{t,x,\xi}_u,X^{t,\xi}_u\r) Y^{t,\xi,\eta}_u \r|^2 \,\Big|\,\CF_t \r] \d u
            \\&\quad
            + \sum_{f\in F} \int_t^s \E\l[ \l| \D x f\!\l(u,X^{t,x,\xi}_u,X^{t,\xi}_u\r)\!\l(Y^{t,\xi,\eta}_u - Y^{t,x,\xi,\eta}_u\r) \r|^2 \,\Big|\,\CF_t\r] \d u
            \\&\quad
            + \sum_{f\in F} \int_t^s \E\l[\l| \D \xi f\!\l(u,X^{t,\xi}_u,X^{t,\xi}_u\r) Z_u - \D \xi f\!\l(u,X^{t,x,\xi}_u,X^{t,\xi}_u\r) Z_u \r|^2 \,\Big|\,\CF_t\r] \d u
            % Dx f Lipschitz
            \\&\lesssim
            \int_t^s \alpha_1\!\l(u\r)^2 \E\l[\l\| X^{t,\xi}_u - X^{t,x,\xi}_u\r\|^2 \l\| Y^{t,\xi,\eta}_u \r\|^2 \,\Big|\,\CF_t \r] + \alpha_0\!\l(u\r)^2 \E\l[\l\| Y^{t,\xi,\eta}_u - Y^{t,x,\xi,\eta}_u\r\|^2 \,\Big|\,\CF_t \r] \d u
            \\&\quad 
            + \int_t^s \alpha_1\!\l(u\r)^2 \l\| Z_u \r\|_{\L^2}^2 \E\l[\l\| X^{t,\xi}_u - X^{t,x,\xi}_u\r\|^2 \,\Big|\,\CF_t \r] \d u
            % Cauchy-Schwarz
            \\&\leq 
            \int_t^s \alpha_1\!\l(u\r)^2 \E\l[\l\| X^{t,y,\xi}_u - X^{t,x,\xi}_u\r\|^4 \,\Big|\,\CF_t \r]^{\frac12} \Big|_{y=\xi} \E\l[\l\| Y^{t,\xi,\eta}_u \r\|^4 \,\Big|\,\CF_t \r]^{\frac12} \d u
            \\&\quad 
            + \int_t^s \alpha_0\!\l(u\r)^2 \E\l[\l\| Y^{t,\xi,\eta}_u - Y^{t,x,\xi,\eta}_u \r\|^2 \,\Big|\,\CF_t \r] \d u 
            \\&\quad 
            + \int_t^s \alpha_1\!\l(u\r)^2 \l\| Z_u \r\|_{\L^2}^2 \E\l[\l\| X^{t,y,\xi}_u - X^{t,x,\xi}_u\r\|^2 \,\Big|\,\CF_t \r] \Big|_{y=\xi}\d u
            % lem: Y-p-bound, xxi-yeta-p-bound
            %\\&\lesssim
            %\int_t^s \alpha_1\!\l(u\r)^2 \l\| \xi - x\r\|^2 \l(\l\|\zeta\r\|^2 + \l\|\zeta\r\|_{\L^2}^2 \r) + \alpha_0\!\l(u\r)^2 \E\l[\l\| Y^{t,\xi,\eta}_u - Y^{t,x,\xi,\eta}_u \r\|^2 \,\Big|\,\CF_t \r] \d u
            \\&\lesssim
            \l\| \xi - x\r\|^2 \l(\l\|\zeta\r\|^2 + \l\|\zeta\r\|_{\L^2}^2 \r) 
            + \int_t^s \alpha_0\!\l(u\r)^2 \E\l[\l\| Y^{t,\xi,\eta}_u - Y^{t,x,\xi,\eta}_u \r\|^2 \,\Big|\,\CF_t \r] \d u
            .
        \end{flalign*}
        Grönwall's inequality yields
        \begin{equation*}
            \E\l[\sup_{t\leq w\leq s}\l\| Y^{t,\xi,\eta}_w - Y^{t,x,\xi,\eta}_w\r\|^2 \,\Big|\,\CF_t \r] \lesssim \l\| \xi - x\r\|^2 \l(\l\|\zeta\r\|^2 + \l\|\zeta\r\|_{\L^2}^2 \r)
        \end{equation*}
        and, thus, the aggregation property implies
        \begin{align*}
            \l\|Y^{t,\xi,\eta}-Y^{t,\xi,\xi,\eta}\r\|_{\H^2}^2 
            &= \E\l[ \E\l[\sup_{t\leq w\leq T}\l\| Y^{t,\xi,\eta}_w - Y^{t,x,\xi,\eta}_w\r\|^2 \,\Big|\,\CF_t \r]\bigg|_{x=\xi} \r]
            %\\&\lesssim
            %\E\l[ \l\| \xi - x\r\|^2 \l(\l\|\zeta\r\|^2 + \l\|\zeta\r\|_{\L^2}^2 \r) \Big|_{x=\xi} \r]
            = 0.
        \end{align*}
    \end{proof}

    \begin{lem}\label{lem:Y-Y-p-bound}
        Let $2\leq p\leq \l( q_1 \wedge \frac{q_0}2\r)$.
        If Assumptions~\ref{ass:1-lipschitz} and \ref{ass:2-first-derivative} are satisfied with $q_0\geq 4$, then 
        \begin{align*}
            \E\l[ \sup_{t\leq w\leq T}\l\| Y^{t,x,\xi,\zeta}_w - Y^{t,y,\eta,\zeta}_w \r\|^p\,\Big|\,\CF_t\r] 
            & \lesssim \l\|\zeta\r\|_{\L^2}^p \l( \l\|x-y\r\|^p +\l\|\xi-\eta\r\|_{\L^2}^p\r)
        \end{align*}
        for all $0\leq t\leq T$, $x,y\in\Rbb^d$ and $\xi,\eta,\zeta\in\L^{2,d}(t)$.
    \end{lem}
    \begin{proof}
        Set $Z^{\xi}:= \D x X^{t,\xi,\xi}\zeta + Y^{t,\xi,\zeta}$ and $Z^{\eta}:= \D x X^{t,\eta,\eta}\zeta + Y^{t,\eta,\zeta}$,
        then
        \begin{align*}
            \l\| Z^\xi \r\|_{\H^2} + \l\| Z^\eta \r\|_{\H^2} &\lesssim \l\|\zeta\r\|_{\L^2}
        \end{align*}
        due to Lemmas~\ref{lem:Dx-concatenation} and \ref{lem:Y-2-bound}. Moreover,
        \begin{align*}
            \l\| Z^\xi_s - Z^\eta_s \r\|_{\L^1} 
            % triangle inequality
            &\leq 
            \l\| \D x X^{t,\xi,\xi}_s \zeta - \D x X^{t,\eta,\eta}_s \zeta \r\|_{\L^1} + \l\|Y^{t,\xi,\zeta}_s - Y^{t,\eta,\zeta}_s \r\|_{\L^1}
            % cor:Dx-Dx-L1
            \\&\lesssim
            \l\|\zeta\r\|_{\L^2}\l\|\xi-\eta\r\|_{\L^2} + \l\|Y^{t,\xi,\zeta}_s - Y^{t,\eta,\zeta}_s \r\|_{\L^1}
            \numberthis\label{ineq:Z-Z}
        \end{align*}
        for all $t\leq s\leq T$ due to Corollary~\ref{cor:Dx-Dx-L1}.
        
        By Lemma~\ref{lem:conditional}, we have for all $t\leq s\leq T$
        \begin{flalign*}
            \quad& 
            \E\l[ \sup_{t\leq w\leq s}\l\| Y^{t,x,\xi,\zeta}_w - Y^{t,y,\eta,\zeta}_w \r\|^p\,\Big|\,\CF_t\r]
            &&
            % lem:conditional
            %\\&\lesssim
            %\sum_{f\in F} \int_t^s \E\l[ \l| \D x f\!\l(u,X^{t,x,\xi}_u,X^{t,\xi}_u\r)Y^{t,x,\xi,\zeta}_u - \D x f\!\l(u,X^{t,y,\eta}_u,X^{t,\eta}_u\r)Y^{t,y,\eta,\zeta}_u  \r|^p\,\Big|\,\CF_t \r] \d u
            %\\&\quad 
            %+ \sum_{f\in F} \int_t^s \E\l[ \l| \D \xi f\!\l(u,X^{t,x,\xi}_u,X^{t,\xi}_u\r)Z^{\xi}_u - \D \xi f\!\l(u,X^{t,y,\eta}_u,X^{t,\eta}_u\r)Z^{\eta}_u  \r|^p\,\Big|\,\CF_t \r] \d u
            \\&\lesssim
            \sum_{f\in F} \int_t^s \E\l[ \l| \D x f\!\l(u,X^{t,x,\xi}_u,X^{t,\xi}_u\r)Y^{t,x,\xi,\zeta}_u - \D x f\!\l(u,X^{t,y,\eta}_u,X^{t,\eta}_u\r) Y^{t,x,\xi,\zeta}_u \r|^p\,\Big|\,\CF_t \r] \d u
            \\&\quad 
            + \sum_{f\in F} \int_t^s \E\l[ \l| \D x f\!\l(u,X^{t,y,\eta}_u,X^{t,\eta}_u\r)\!\l( Y^{t,x,\xi,\zeta}_u - Y^{t,y,\eta,\zeta}_u\r)  \r|^p\,\Big|\,\CF_t \r] \d u
            \\&\quad
            + \sum_{f\in F} \int_t^s \E\l[ \l|  \D \xi f\!\l(u,X^{t,x,\xi}_u,X^{t,\xi}_u\r)Z^{\xi}_u  - \D \xi f\!\l(u,X^{t,y,\eta}_u,X^{t,\eta}_u\r) Z^{\xi}_u \r|^p\,\Big|\,\CF_t \r] \d u
            \\&\quad 
            + \sum_{f\in F} \int_t^s \E\l[ \l| \D \xi f\!\l(u,X^{t,y,\eta}_u,X^{t,\eta}_u\r)\!\l( Z^{\xi}_u - Z^{\eta}_u\r)  \r|^p\,\Big|\,\CF_t \r] \d u
            % Dx f, Dxi f Lipschitz, bounded
            \\&\lesssim
            \int_t^s \alpha_1\!\l(u\r)^p \E\l[ \l\| Y^{t,x,\xi,\zeta}_u \r\|^p \l\| X^{t,x,\xi}_u - X^{t,y,\eta}_u\r\|^p \,\Big|\,\CF_t \r] \d u
            \\&\quad 
            +\int_t^s  \alpha_1\!\l(u\r)^p \l\| X^{t,\xi}_u - X^{t,\eta}_u \r\|_{\L^2}^p \E\l[ \l\| Y^{t,x,\xi,\zeta}_u \r\|^p \,\Big|\,\CF_t \r] \d u
            \\&\quad 
            + \int_t^s \alpha_0\!\l(u\r)^p \E\l[ \l\| Y^{t,x,\xi,\zeta}_u - Y^{t,y,\eta,\zeta}_u \r\|^p\,\Big|\,\CF_t \r] \d u
            \\&\quad
            + \int_t^s \alpha_1\!\l(u\r)^p \l\|Z^\xi_u\r\|_{\L^2}^p \l( \E\l[ \l\| X^{t,x,\xi}_u - X^{t,y,\eta}_u\r\|^p \,\Big|\,\CF_t\r] + \l\| X^{t,\xi}_u - X^{t,\eta}_u \r\|_{\L^2}^p \r)  \d u
            \\&\quad 
            + \int_t^s \alpha_1\!\l(u\r)^p \l\| Z^\xi_u - Z^\eta_u \r\|_{\L^1}^p \d u
            % Cauchy-Schwarz
            %\\&\leq
            %\int_t^s \alpha_1\!\l(u\r)^p \E\l[ \l\| Y^{t,x,\xi,\zeta}_u \r\|^{2p}  \,\Big|\,\CF_t \r]^{\frac12} \E\l[ \l\| X^{t,x,\xi}_u - X^{t,y,\eta}_u\r\|^{2p} \,\Big|\,\CF_t \r]^{\frac12} \d u
            %\\&\quad 
            %+\int_t^s  \alpha_1\!\l(u\r)^p \l\| X^{t,\xi}_u - X^{t,\eta}_u \r\|_{\L^2}^p \E\l[ \l\| Y^{t,x,\xi,\zeta}_u \r\|^p \,\Big|\,\CF_t \r] \d u
            %\\&\quad 
            %+ \int_t^s \alpha_0\!\l(u\r)^p \E\l[ \l\| Y^{t,x,\xi,\zeta}_u - Y^{t,y,\eta,\zeta}_u \r\|^p\,\Big|\,\CF_t \r] \d u
            %\\&\quad
            %+ \int_t^s \alpha_1\!\l(u\r)^p \l\|Z^\xi_u\r\|_{\L^2}^p \l( \E\l[ \l\| X^{t,x,\xi}_u - X^{t,y,\eta}_u\r\|^p \,\Big|\,\CF_t\r] + \l\| X^{t,\xi}_u - X^{t,\eta}_u \r\|_{\L^2}^p \r)  \d u
            %\\&\quad 
            %+ \int_t^s \alpha_1\!\l(u\r)^p \l\| Z^\xi_u - Z^\eta_u \r\|_{\L^1}^p \d u
            % lem:Y-p-bound, xxi-yeta, ineq:Z-Z
            \\&\lesssim
            \int_t^s \alpha_1\!\l(u\r)^p \l( \l\|\zeta\r\|^p + \l\|\zeta\r\|_{\L^2}^p \r) 
            \l( \l\|x-y\r\|^p + \l\|\xi-\eta\r\|_{\L^2}^p\r) \d u
            \\&\quad 
            +\int_t^s  \alpha_1\!\l(u\r)^p \l\| \xi-\eta \r\|_{\L^2}^p \l( \l\|\zeta\r\|^p + \l\|\zeta\r\|_{\L^2}^p \r) \d u
            \\&\quad 
            + \int_t^s \alpha_0\!\l(u\r)^p \E\l[ \l\| Y^{t,x,\xi,\zeta}_u - Y^{t,y,\eta,\zeta}_u \r\|^p\,\Big|\,\CF_t \r] \d u
            \\&\quad
            + \int_t^s \alpha_1\!\l(u\r)^p \l\|\zeta \r\|_{\L^2}^p \l( \l\|x-y\r\|^p + \l\|\xi-\eta\r\|_{\L^2}^p\r)  \d u
            \\&\quad 
            + \int_t^s \alpha_1\!\l(u\r)^p \l\|\zeta\r\|_{\L^2}^p\l\|\xi-\eta\r\|_{\L^2}^p + \alpha_1\!\l(u\r)^p \l\|Y^{t,\xi,\zeta}_u - Y^{t,\eta,\zeta}_u \r\|_{\L^1}^p \d u
            \\&\lesssim
            \l( \l\|\zeta\r\|^p + \l\|\zeta\r\|_{\L^2}^p \r) 
            \l( \l\|x-y\r\|^p + \l\|\xi-\eta\r\|_{\L^2}^p\r)
            +\int_t^s 
            \alpha_1\!\l(u\r)^p \l\|Y^{t,\xi,\zeta}_u - Y^{t,\eta,\zeta}_u \r\|_{\L^1}^p \d u
            \\&\quad 
            + \int_t^s \alpha_0\!\l(u\r)^p \E\l[ \l\| Y^{t,x,\xi,\zeta}_u - Y^{t,y,\eta,\zeta}_u \r\|^p\,\Big|\,\CF_t \r] \d u
            .
        \end{flalign*}
        due to \eqref{ineq:Z-Z} and Lemmas~\ref{lem:xxi-yeta-p-bound} and \ref{lem:Y-p-bound}.
        Further, Grönwall's inequality implies that
        \begin{flalign*}
            \quad &
            \E\l[ \sup_{t\leq w\leq s}\l\| Y^{t,x,\xi,\zeta}_w - Y^{t,y,\eta,\zeta}_w \r\|^p\,\Big|\,\CF_t\r]
            &&
            \\&\lesssim
            \l( \l\|\zeta\r\|^p + \l\|\zeta\r\|_{\L^2}^p \r) 
            \l( \l\|x-y\r\|^p + \l\|\xi-\eta\r\|_{\L^2}^p\r)
            +\int_t^s 
            \alpha_1\!\l(u\r)^p \l\|Y^{t,\xi,\zeta}_u - Y^{t,\eta,\zeta}_u \r\|_{\L^1}^p \d u
            \numberthis\label{ineq:Y-Y-p-0}
        \end{flalign*}
        for all $t\leq s\leq T$. From Lemma~\ref{lem:Y-concatenation} and \eqref{ineq:Y-Y-p-0} we obtain
        \begin{flalign*}
            \quad &
            \E\l[ \sup_{t\leq w\leq s}\l\| Y^{t,\xi,\zeta}_w - Y^{t,\eta,\zeta}_w \r\|\r]^2
            &&
            % aggregation property
            %\\&=
            %\E\l[ \E\l[ \sup_{t\leq w\leq s}\l\| Y^{t,x,\xi,\zeta}_w - Y^{t,y,\eta,\zeta}_w \r\|\,\Big|\,\CF_t\r] \bigg|_{x=\xi,\,y=\eta}\r]^2
            % Jensen
            \\&\leq 
            \E\l[ \E\l[ \sup_{t\leq w\leq s}\l\| Y^{t,x,\xi,\zeta}_w - Y^{t,y,\eta,\zeta}_w \r\|^2\,\Big|\,\CF_t\r]^{\frac12} \bigg|_{x=\xi,\,y=\eta}\r]^2
            %\\&\lesssim
            %\E\l[ \l(\l( \l\|\zeta\r\|^2 + \l\|\zeta\r\|_{\L^2}^2 \r) 
            %\l( \l\|\xi-\eta\r\|^2 + \l\|\xi-\eta\r\|_{\L^2}^2\r) + \int_t^s 
            %\alpha_1\!\l(u\r)^2 \l\|Y^{t,\xi,\zeta}_u - Y^{t,\eta,\zeta}_u \r\|_{\L^1}^2 \d u\r)^{\frac12} \r]^2
            % square-root subadditive
            %\\&\leq 
            %\E\l[ \l( \l\|\zeta\r\| + \l\|\zeta\r\|_{\L^2} \r)
            %\l( \l\|\xi-\eta\r\| + \l\|\xi-\eta\r\|_{\L^2}\r) + \l( \int_t^s 
            %\alpha_1\!\l(u\r)^2 \l\|Y^{t,\xi,\zeta}_u - Y^{t,\eta,\zeta}_u \r\|_{\L^1}^2 \d u\r)^{\frac12} \r]^2
            % Jensen
            \\&\lesssim
            \E\l[ \l( \l\|\zeta\r\| + \l\|\zeta\r\|_{\L^2} \r) 
            \l( \l\|\xi-\eta\r\| + \l\|\xi-\eta\r\|_{\L^2}\r) \r]^2
            + \int_t^s \alpha_1\!\l(u\r)^2 \l\|Y^{t,\xi,\zeta}_u - Y^{t,\eta,\zeta}_u \r\|_{\L^1}^2 \d u
            % Cauchy-Schwarz, Jensen
            \\&\lesssim 
            \l\|\zeta\r\|_{\L^2}^2 \l\|\xi-\eta\r\|_{\L^2}^2 + \int_t^s 
            \alpha_1\!\l(u\r)^2 \l\|Y^{t,\xi,\zeta}_u - Y^{t,\eta,\zeta}_u \r\|_{\L^1}^2 \d u
            ,
        \end{flalign*}
        and Grönwall's inequality yields
        \begin{equation*}
            \E\l[ \sup_{t\leq w\leq s}\l\| Y^{t,\xi,\zeta}_w - Y^{t,\eta,\zeta}_w \r\|\r]
            \lesssim 
            \l\|\zeta\r\|_{\L^2} \l\|\xi-\eta\r\|_{\L^2}.
        \end{equation*}
        
        Hence, \eqref{ineq:Y-Y-p-0} becomes
        \begin{align*}
            \E\l[ \sup_{t\leq w\leq T}\l\| Y^{t,x,\xi,\zeta}_w - Y^{t,y,\eta,\zeta}_w \r\|^p\,\Big|\,\CF_t\r]
            \lesssim \l( \l\|\zeta\r\|^p + \l\|\zeta\r\|_{\L^2}^p \r) 
            \l( \l\|x-y\r\|^p + \l\|\xi-\eta\r\|_{\L^2}^p\r).
        \end{align*}
    \end{proof}

    We immediately obtain the followin corollary.
    \begin{cor}\label{cor:Y-Y-2-bound}
        If Assumptions~\ref{ass:1-lipschitz} and \ref{ass:2-first-derivative} are satisfied with $q_0\geq 4$, then 
        \begin{align*}
            \E\l[ \sup_{t\leq w\leq T}\l\| Y^{t,\xi,\zeta}_w - Y^{t,\eta,\zeta}_w \r\|\r]
            \lesssim 
            \l\|\zeta\r\|_{\L^2} \l\|\xi-\eta\r\|_{\L^2}
        \end{align*}
        for all $0\leq t\leq T$ and $\xi,\eta,\zeta\in\L^{2,d}(t)$.
    \end{cor}

    \begin{lem}\label{lem:Y-diff}
        If Assumptions~\ref{ass:1-lipschitz} and \ref{ass:2-first-derivative} are satisfied with $q_0\geq 4$, then 
        \begin{flalign*}
             \quad&
             \E\l[ \sup_{t\leq w\leq s}\l\| X^{t,x,\xi+\eta}_w - X^{t,x,\xi}_w - Y^{t,x,\xi,\eta}_w \r\|^2 \,\Big|\,\CF_t \r]
             &&
             \\&\lesssim 
             \l\|\eta\r\|_{\L^2}^4 + \int_t^s \alpha_1\!\l(u\r)^2 \E\l[ \l\| X^{t,\xi+\eta}_u - X^{t,\xi}_u - \D x X^{t,\xi,\xi}_u\eta - Y^{t,\xi,\eta}_u \r\| \r]^2 \d u.
        \end{flalign*}
        for all $0\leq t\leq s\leq T$, $x\in\Rbb^d$ and $\xi,\eta\in\L^{2,d}(t)$.
    \end{lem}
    \begin{proof}
        Set 
        \begin{align*}
            \Delta^\xi &:=X^{t,x,\xi+\eta}-X^{t,x,\xi},
            &
            Y&:=Y^{t,x,\xi,\eta},
            \\
            \Delta&:=X^{t,\xi+\eta}-X^{t,\xi},
            &
            Z&:=\D x X^{t,\xi,\xi}\eta + Y^{t,\xi,\eta}.
        \end{align*}
        Lemmas~\ref{lem:Dx-concatenation} and \ref{lem:Y-p-bound} yield
        \begin{equation*}
            \l\|Z\r\|_{\H^2} + \l\|Y\r\|_{\H^2}\lesssim \l\|\eta\r\|_{\L^2}
            .\numberthis\label{ineq:ZY}
        \end{equation*}
        Moreover, Lemma~\ref{lem:H2-bound} implies
        \begin{align*}
            \l\|\Delta\r\|_{\H^2} &\lesssim \l\|\eta\r\|_{\L^2} 
            & 
            \E\l[ \sup_{t\leq w\leq T}\l\|\Delta^\xi_w\r\|^4\,\Big|\,\CF_t\r] &\lesssim \l\|\eta\r\|_{\L^2}^4
            \numberthis\label{ineq:Delta}
        \end{align*}
        
        By Lemma~\ref{lem:conditional}, we have for all $t\leq s\leq T$
        \begin{flalign*}
            \quad&
            \E\l[\sup_{t\leq w\leq s}\l\| \Delta^\xi_w - Y^{t,x,\xi,\eta}_w\r\|^2 \,\Big|\,\CF_t\r]
            &&
            % lem:conditional
            %\\&\lesssim
            %\sum_{f\in F} \int_t^s \E\l[ \l| f\!\l(u,X^{t,x,\xi+\eta}_u,X^{t,\xi+\eta}_u\r) - f\!\l(u,X^{t,x,\xi}_u,X^{t,\xi}_u\r) - \D x f\!\l(u,X^{t,x,\xi}_u,X^{t,\xi}_u\r) Y_u - \D \xi f\!\l(u,X^{t,x,\xi}_u,X^{t,\xi}_u\r) Z_u \r|^2 \,\Big|\,\CF_t \r] \d u
            \\&\lesssim
            \sum_{f\in F} \int_t^s \E\l[ \l| f\!\l(u,X^{t,x,\xi+\eta}_u,X^{t,\xi+\eta}_u\r) - f\!\l(u,X^{t,x,\xi}_u,X^{t,\xi+\eta}_u\r) - \D x f\!\l(u,X^{t,x,\xi}_u,X^{t,\xi}_u\r) Y_u \r|^2 \,\Big|\,\CF_t \r] \d u
            \\&\quad 
            + \sum_{f\in F} \int_t^s \E\l[ \l| f\!\l(u,X^{t,x,\xi}_u,X^{t,\xi+\eta}_u\r) - f\!\l(u,X^{t,x,\xi}_u,X^{t,\xi}_u\r)  - \D \xi f\!\l(u,X^{t,x,\xi}_u,X^{t,\xi}_u\r) Z_u\r|^2 \,\Big|\,\CF_t \r] \d u
            % continuously differentiable
            %\\&=
            %\sum_{f\in F} \int_t^s \E\l[ \l| \int_0^1 \D x f\!\l(u,X^{t,x,\xi}_u + \lambda \Delta^\xi_u,X^{t,\xi+\eta}_u\r) \Delta^\xi_u \d \lambda - \D x f\!\l(u,X^{t,x,\xi}_u,X^{t,\xi}_u\r) Y_u \r|^2 \,\Bigg|\,\CF_t \r] \d u
            %\\&\quad 
            %+ \sum_{f\in F} \int_t^s \E\l[ \l| \int_0^1 \D \xi f\!\l(u,X^{t,x,\xi}_u,X^{t,\xi}_u+\lambda \Delta_u \r) \Delta_u \d \lambda - \D \xi f\!\l(u,X^{t,x,\xi}_u,X^{t,\xi}_u\r) Z_u\r|^2 \,\Bigg|\,\CF_t \r] \d u
            \\&\lesssim
            \sum_{f\in F} \int_t^s \int_0^1 \E\l[ \l| \D x f\!\l(u,X^{t,x,\xi}_u + \lambda \Delta^\xi_u ,X^{t,\xi+\eta}_u\r) \Delta^\xi_u - \D x f\!\l(u,X^{t,x,\xi}_u,X^{t,\xi}_u\r) \Delta^\xi_u \r|^2 \,\Big|\,\CF_t \r] \d \lambda \d u
            \\&\quad 
            + \sum_{f\in F} \int_t^s \E\l[ \l| \D x f\!\l(u,X^{t,x,\xi}_u,X^{t,\xi}_u\r)\!\l( \Delta^\xi_u - Y_u\r) \r|^2 \,\Big|\,\CF_t \r] \d u
            \\&\quad 
            + \sum_{f\in F} \int_t^s \int_0^1 \E\l[ \l| \D \xi f\!\l(u,X^{t,x,\xi}_u,X^{t,\xi}_u+\lambda\Delta_u \r) \Delta_u - \D \xi f\!\l(u,X^{t,x,\xi}_u,X^{t,\xi}_u\r) \Delta_u \r|^2 \,\Big|\,\CF_t \r] \d \lambda \d u 
            \\&\quad 
            +\sum_{f\in F} \int_t^s \E\l[ \l| \D \xi f\!\l(u,X^{t,x,\xi}_u,X^{t,\xi}_u\r)\!\l(\Delta_u - Z_u\r)\r|^2 \,\Big|\,\CF_t \r] \d u
            % Dx f Lipschitz, Dx f bounded, D\xi f Lipschitz, D\xi f bounded
            %\\&\lesssim
            %\int_t^s \int_0^1 \alpha_1\!\l(u\r)^2 \l( \lambda^2 \E\l[\l\|\Delta^\xi_u\r\|^4\,\Big|\,\CF_t\r] + \l\|\Delta_u\r\|_{\L^2}^2 \E\l[\l\|\Delta^\xi_u\r\|^2\,\Big|\,\CF_t\r] + \lambda^2 \l\|\Delta_u\r\|_{\L^2}^4 \r)  \d \lambda \d u
            %\\&\quad 
            %+ \int_t^s 
            %\alpha_0\!\l(u\r)^2 \E\l[ \l\| \Delta^\xi_u - Y_u \r\|^2 \,\Big|\,\CF_t \r] 
            %+ \alpha_1\!\l(u\r)^2 \l\|\Delta_u - Z_u\r\|_{\L^1}^2 
            %\d u
            % lambda < 1
            \\&\lesssim
            \int_t^s \alpha_1\!\l(u\r)^2 \l( \E\l[\l\|\Delta^\xi_u\r\|^4\,\Big|\,\CF_t\r] + \l\|\Delta_u\r\|_{\L^2}^2 \E\l[\l\|\Delta^\xi_u\r\|^2\,\Big|\,\CF_t\r] + \l\|\Delta_u\r\|_{\L^2}^4 \r)  \d \lambda \d u
            \\&\quad 
            + \int_t^s 
            \alpha_0\!\l(u\r)^2 \E\l[ \l\| \Delta^\xi_u - Y_u \r\|^2 \,\Big|\,\CF_t \r] 
            + \alpha_1\!\l(u\r)^2 \l\|\Delta_u - Z_u\r\|_{\L^1}^2 
            \d u
            %ineq:Delta, YZ
            \\&\lesssim
            \l\|\eta\r\|_{\L^2}^4
            + \int_t^s 
            \alpha_0\!\l(u\r)^2 \E\l[ \l\| \Delta^\xi_u - Y_u \r\|^2 \,\Big|\,\CF_t \r] 
            + \alpha_1\!\l(u\r)^2 \l\|\Delta_u - Z_u\r\|_{\L^1}^2 
            \d u
        \end{flalign*}
        due to \eqref{ineq:Delta} and \eqref{ineq:ZY}.
        Finally, Grönwall's inequality implies the desired result.
    \end{proof}

    \begin{lem}\label{lem:xi-differentiability}
        Let $0\leq t\leq T$ and $\xi,\eta\in\L^{2,d}(t)$. 
        If Assumptions~\ref{ass:1-lipschitz} and \ref{ass:2-first-derivative} are satisfied with $q_0\geq 4$, then
        \begin{equation*}
            \lim_{\|\eta\|_{\L^2}\rightarrow 0} \frac{\l\| X^{t,\xi+\eta} - X^{t,\xi} - \D x X^{t,\xi,\xi}\eta - Y^{t,\xi,\eta} \r\|_{\H^{1}}}{\l\|\eta\r\|_{\L^2}} = 0,
        \end{equation*}
        where the limit is taken over $\eta\in\L^{2,d}(t)$.
    \end{lem}
    \begin{proof}
        By Lemmas~\ref{lem:Y-concatenation} and \ref{lem:Y-diff}, we have
        \begin{flalign*}
            \quad&
            \E\l[\sup_{t\leq w\leq s}\l\| X^{t,\xi,\xi+\eta}_w - X^{t,\xi,\xi}_w - Y^{t,\xi,\eta}_w \r\|^2\r]
            &&
            % lem:Y-concatenation
            \\&=
            \E\l[\E\l[ \sup_{t\leq w\leq s}\l\| X^{t,x,\xi+\eta}_w - X^{t,x,\xi}_w - Y^{t,x,\xi,\eta}_w \r\|^2\,\Big|\,\CF_t\r]\Bigg|_{x=\xi}\r]
            % lem:Y-diff
            \\&\lesssim
            %\l\|\eta\r\|_{\L^2}^4 + \E\l[ \int_t^s \alpha_1\!\l(u\r)^2 \E\l[ \l\| X^{t,\xi+\eta}_u - X^{t,\xi}_u - \D x X^{t,\xi,\xi}_u\eta - Y^{t,\xi,\eta}_u \r\|\r]^2 \d u\Bigg|_{x=\xi}\r]
            %\\&=
            \l\|\eta\r\|_{\L^2}^4 + \int_t^s \alpha_1\!\l(u\r)^2 \E\l[ \l\| X^{t,\xi+\eta}_u - X^{t,\xi}_u - \D x X^{t,\xi,\xi}_u\eta - Y^{t,\xi,\eta}_u \r\|\r]^2 \d u
            % lem:concatenatin-identity
            \\&\lesssim
            \l\|\eta\r\|_{\L^2}^4 + \int_t^s \alpha_1\!\l(u\r)^2 \E\l[ \l\| X^{t,\xi+\eta,\xi+\eta}_u - X^{t,\xi,\xi+\eta}_u - \D x X^{t,\xi,\xi}_u\eta\r\|\r]^2 \d u
            \\&\quad 
            + \int_t^s \alpha_1\!\l(u\r)^2 \E\l[ \l\| X^{t,\xi,\xi+\eta}_u - X^{t,\xi,\xi}_u - Y^{t,\xi,\eta}_u \r\|^2 \r] \d u
            \\&\lesssim
            \l\|\eta\r\|_{\L^2}^4 + \l\| X^{t,\xi+\eta,\xi+\eta} - X^{t,\xi,\xi+\eta} - \D x X^{t,\xi,\xi}\eta \r\|_{\H^1}^2
            \\&\quad 
            + \int_t^s \alpha_1\!\l(u\r)^2 \E\l[ \l\| X^{t,\xi,\xi+\eta}_u - X^{t,\xi}_u - Y^{t,\xi,\eta}_u \r\|^2 \r] \d u
            ,
        \end{flalign*}
        and Grönwall's inequality yields
        \begin{align*}
            \l\| X^{t,\xi,\xi+\eta} - X^{t,\xi,\xi} - Y^{t,\xi,\eta} \r\|_{\H^2}
            &\lesssim
            \l\|\eta\r\|_{\L^2}^2 + \l\| X^{t,\xi+\eta,\xi+\eta} - X^{t,\xi,\xi+\eta} - \D x X^{t,\xi,\xi}\eta \r\|_{\H^1}
            .\numberthis \label{ineq:xi-diff}
        \end{align*}
        Finally, observe that
        \begin{flalign*}
            \quad& 
            \l\| X^{t,\xi+\eta} - X^{t,\xi} - \D x X^{t,\xi,\xi}\eta - Y^{t,\xi,\eta} \r\|_{\H^{1}} 
            &&
            % triangle inequality
            %\\& \leq 
            %\l\| X^{t,\xi+\eta,\xi+\eta} - X^{t,\xi,\xi+\eta} - \D x X^{t,\xi,\xi}\eta \r\|_{\H^1} + \l\| X^{t,\xi,\xi+\eta} - X^{t,\xi,\xi} - Y^{t,\xi,\eta} \r\|_{\H^1}
            % H1-H2 inequality
            \\&\leq 
            \l\| X^{t,\xi+\eta,\xi+\eta} - X^{t,\xi,\xi+\eta} - \D x X^{t,\xi,\xi}\eta \r\|_{\H^1} + \l\| X^{t,\xi,\xi+\eta} - X^{t,\xi,\xi} - Y^{t,\xi,\eta} \r\|_{\H^2}
            % grönwall inequality
            \\&\lesssim
            \l\|\eta\r\|_{\L^2}^2 + \l\| X^{t,\xi+\eta,\xi+\eta} - X^{t,\xi,\xi+\eta} - \D x X^{t,\xi,\xi}\eta \r\|_{\H^1} 
        \end{flalign*}
        due to \eqref{ineq:xi-diff} and, thus, Lemma~\ref{lem:H1-diff} implies
        \begin{align*} 
            \lim_{\|\eta\|_{\L^2}\rightarrow 0} \frac{\l\| X^{t,\xi+\eta} - X^{t,\xi} - \D x X^{t,\xi,\xi}\eta - Y^{t,\xi,\eta} \r\|_{\H^{1}}}{\l\|\eta\r\|_{\L^2}}
            %&&
            %\\&\lesssim
            %\lim_{\|\eta\|_{\L^2}\rightarrow 0} \frac{\l\|\eta\r\|_{\L^2}^2 + \l\| X^{t,\xi+\eta,\xi+\eta} - X^{t,\xi,\xi+\eta} - \D x X^{t,\xi,\xi}\eta \r\|_{\H^1}}{\l\|\eta\r\|_{\L^2}} 
            = 0.
        \end{align*}
    \end{proof}

    \begin{prp}\label{prp:xi-differentiability-xi}
        Let $0\leq t\leq T$.
        If Assumptions~\ref{ass:1-lipschitz} and \ref{ass:2-first-derivative} are satisfied with $q_0\geq 4$, then the map
        \begin{equation*}
            \L^{2,d}(t)\rightarrow\H^{1,d}(t,T), \qquad \xi \mapsto X^{t,\xi}
        \end{equation*}
        is continuously Fréchet differentiable with Fréchet derivative 
        \begin{equation*}
            \D x X^{t,\xi}:\,
            \L^{2,d}(t) \rightarrow \H^{2,d}(t,T),
            \qquad
            \eta \mapsto \D \xi X^{t,\xi}\eta:=\D x X^{t,\xi,\xi}\eta + Y^{t,\xi,\eta}
        \end{equation*}
        at $\xi\in\L^{2,d}(t)$.
        %More precisely, for every $\xi\in\L^{2,d}(t)$, the map
        %\begin{equation*}
        %    \D \xi X^{t,\xi}: \, \L^{2,d}(t) \rightarrow \H^{1,d}(t,T), \qquad \eta \mapsto \D \xi X^{t,\xi}\eta:=\D x X^{t,\xi,\xi}\eta + Y^{t,\xi,\eta}
        %\end{equation*}
        %is linear and continuous and such that
        %\begin{equation*}
        %    \lim_{\|\eta\| \rightarrow 0} \frac{\l\| X^{t,x,\xi+\epsilon\eta} - X^{t,x,\xi} - \D \xi X^{t,\xi}\eta \r\|_{\H^{1}}}{\l\|\eta\r\|_{\L^2}} = 0,
        %\end{equation*}
        %where the limit is taken over $\eta\in\L^{2,d}(t)$.
    \end{prp}
    \begin{proof}
        Lemmas~\ref{lem:A-linear}, \ref{lem:Dx-concatenation}, \ref{lem:Y-SDE} and \ref{lem:Y-2-bound} imply that the map
        \begin{equation*}
            \L^{2,d}(t) \rightarrow \H^{1,d}(t,T), \qquad \eta \mapsto \D x X^{t,\xi,\xi}\eta + Y^{t,\xi,\eta}
        \end{equation*}
        is linear and continuous.
        
        Further, Lemma~\ref{lem:xi-differentiability} implies
        \begin{equation*}
            \lim_{\|\eta\| \rightarrow 0} \frac{\l\| X^{t,x,\xi+\epsilon\eta} - X^{t,x,\xi} - \D x X^{t,\xi,\xi}\eta - Y^{t,\xi,\eta} \r\|_{\H^{1}}}{\l\|\eta\r\|_{\L^2}} = 0.
        \end{equation*}
        
        Finally, observe that
        \begin{flalign*}
            \quad 
            &\l\| \D x X^{t,\xi+\eta,\xi+\eta} \zeta + Y^{t,\xi+\eta,\zeta} - \D x X^{t,\xi,\xi}\zeta - Y^{t,\xi,\zeta} \r\|_{\H^1} 
            &&
            \\&=
            \E\l[ \E\l[ \sup_{t\leq w \leq T}\l\| \D x X^{t,x+y,\xi+\eta}_w z - \D x X^{t,x,\xi}_w z + Y^{t,x+y,\xi+\eta,\zeta}_w - Y^{t,x,\xi,\zeta}_w \r\| \,\Big|\,\CF_t\r] \bigg|_{x=\xi,y=\eta,z=\zeta}\r]
            % trianlge, sublinearity, Jensen's inequality
            \\&\lesssim 
            \E\l[ \E\l[ \sup_{t\leq w \leq T}\l\| \D x X^{t,x+y,\xi+\eta}_w z - \D x X^{t,x,\xi}_w z\r\|^2 \,\Big|\,\CF_t\r]^{\frac12}  \bigg|_{x=\xi,y=\eta,z=\zeta}\r]
            \\&\quad 
            +\E\l[ \E\l[ \sup_{t\leq w \leq T}\l\| Y^{t,x+y,\xi+\eta,\zeta}_w - Y^{t,x,\xi,\zeta}_w \r\|^2 \,\Big|\,\CF_t\r]^{\frac12}  \bigg|_{x=\xi,y=\eta,z=\zeta}\r]
            % lem: Dx-Dx, Y-Y-p-bound
            %\\&\lesssim
            %\E\l[ \l\|\zeta\r\| \l(\l\|\eta\r\| + \l\|\eta\r\|_{\L^2}\r)\r]
            %+ \E\l[ \l\|\zeta\r\|_{\L^2} \l(\l\|\eta\r\| + \l\|\eta\r\|_{\L^2}\r)\r]
            % Cauchy-Schwarz
            \\&\lesssim
            \l\|\zeta\r\|_{\L^2} \l\|\eta\r\|_{\L^2}
        \end{flalign*}
        due to Lemmas~\ref{lem:Dx-Dx}, \ref{lem:Dx-concatenation}, \ref{lem:Y-concatenation} and \ref{lem:Y-Y-p-bound}. 
        Thus, $\xi \mapsto \D \xi X^{t,\xi}$ is continuous with respect to the operator norm.
    \end{proof}

    \begin{prp}\label{prp:xi-differentiability-x}
        Let $0\leq t\leq T$ and $x\in\Rbb^d$.
        If Assumptions~\ref{ass:1-lipschitz} and \ref{ass:2-first-derivative} are satisfied with $q_0\geq 4$, then the map
        \begin{equation*}
            \L^{2,d}(t)\rightarrow\H^{2,d}(t,T), \qquad \xi \mapsto X^{t,x,\xi}
        \end{equation*}
        is continuously Fréchet differentiable with Fréchet derivative 
        \begin{equation*}
            \D \xi X^{t,x,\xi}:\,
            \L^{2,d}(t)\rightarrow\H^{2,d}(t,T)
            , \qquad 
            \eta \mapsto \D \xi X^{t,x,\xi}\eta:=Y^{t,x,\xi,\eta}
        \end{equation*}
        at $\xi\in \L^{2,d}(t)$.
        %More precisely, for every $x\in\Rbb^d$ and $\xi\in\L^{2,d}(t)$, the map
        %\begin{equation*}
        %    \D \xi X^{t,x,\xi}:\, \L^{2,d}(t)\rightarrow \H^{2,d}(t,T),\qquad \eta \mapsto \D \xi X^{t,x,\xi}\eta:=Y^{t,x,\xi,\eta}
        %\end{equation*}
        %is linear and continuous and such that
        %\begin{equation*}
        %    \lim_{\|\eta\| \rightarrow 0} \frac{\l\| X^{t,x,\xi+\eta} - X^{t,x,\xi} - \D \xi X^{t,x,\xi}\eta \r\|_{\H^{2}}}{\l\|\eta\r\|_{\L^2}} = 0,
        %\end{equation*}
        %and the map
        %\begin{equation*}
        %    \L^{2,d}(t)\rightarrow L(\L^{2,d}(t),\H^{2,d}(t,T)), \qquad
        %    \xi \mapsto \D \xi X^{t,x,\xi}
        %\end{equation*}
        %is continuous with respect to the operator norm.
    \end{prp}
    \begin{proof}
        Lemmas~\ref{lem:Y-SDE} and \ref{lem:Y-2-bound} imply that the map
        \begin{equation*}
            \L^{2,d}(t)\rightarrow \H^{2,d}(t,T),\qquad \eta \mapsto Y^{t,x,\xi,\eta}
        \end{equation*}
        is linear and continuous.
        Moreover, we have
        \begin{equation*}
            \l\| X^{t,x,\xi+\eta} - X^{t,x,\xi} - Y^{t,x,\xi,\eta}\r\|_{\H^2}
            \lesssim
            \l\| \eta\r\|_{\L^2}^2 + \l\| X^{t,\xi+\eta} - X^{t,\xi} - \D x X^{t,\xi,\xi}\eta - Y^{t,\xi,\eta}\r\|_{\H^1}.
        \end{equation*}
        due to Lemma~\ref{lem:Y-diff} and, thus, Lemma~\ref{lem:xi-differentiability} yields
        \begin{equation*}
            \lim_{\|\eta\| \rightarrow 0} \frac{\l\| X^{t,x,\xi+\eta} - X^{t,x,\xi} - Y^{t,x,\xi,\eta} \r\|_{\H^{2}}}{\l\|\eta\r\|_{\L^2}} = 0.
        \end{equation*}

        Finally, observe that
        \begin{flalign*}
            \l\| Y^{t,x,\xi,\zeta} - Y^{t,x,\eta,\zeta}\r\|_{\H^2} \lesssim \l\|\zeta\r\|_{\L^2} \l\|\xi-\eta\r\|_{\L^2} 
        \end{flalign*}
        due to Lemma~\ref{lem:Y-Y-p-bound}. Thus, the map $\xi\mapsto \D \xi X^{t,x,\xi}$ is continuous with respect to the operator norm.
    \end{proof}
    
    \section{Second Order Derivatives}\label{sec:derivative-2}
    For a normed real vector space $V$, let $\C^{2}(V)$ denote the space of all $f\in\C^1(V)$ such that $\D v f\l(\cdot\r)v\in\C^1(V)$ for all $v\in V$ and, for convenience, we set $\D v^2 f(v_0)(v_2,v_1):=\D v \D v f\!\l(v_0\r)v_1\,v_2$ for $v_0,v_1,v_2\in V$.
    
    \begin{ass}\label{ass:3-second-derivative}
        Let $b:\, [0,T]\times\Omega\times\Rbb^d\times\L^{2,d}\rightarrow\Rbb^d$, $h:\, [0,T]\times\Omega\times\Rbb^d\times\L^{2,d}\rightarrow\Rbb^{d\times n\times n}$, and $g:\, [0,T]\times\Omega\times\Rbb^d\times\L^{2,d}\rightarrow\Rbb^{d\times n}$ be such that the following holds for all components $f=b_k,h_{kij},g_{ki}$ with $1\leq i,j\leq n$, $1\leq k\leq d$.
        \begin{enumerate}
            \item 
            We have 
            $ f\!\l(s,\omega,\cdot,\xi\r)\in \C^{2}(\Rbb^d)$,  
            $\D \xi f\!\l(s,\omega,\cdot,\xi\r)\eta \in \C^1(\Rbb^d)$ and
            $\D x f\!\l(s,\omega,x,\cdot\r)y \in \C^1(\L^{2,d})$
            for all $0\leq s\leq T$, $\omega\in\Omega$, $x,y\in\Rbb^d$ and $\xi,\eta\in\L^{2,d}$.
            
            \item 
            There exists a square-integrable $\alpha_2:\,[0,T]\rightarrow [1,\infty)$ such that
            \begin{align*}
                \l| \D x^2 f\!\l(s,\omega,x,\xi\r)\!\l(y,z\r) - \D x^2 f\!\l(s,\omega,v,\xi\r)\!\l(y,z\r) \r| & \leq \kappa\!\l(s\r) \l\|y\r\| \l\|z\r\| \l\| x-v\r\|
                \!,\\
                \l| \D x \D \xi f\!\l(s,\omega,x,\xi\r) \zeta\, z - \D x \D \xi f\!\l(s,\omega,y,\eta\r) \zeta\, z \r| 
                &\leq \alpha_2\!\l(s\r) \l\|z\r\| \l\|\zeta\r\|_{\L^2} \!\l( \l\|x-y\r\| + \l\|\xi-\eta\r\|_{\L^2}\r)
                \!,\\
                \l| \D \xi \D x f\!\l(s,\omega,x,\xi\r) z\,\zeta - \D \xi \D x f\!\l(s,\omega,y,\eta\r) z\,\zeta \r| 
                &\leq \alpha_2\!\l(s\r) \l\|z\r\| \l\|\zeta\r\|_{\L^2} \!\l( \l\|x-y\r\| + \l\|\xi-\eta\r\|_{\L^2}\r)
            \end{align*}
            for all $0\leq s\leq T$, $\omega\in\Omega$, $v,x,y,z\in\Rbb^d$ and $\xi,\eta,\zeta\in\L^{2,d}$.
        \end{enumerate}
    \end{ass}

    \begin{lem}\label{lem:C-SDE}
        Let $0\leq t\leq T$, $x\in\Rbb^d$ and $\xi\in\L^{2,d}(t)$.
        If Assumptions~\ref{ass:1-lipschitz}, \ref{ass:2-first-derivative} and \ref{ass:3-second-derivative} are satisfied with $q_0\geq 4$, then the $G$-SDE
        \begin{align*}
             \d C^{t,x,\xi,y,z}_s  
            &= \D x b\!\l(s,X^{t,x,\xi}_s, X^{t,\xi}_s \r) C^{t,x,\xi,y,z}_s  \d s
            \\&\quad 
            + \D x^2 b\!\l(s,X^{t,x,\xi}_s, X^{t,\xi}_s \r) \!\l( \!\D x X^{t,x,\xi}_sy, \D x X^{t,x,\xi}_s z \r) \d s 
            \\&\quad
            + \D x h\!\l(s,X^{t,x,\xi}_s, X^{t,\xi}_s \r) C^{t,x,\xi,y,z}_s  \d \l<B\r>_s
            \\&\quad
            + \D x^2 h\!\l(s,X^{t,x,\xi}_s, X^{t,\xi}_s \r)  \!\l(\! \D x X^{t,x,\xi}_sy, \D x X^{t,x,\xi}_s z \r) \d \l<B\r>_s
            \\&\quad
            + \D x g\!\l(s,X^{t,x,\xi}_s, X^{t,\xi}_s \r) C^{t,x,\xi,y,z}_s \d B_s
            \\&\quad 
            + \D x^2 g\!\l(s,X^{t,x,\xi}_s, X^{t,\xi}_s \r)  \!\l( \!\D x X^{t,x,\xi}_sy, \D x X^{t,x,\xi}_s z \r) \d B_s, 
            & t\leq s\leq T,
            \\
            C^{t,x,\xi,y,z}_t &= 0
            \numberthis\label{eq:Dxx-derivative-SDE}
        \end{align*}
        admits a unique solution $C^{t,x,\xi,y,z}\in\H^{2,d}(t,T)$ for all $0\leq t\leq T$, $x,y,z\in\Rbb^d$ and $\xi\in\L^{2,d}(t)$. 
        Moreover, the map
        \begin{equation*}
            \Rbb^d\times \Rbb^d \rightarrow \H^{2,d}(t,T),\qquad (y,z)\mapsto C^{t,x,\xi,y,z}
        \end{equation*}
        is bilinear.        
    \end{lem}
    \begin{proof}
        The SDE \eqref{eq:Dxx-derivative-SDE} has a unique solution $C^{t,x,\xi,y,z}\in\H^{2,d}(t,T)$ since the coefficients are Lipschitz and of linear growth due to Lemma~\ref{lem:Dx-p-bound} for any $y,z\in\Rbb^d$.
        Thus, the map $(y,z)\mapsto C^{t,x,\xi,y,z}$ is well-defined.

        Let $\lambda\in\Rbb$ and $v,x,y,z\in\Rbb^d$.
        By Lemma~\ref{lem:conditional}, we have for all $t\leq s\leq T$
        \begin{flalign*}
            \quad& 
            \E\l[ \sup_{t\leq w\leq s}\l\| C^{t,x,\xi,y+\lambda v,z}_w - C^{t,x,\xi,y,z}_w - \lambda C^{t,x,\xi,v,z}_w\r\|^2\r]
            &&
            % lem:conditional
            \\&\lesssim
            % \sum_{f\in F} \int_t^s \E\l[\l| \D x^2 f\!\l(u,X^{t,x,\xi}_u, X^{t,\xi}_u \r)  \!\l(\! \D x X^{t,x,\xi}_u (y+\lambda v) - \D x X^{t,x,\xi}_u y - \lambda \D x X^{t,x,\xi}_u v, \D x X^{t,x,\xi}_u z \r) \r|^2 \r] \d u
            %\\&\quad 
            %+ \sum_{f\in F} \int_t^s \E\l[ \l| \D x f\!\l(u,X^{t,x,\xi}_u, X^{t,\xi}_u \r)\!\l( C^{t,x,\xi,y+\lambda v,z}_u - C^{t,x,\xi,y,z}_u - \lambda C^{t,x,\xi,v,z}_u \r) \r|^2 \r] \d u
            % Dx X linear
            %\\&=
            \sum_{f\in F} \int_t^s \E\l[ \l| \D x f\!\l(u,X^{t,x,\xi}_u, X^{t,\xi}_u \r)\!\l( C^{t,x,\xi,y+\lambda v,z}_u - C^{t,x,\xi,y,z}_u - \lambda C^{t,x,\xi,v,z}_u \r) \r|^2 \r] \d u
            % Dx f bound
            \\&\lesssim
            \int_t^s \alpha_0\!\l(u\r)^2 \E\l[ \l\| C^{t,x,\xi,y+\lambda v,z}_u - C^{t,x,\xi,y,z}_u - \lambda C^{t,x,\xi,v,z}_u \r\|^2\r] \d u,
        \end{flalign*}
        and Grönwall's inequality implies
        \begin{equation*}
            \l\| C^{t,x,\xi,y+\lambda v,z} - C^{t,x,\xi,y,z} - \lambda C^{t,x,\xi,v,z}\r\|_{\H^2} =0,
        \end{equation*}
        i.e., $y\mapsto C^{t,x,\xi,y,z}$ is linear.
        Analogously, we obtain that $z\mapsto C^{t,x,\xi,y,z}$ is linear.
    \end{proof}

    \begin{lem}\label{lem:C-bound}
        If Assumptions~\ref{ass:1-lipschitz}, \ref{ass:2-first-derivative} and \ref{ass:3-second-derivative} are satisfied with $q_0\geq 4$, then
        \begin{equation*}
            \E\l[ \sup_{t\leq w\leq T}\l\| C^{t,x,\xi,y,z}_w \r\|^2\r] \lesssim \l\|y\r\|^2 \l\|z\r\|^2
        \end{equation*}
        for all $0\leq t\leq T$, $x,y,z\in\Rbb^d$ and $\xi\in\L^{2,d}(t)$.
    \end{lem}
    \begin{proof}
        By Lemma~\ref{lem:conditional}, we have for all $t\leq s\leq T$
        \begin{flalign*}
            \quad &
            \E\l[ \sup_{t\leq w\leq s}\l\| C^{t,x,\xi,y,z}_w \r\|^2\r]
            && \\
            &\lesssim
            \sum_{f\in F} \int_t^s \E\l[ \l| \D x^2 f\!\l(u,X^{t,x,\xi}_u, X^{t,\xi}_u \r)\!\l( \!\D xX^{t,x,\xi}_u y, \D xX^{t,x,\xi}_u z\r)\r|^2 \r] \d u
            \\&\quad +
            \sum_{f\in F} \int_t^s \E\l[\l| \D x f\!\l(u,X^{t,x,\xi}_u, X^{t,\xi}_u \r)C^{t,x,\xi,y,z}_u \r|^2 \r] \d u
            \\&\leq 
            \int_t^s \alpha_1\!\l(u\r)^2 \E\l[ \l\| \D xX^{t,x,\xi}_u y\r\|^2 \l\|\D xX^{t,x,\xi}_u z \r\|^2 \r] + \alpha_0\!\l(u\r)^2 \E\l[ \l\| C^{t,x,\xi,y,z}_u \r\|^2 \r] \d u
            % Cauchy-Schwarz
            %\\&\leq 
            %\int_t^s \alpha_1\!\l(u\r)^2 \E\l[ \l\| \D xX^{t,x,\xi}_u y\r\|^4 \r]^{\frac12} \E\l[ \l\|\D xX^{t,x,\xi}_u z \r\|^4 \r]^{\frac12} + \alpha_0\!\l(u\r)^2 \E\l[ \l\| C^{t,x,\xi,y,z}_u \r\|^2 \r] \d u
            % lem: Dx-p-bound
            \\&\lesssim
            \l\|y\r\|^2 \l\|z\r\|^2 + \int_t^s \alpha_0\!\l(u\r)^2 \E\l[ \l\| C^{t,x,\xi,y,z}_u \r\|^2 \r] \d u.
        \end{flalign*}
        Finally, Grönwall's inequality implies the desired result.
    \end{proof}

    \begin{prp}\label{prp:Dx2-differentiability}
        Let $0\leq t\leq T$.
        If Assumptions~\ref{ass:1-lipschitz}, \ref{ass:2-first-derivative} and \ref{ass:3-second-derivative} are satisfied with $q_0\geq 6 $ and $q_1\geq 3$, then the map
        \begin{equation*}
            \Rbb^d \rightarrow \H^{2,d}(t,T), \qquad x \mapsto X^{t,x,\xi}
        \end{equation*}
        is twice Fréchet differentiable for every $\xi\in\L^{2,d}(t)$.
        More precisely, for every $x\in\Rbb^d$ and $\xi\in\L^{2,d}(t)$, the map
        \begin{equation*}
            \D x^2 X^{t,x,\xi}:\, \Rbb^d\times \Rbb^d \rightarrow \H^{2,d}(t,T), \qquad (y,z)\mapsto \D x^2 X^{t,x,\xi}\!\l(y,z\r):=C^{t,x,\xi,y,z}
        \end{equation*}
        is bilinear and continuous and such that 
        \begin{equation*}
            \lim_{\|y\|\rightarrow 0} \frac{\l\| \D x X^{t,x+y,\xi}z - \D x X^{t,x,\xi}z - \D x ^2 X^{t,x,\xi} \!\l(y,z\r) \r\|_{\H^2}}{\l\|y\r\|} = 0
        \end{equation*}
        for all $z\in\Rbb^d$.
    \end{prp}
    \begin{proof}
        The map $(y,z)\mapsto C^{t,x,\xi,y,z}$ is bilinear and continuous due to Lemmas~\ref{lem:C-SDE} and \ref{lem:C-bound}.
        Set $\Delta^x:=X^{t,x+y,\xi} - X^{t,x,\xi}$, then
        \begin{align*}
            \E\l[ \sup_{t\leq w\leq T}\l\| \Delta^x_w \r\|^6\,\Big|\,\CF_t \r] &\lesssim \l\|y\r\|^6
            \numberthis\label{ineq:Delta1}
        \end{align*}
        due to Lemma~\ref{lem:xxi-yeta-p-bound}, and Lemma~\ref{lem:Dx-Dx} implies
        \begin{align*}
            \E\l[ \sup_{t\leq w\leq T} \l\|\Delta^x_w - \D x X^{t,x,\xi}_w y\r\|^3\,\Big|\,\CF_t\r]
            % continuously Fréchet differentiable
            &\leq 
            \int_0^1 \E\l[ \sup_{t\leq w\leq T} \l\|  \D x X^{t,x+\lambda y,\xi}_w y   - \D x X^{t,x,\xi}_w y\r\|^3\,\Big|\,\CF_t\r] \d \lambda
            \\&\lesssim
            \l\|y\r\|^6
            .\numberthis\label{ineq:Delta1-DX}
        \end{align*}
        Further, set $\Delta^{x,x}:=\D x X^{t,x+y,\xi}z - \D x X^{t,x,\xi}z$, then Lemma~\ref{lem:Dx-Dx} yields
        \begin{equation*}
            \E\l[ \sup_{t\leq w\leq T}\l\| \Delta^{x,x}_w\r\|^3 \r] \lesssim \l\|y\r\|^3.
            \numberthis\label{ineq:Delta2}
        \end{equation*}
        
        By Lemma~\ref{lem:conditional}, we have for all $t\leq s\leq T$
        \begin{flalign*}
            \quad 
            \E&\l[\sup_{t\leq w\leq s}\l\| \Delta^{x,x}_w - C^{t,x,\xi,y,z}_w\r\|^2 \r]
            &&
            % lem:conditional
            \\&\lesssim
            \sum_{f\in F} \int_t^s \E\bigg[\Big| \D x f\!\l(u,X^{t,x+y,\xi}_u,X^{t,\xi}_u\r) \D x X^{t,x+y,\xi}_uz - \D x f\!\l(u,X^{t,x,\xi}_u,X^{t,\xi}_u\r)\D x X^{t,x,\xi}_u z 
            \\&\quad 
            - \D x f\!\l(u,X^{t,x,\xi}_u,X^{t,\xi}_u\r) C^{t,x,\xi,y,z}_u - \D x ^2 f\!\l(u,X^{t,x,\xi}_u,X^{t,\xi}_u\r) \!\l(\!\D x X^{t,x,\xi}_u y, \D x X^{t,x,\xi}_u z\r) \Big|^2 \bigg] \d u
            \\&\lesssim
            %\sum_{f\in F} \int_t^s \E\l[ \l| \D x f\!\l(u,X^{t,x+y,\xi}_u,X^{t,\xi}_u\r)\Delta^{x,x}_u - \D x f\!\l(u,X^{t,x,\xi}_u,X^{t,\xi}_u\r)\Delta^{x,x}_u \r|^2 \r] \d u
            %\\&\quad +
            %\sum_{f\in F} \int_t^s \E\l[ \l| \D x ^2 f\!\l(u,X^{t,x,\xi}_u,X^{t,\xi}_u\r) \!\l(\Delta^x_u - \D x X^{t,x,\xi}y, \D x X^{t,x,\xi} z\r)\r|^2 \r] \d u
            %\\&\quad +
            %\sum_{f\in F} \int_t^s \E\l[ \l| \D x f\!\l(u,X^{t,x+y,\xi}_u,X^{t,\xi}_u\r) \D x X^{t,x,\xi}_uz - \D x f\!\l(u,X^{t,x,\xi}_u,X^{t,\xi}_u\r)\D x X^{t,x,\xi}_u z 
            %- \D x ^2 f\!\l(u,X^{t,x,\xi}_u,X^{t,\xi}_u\r) \!\l( \Delta^x_u,\D x X^{t,x,\xi}_u z\r) \r|^2 \r] \d u
            %\\&\quad +
            %\sum_{f\in F} \int_t^s \E\bigg[ \Big| \D x f\!\l(u,X^{t,x,\xi}_u,X^{t,\xi}_u\r) \!\l(\Delta^{x,x}_u - C^{t,x,\xi,y,z}_u\r)  \Big|^2 \bigg] \d u
            % continuously Fréchet differetniable
            %\\&=
            %\sum_{f\in F} \int_t^s \E\l[ \l| \int_0^1 \l( \D x^2 f\!\l(u,X^{t,x,\xi}_u+\lambda \Delta^x_u,X^{t,\xi}_u\r)
            %- \D x ^2 f\!\l(u,X^{t,x,\xi}_u,X^{t,\xi}_u\r) \r)\!\l( \Delta^x_u,\D x X^{t,x,\xi}_u z\r) \d \lambda \r|^2 \r] \d u
            %\\&\quad +
            %\sum_{f\in F} \int_t^s \E\l[ \l| \D x f\!\l(u,X^{t,x+y,\xi}_u,X^{t,\xi}_u\r)\Delta^{x,x}_u - \D x f\!\l(u,X^{t,x,\xi}_u,X^{t,\xi}_u\r)\Delta^{x,x}_u \r|^2 \r] \d u
            %\\&\quad +
            %\sum_{f\in F} \int_t^s \E\l[ \l| \D x ^2 f\!\l(u,X^{t,x,\xi}_u,X^{t,\xi}_u\r) \!\l(\Delta^x_u - \D x X^{t,x,\xi}y, \D x X^{t,x,\xi} z\r)\r|^2 \r] \d u
            %\\&\quad +
            %\sum_{f\in F} \int_t^s \E\bigg[ \Big| \D x f\!\l(u,X^{t,x,\xi}_u,X^{t,\xi}_u\r) \!\l(\Delta^{x,x}_u - C^{t,x,\xi,y,z}_u\r)  \Big|^2 \bigg] \d u
            % Jensen
            %\\&\leq 
            \sum_{f\in F} \int_t^s \int_0^1 \E\l[ \l|  \l( \D x^2 f\!\l(u,X^{t,x,\xi}_u+\lambda \Delta^x_u,X^{t,\xi}_u\r)
            - \D x ^2 f\!\l(u,X^{t,x,\xi}_u,X^{t,\xi}_u\r) \r)\!\l( \Delta^x_u,\D x X^{t,x,\xi}_u z\r) \r|^2 \r] \d \lambda \d u
            \\&\quad +
            \sum_{f\in F} \int_t^s \E\l[ \l| \D x f\!\l(u,X^{t,x+y,\xi}_u,X^{t,\xi}_u\r)\Delta^{x,x}_u - \D x f\!\l(u,X^{t,x,\xi}_u,X^{t,\xi}_u\r)\Delta^{x,x}_u \r|^2 \r] \d u
            \\&\quad +
            \sum_{f\in F} \int_t^s \E\l[ \l| \D x ^2 f\!\l(u,X^{t,x,\xi}_u,X^{t,\xi}_u\r) \!\l(\Delta^x_u - \D x X^{t,x,\xi}y, \D x X^{t,x,\xi} z\r)\r|^2 \r] \d u
            \\&\quad +
            \sum_{f\in F} \int_t^s \E\bigg[ \Big| \D x f\!\l(u,X^{t,x,\xi}_u,X^{t,\xi}_u\r) \!\l(\Delta^{x,x}_u - C^{t,x,\xi,y,z}_u\r)  \Big|^2 \bigg] \d u
            % ass:2, ass:3
            \\&\lesssim
            \int_t^s \alpha_2\!\l(u\r)^2 \E\l[ \l\| \Delta^x_u\r\|^4 \l\|\D x X^{t,x,\xi}_u z\r\|^2 \r] + \alpha_1\!\l(u\r)^2 \E\l[ \l\| \Delta^x_u \r\|^2 \l\|\Delta^{x,x}_u\r\|^2 \r] \d u
            \\&\quad +
            \int_t^s \alpha_1\!\l(u\r)^2 \E\l[ \l\| \Delta^x_u - \D x X^{t,x,\xi}_u y\r\|^2 \l\| \D x X^{t,x,\xi}_u z\r\|^2\r]
            \\&\quad + 
            \int_t^s \alpha_0\!\l(u\r)^2 \E\l[ \l\| \Delta^{x,x}_u - C^{t,x,\xi,y,z}_u \r\|^2 \r]
            % Hölder inequality
            \\&\leq 
            \int_t^s \alpha_2\!\l(u\r)^2 \E\l[ \l\| \Delta^x_u\r\|^6 \r]^{\frac23} \E\l[ \l\|\D x X^{t,x,\xi}_u z\r\|^6 \r]^{\frac13} + \alpha_1\!\l(u\r)^2 \E\l[ \l\| \Delta^x_u \r\|^6 \r]^{\frac13} \E\l[ \l\|\Delta^{x,x}_u\r\|^3 \r]^{\frac23} \d u
            \\&\quad 
            + \int_t^s \alpha_1\!\l(u\r)^2 \E\l[ \l\| \Delta^x_u - \D x X^{t,x,\xi}_u y\r\|^3\r]^{\frac23} \E\l[ \l\| \D x X^{t,x,\xi}_u z\r\|^6\r]^{\frac13}
            \\&\quad + 
            \int_t^s \alpha_0\!\l(u\r)^2 \E\l[ \l\| \Delta^{x,x}_u - C^{t,x,\xi,y,z}_u \r\|^2 \r]
            % ineq:Delta1, Delta1-DX, Delta-2, lem:Dx-p-bound
            %\\&\lesssim
            %\int_t^s \alpha_2\!\l(u\r)^2 \l\|y\r\|^4 \l\|z\r\|^2  + \alpha_1\!\l(u\r)^2 \l\|y\r\|^4 + \alpha_1\!\l(u\r)^2 \l\|y\r\|^4 \l\|z\r\|^2 \d u
            %\\&\quad + 
            %\int_t^s \alpha_0\!\l(u\r)^2 \E\l[ \l\| \Delta^{x,x}_u - C^{t,x,\xi,y,z}_u \r\|^2 \r]
            \\&\lesssim
            \l\|y\r\|^4 \l( 1+\l\|z\r\|^2\r) + 
            \int_t^s \alpha_0\!\l(u\r)^2 \E\l[ \l\| \Delta^{x,x}_u - C^{t,x,\xi,y,z}_u \r\|^2 \r]
        \end{flalign*}
        due to \eqref{ineq:Delta1}, \eqref{ineq:Delta1-DX}, \eqref{ineq:Delta2} and Lemma~\ref{lem:Dx-p-bound}.
        Finally, Grönwall's inequality yields
        \begin{equation*}
            \l\| \D x X^{t,x+y,\xi}z - \D x X^{t,x,\xi}z - \D x ^2 X^{t,x,\xi} \!\l(y,z\r) \r\|_{\H^2}\lesssim \l\|y\r\|^2 \l( 1+\l\|z\r\|\r)
        \end{equation*}
        which implies the desired result. 
    \end{proof}

    \begin{lem}\label{lem:D-SDE}
        If Assumption~\ref{ass:1-lipschitz}, \ref{ass:2-first-derivative} and \ref{ass:3-second-derivative} are satisfied with $q_0\geq 6$ and $q_1\geq 3$, then the $G$-SDE
        \begin{align*}
            \d D^{t,x,\xi,y,\eta}_s &= 
            \D x b\!\l(s,X^{t,x,\xi}_s,X^{t,\xi}_s\r) D^{t,x,\xi,y,\eta}_s \d s
            \\&\quad 
            + \D x^2 b\!\l(s,X^{t,x,\xi}_s,X^{t,\xi}_s\r) \!\l(\!\D x X^{t,x,\xi}_s y, \D \xi X^{t,x,\xi}_s \eta\r) \d s
            \\&\quad 
            + \D x \D \xi b\!\l(s,X^{t,x,\xi}_s,X^{t,\xi}_s\r) \!\D \xi X^{t,\xi}_s\eta \, \D x X^{t,x,\xi}_s y \d s
            \\&\quad +
            \D x h\!\l(s,X^{t,x,\xi}_s,X^{t,\xi}_s\r) D^{t,x,\xi,y,\eta}_s \d\l<B\r>_s
            \\&\quad 
            + \D x^2 h\!\l(s,X^{t,x,\xi}_s,X^{t,\xi}_s\r) \!\l(\!\D x X^{t,x,\xi}_s y, \D \xi X^{t,x,\xi}_s \eta\r) \d\l<B\r>_s
            \\&\quad 
            + \D x \D \xi h\!\l(s,X^{t,x,\xi}_s,X^{t,\xi}_s\r) \!\D \xi X^{t,\xi}_s\eta \, \D x X^{t,x,\xi}_s y \d \l<B\r>_s
            \\&\quad +
            \D x g\!\l(s,X^{t,x,\xi}_s,X^{t,\xi}_s\r) D^{t,x,\xi,y,\eta}_s \d B_s
            \\&\quad 
            + \D x^2 g\!\l(s,X^{t,x,\xi}_s,X^{t,\xi}_s\r) \!\l(\!\D x X^{t,x,\xi}_s y, \D \xi X^{t,x,\xi}_s \eta\r) \d B_s
            \\&\quad 
            + \D x \D \xi g\!\l(s,X^{t,x,\xi}_s,X^{t,\xi}_s\r) \!\D \xi X^{t,\xi}_s\eta \, \D x X^{t,x,\xi}_s y \d B_s, 
            & t\leq s\leq T, \\
            D^{t,x,\xi,y,\eta}_t &= 0
            \numberthis\label{eq:D-SDE}
        \end{align*}
        admits a unique solution $D^{t,x,\xi,y,\eta}\in\H^{2,d}(t,T)$ for all 
        $0\leq t\leq T$, $x,y\in\Rbb^d$, $\xi,\eta\in\L^{2,d}(t)$. 
        Moreover, the map
        \begin{equation*}
            \Rbb^d \times \L^{2,d}(t) \mapsto \H^{2,d}(t,T), \qquad
            (y,\eta)\mapsto D^{t,x,\xi,y,\eta}
        \end{equation*}
        is bilinear.
    \end{lem}
    \begin{proof}
        The SDE \eqref{eq:D-SDE} has a unique solution $D^{t,x,\xi,y,\eta}\in\H^{2,d}(t,T)$ since the coefficients are Lipschitz and of linear growth due to Lemmas~\ref{lem:Dx-p-bound} and \ref{lem:Y-p-bound} for any $y\in\Rbb^d$ and $\eta\in\L^{2,d}(t)$. Thus, the map $(y,\eta)\mapsto D^{t,x,\xi,y,\eta}$ is well defined. 

        Let $\lambda\in\Rbb$, $y,z\in\Rbb^d$ and $\eta,\zeta \in\L^{2,d}(t)$.
        By Lemma~\ref{lem:conditional}, we have for all $t\leq s\leq T$
        \begin{flalign*}
            \quad
            \E&\l[ \sup_{t\leq w\leq s}\l\| D^{t,x,\xi,y+\lambda z,\eta}_w - D^{t,x,\xi,y,\eta}_w - \lambda D^{t,x,\xi,z,\eta}_w \r\|^2 \r]
            &&
            % lem:conditional
            \\&\lesssim
            \sum_{f\in F} \int_t^s \E\l[ \l| \D x f\!\l(u,X^{t,x,\xi}_u,X^{t,\xi}_u\r) \!\l( D^{t,x,\xi,y+\lambda z,\eta}_u - D^{t,x,\xi,y,\eta}_u - \lambda D^{t,x,\xi,z,\eta}_u \r) \r|^2 \r] \d u
            % Dx f bounded
            \\&\leq 
            \int_t^s \alpha_0\!\l(u\r)^2 \E\l[ \l\| D^{t,x,\xi,y+\lambda z,\eta}_u - D^{t,x,\xi,y,\eta}_u - \lambda D^{t,x,\xi,z,\eta}_u  \r\|^2 \r] \d u
        \end{flalign*}
        and Grönwall's inequality yields that 
        \begin{equation*}
            \l\|D^{t,x,\xi,y+\lambda z} - D^{t,x,\xi,y} -\lambda D^{t,x,\xi,z}\r\|_{\H^2}=0,
        \end{equation*}
        i.e., $ y\mapsto D^{t,x,\xi,y,\eta}$ is linear.
        Analogously, we obtain that $\eta \mapsto D^{t,x,\xi,y,\eta}$ is linear.
    \end{proof}

    \begin{lem}\label{lem:D-bound}
        If Assumption~\ref{ass:1-lipschitz}, \ref{ass:2-first-derivative} and \ref{ass:3-second-derivative} are satisfied with $q_0\geq 6$ and $q_1\geq 3$, then
        \begin{equation*}
            \E\l[ \sup_{t\leq w\leq T}\l\| D^{t,x,\xi,y,\eta}_w \r\|^2 \r] \lesssim \l\|y\r\|^2 \l\|\eta \r\|_{\L^2}^2.
        \end{equation*}
    \end{lem}
    \begin{proof}
        By Lemma~\ref{lem:conditional}, we have for all $t\leq s\leq T$
        \begin{flalign*}
            \quad &
            \E\l[ \sup_{t\leq w\leq s}\l\| D^{t,x,\xi,y,\eta}_w \r\|^2 \r]
            &&\\
            % lem: conditional
            &\lesssim
            \sum_{f\in F} \int_t^s \E\l[ \l| \D x f\!\l(u,X^{t,x,\xi}_u,X^{t,\xi}_u\r) D^{t,x,\xi,y,\eta}_u \r|^2 \r] \d u
            \\&\quad +
            \sum_{f\in F} \int_t^s \E\l[ \l| \D x^2 f\!\l(u,X^{t,x,\xi}_u,X^{t,\xi}_u\r) \!\l( \! \D x X^{t,x,\xi}_u y, \D\xi X^{t,x,\xi}_u\eta \r) \r|^2 \r] \d u
            \\&\quad + 
            \sum_{f\in F} \int_t^s \E\l[ \l| \D x \D \xi f\!\l(u,X^{t,x,\xi}_u,X^{t,\xi}_u\r) \! \D\xi X^{t,x,\xi}_u\eta \, \D x X^{t,x,\xi}_u y \r|^2 \r] \d u
            \\&\lesssim
            \int_t^s \alpha_0\!\l(u\r)^2 \E\l[ \l\| D^{t,x,\xi,y,\eta}_u \r\|^2 \r]
            + \alpha_1\!\l(u\r)^2 \E\l[\l\| \D x X^{t,x,\xi}_u y\r\|^2 \l\| \D\xi X^{t,x,\xi}_u\eta \r\|^2 \r] \d u,
        \end{flalign*}
        and Grönwall's inequality implies
        \begin{equation*}
            \E\l[ \sup_{t\leq w\leq s}\l\| D^{t,x,\xi,y,\eta}_w \r\|^2 \r]
            \lesssim
            \int_t^s \alpha_1\!\l(u\r)^2 \E\l[\l\| \D x X^{t,x,\xi}_u y\r\|^2 \l\| \D\xi X^{t,x,\xi}_u\eta \r\|^2 \r] \d u
        \end{equation*}
        for all $t\leq s\leq T$. Finally, observe that for all $t\leq s\leq T$
        \begin{align*}
            \E\l[\l\| \D x X^{t,x,\xi}_s y\r\|^2 \l\| \D\xi X^{t,x,\xi}_s\eta \r\|^2 \r]
            &=
            \E\l[ \E\l[\l\| \D x X^{t,x,\xi}_s y\r\|^2 \l\| \D\xi X^{t,x,\xi}_s z \r\|^2 \,\Big|\, \CF_t \r] \Big|_{z=\zeta}\r]
            % Cauchy-Schwarz
            \\&\leq 
            \E\l[ \E\l[\l\| \D x X^{t,x,\xi}_s y\r\|^4 \,\Big|\, \CF_t \r]^{\frac12} \E\l[ \l\| \D\xi X^{t,x,\xi}_s z \r\|^4 \,\Big|\, \CF_t \r]^{\frac12} \Big|_{z=\zeta}\r]
            % lem: Dx-p-bound
            \\&\lesssim
            %\E\l[ \l\|y\r\|^2 \l\|\zeta\r\|^2 \r]
            %\\&= 
            \l\|y\r\|^2 \l\|\zeta\r\|_{\L^2}^2
        \end{align*}
        due to Lemma~\ref{lem:Dx-p-bound}, which implies the desired result.
    \end{proof}
    
    \begin{prp}
        Let $0\leq t\leq T$ and $\xi,\eta\in\L^{2,d}(t)$.
        If Assumption~\ref{ass:1-lipschitz}, \ref{ass:2-first-derivative} and \ref{ass:3-second-derivative} are satisfied with $q_0\geq 6$ and $q_1\geq 3$, then the map
        \begin{equation*}
            \Rbb^d \rightarrow \H^{2,d}(t,T),\qquad x \mapsto \D \xi X^{t,x,\xi} \eta
        \end{equation*}
        is Fréchet differentiable with Fréchet derivative
        \begin{equation*}
            \D x \D \xi X^{t,x,\xi} \eta:\,
            \Rbb^d \rightarrow \H^{2,d}(t,T)
            , \qquad 
            y \mapsto \D \xi X^{t,x,\xi} \eta\,y:= D^{t,x,\xi,y,\eta}
        \end{equation*}
        at $x\in \Rbb^d$.
    \end{prp}
    \begin{proof}
        By Lemmas~\ref{lem:D-SDE} and \ref{lem:D-bound}, the map $y\mapsto D^{t,x,\xi,y,\eta}$ is linear and continuous.

        Set $\Delta^\xi:=\D\xi X^{t,x+y,\xi}\eta - \D\xi X^{t,x,\xi}\eta$, then Lemma~\ref{lem:Y-Y-p-bound} yields
        \begin{align*}
            \E\l[ \sup_{t\leq w\leq T} \l\| \Delta^\xi_w \r\|^3\,\Big|\,\CF_t\r] \lesssim \l\| \eta\r\|_{\L^2}^3 \l\|y\r\|^3.
            \numberthis\label{ineq:Delta-xi}
        \end{align*}
        As in the proof of Proposition~\ref{prp:Dx2-differentiability}, set $\Delta^x:=X^{t,x+y,\xi} - X^{t,x,\xi}$, then
        \begin{align*}
            \E\l[ \sup_{t\leq w\leq T}\l\| \Delta^x_w \r\|^6\,\Big|\,\CF_t \r] &\lesssim \l\|y\r\|^6,
            &
            \E\l[ \sup_{t\leq w\leq T} \l\|\Delta^x_w - \D x X^{t,x,\xi}_w y\r\|^3\,\Big|\,\CF_t\r] & \lesssim \l\|y\r\|^6.
            \numberthis\label{ineq:Delta-x}
        \end{align*}
        
        By Lemma~\ref{lem:conditional}, we have for $t\leq s\leq T$
        \begin{flalign*}
            \quad 
            \E&\l[ \sup_{t\leq w\leq s}\l\| \Delta^\xi_w - D^{t,x,\xi,y,\eta}_w \r\|^2 \r]
            &&
            % lem:conditional
            \\&\lesssim
            \sum_{f\in F} \int_t^s \E\bigg[ \Big| \D x f\!\l(u,X^{t,x+y,\xi}_u,X^{t,\xi}_u\r) \D \xi X^{t,x+y,\xi}_u \eta 
            + \D \xi f\!\l(u,X^{t,x+y,\xi}_u,X^{t,\xi}_u\r) \D \xi X^{t,\xi}_u \eta
            \\&\quad 
            - \D x f\!\l(u,X^{t,x,\xi}_u,X^{t,\xi}_u\r) \D \xi X^{t,x,\xi}_u \eta 
            - \D \xi f\!\l(u,X^{t,x,\xi}_u,X^{t,\xi}_u\r) \D \xi X^{t,\xi}_u \eta 
            \\&\quad 
            - \D x f\!\l(u,X^{t,x,\xi}_u,X^{t,\xi}_u\r) D^{t,x,\xi,y,\eta}_u - \D x^2 f\!\l(u,X^{t,x,\xi}_u,X^{t,\xi}_u\r)\!\l(\! \D x X^{t,x,\xi}_u y, \D \xi X^{t,x,\xi}_u \eta \r)
            \\&\quad 
            - \D x \D \xi f\!\l(u,X^{t,x,\xi}_u,X^{t,\xi}_u\r) \D \xi X^{t,\xi}_u\eta \,\D x X^{t,x,\xi}_u y \Big|^2 \bigg]
            % Jensen
            \\&\lesssim
            \int_t^s \alpha_1\!\l(u\r)^2 \E\l[ \l\|\Delta^\xi_u\r\|^2 \l\| \Delta^x_u\r\|^2 \r]
            + \alpha_2\!\l(u\r)^2 \E\l[ \l\| \Delta^x_u \r\|^4 \l\| \D \xi X^{t,x,\xi}_u\eta \r\|^2 \r] \d u
            \\&\quad +
            \int_t^s \alpha_1\!\l(u\r)^2 \E\l[ \l\| \Delta^x_u - \D x X^{t,x,\xi}_u y \r\|^2 \l\| \D \xi X^{t,x,\xi}_u\eta \r\|^2 \r]
            +\alpha_2\!\l(u\r)^2 \E\l[ \l\| \Delta^x_u \r\|^4 \r] \l\| \D \xi X^{t,\xi}_u \eta \r\|_{\L^2}^2 \d u
            \\&\quad +
            \int_t^s \alpha_1\!\l(u\r)^2 \l\| \D \xi X^{t,\xi}_u\eta \r\|_{\L^2}^2 \E\l[ \l\| \Delta^x_u - \D x X^{t,x,\xi}_u \r\|^2 \r] + \alpha_0\!\l(u\r)^2 \E\l[ \l\| \Delta^\xi_u - D^{t,x,\xi,y,\eta}_u \r\|^2 \r] \d u
            % Hölder
            \\&\lesssim
            \int_t^s \alpha_1\!\l(u\r)^2 \E\l[ \l\|\Delta^\xi_u\r\|^3\r]^{\frac23} \E\l[ \l\| \Delta^x_u\r\|^6 \r]^{\frac13}
            + \alpha_2\!\l(u\r)^2 \E\l[ \E\l[ \l\| \Delta^x_u \r\|^6\,\Big|\,\CF_t\r]^{\frac23} \E\l[ \l\| \D \xi X^{t,x,\xi}_u\eta \r\|^6\,\Big|\,\CF_t\r]^{\frac13} \r] \d u
            \\&\quad +
            \int_t^s \alpha_1\!\l(u\r)^2 \E\l[ \E\l[ \l\| \Delta^x_u - \D x X^{t,x,\xi}_u y \r\|^3\,\Big|\,\CF_t\r]^{\frac23} \E\l[ \l\| \D \xi X^{t,x,\xi}_u\eta \r\|^6\,\Big|\,\CF_t\r]^{\frac13} \r] \d u
            \\&\quad +
            \int_t^s \alpha_2\!\l(u\r)^2 \E\l[ \l\| \Delta^x_u \r\|^4 \r] \l\| \D \xi X^{t,\xi}_u \eta \r\|_{\L^2}^2 \d u
            \\&\quad +
            \int_t^s \alpha_1\!\l(u\r)^2 \l\| \D \xi X^{t,\xi}_u\eta \r\|_{\L^2}^2 \E\l[ \l\| \Delta^x_u - \D x X^{t,x,\xi}_u \r\|^2 \r] + \alpha_0\!\l(u\r)^2 \E\l[ \l\| \Delta^\xi_u - D^{t,x,\xi,y,\eta}_u \r\|^2 \r] \d u
            % ineq:Delta-x, lem:Y-p-bound
            %\\&\lesssim
            %\int_t^s \alpha_1\!\l(u\r)^2 \l\|\eta\r\|_{\L^2}^2 \l\|y\r\|^4
            %+ \alpha_2\!\l(u\r)^2 \E\l[ \l\|y\r\|^4 \l( \l\|\eta\r\|^2 + \l\|\eta\r\|_{\L^2}^2\r)\r] \d u
            %\\&\quad +
            %\int_t^s \alpha_1\!\l(u\r)^2 \E\l[\l\|y\r\|^4 \l( \l\|\eta\r\|^2 + \l\|\eta\r\|_{\L^2}^2\r)\r] + \alpha_2\!\l(u\r)^2 \l\|y\r\|^4 \l\| \eta \r\|_{\L^2}^2 \d u
            %\\&\quad +
            %\int_t^s \alpha_1\!\l(u\r)^2 \l\|y\r\|^4 \l\| \eta \r\|_{\L^2}^2 + \alpha_0\!\l(u\r)^2 \E\l[ \l\| \Delta^\xi_u - D^{t,x,\xi,y,\eta}_u \r\|^2 \r] \d u
            \\&\lesssim
            \l\|y\r\|^4 \l\| \eta \r\|_{\L^2}^2 + \int_t^s \alpha_0\!\l(u\r)^2 \E\l[ \l\| \Delta^\xi_u - D^{t,x,\xi,y,\eta}_u \r\|^2 \r] \d u
        \end{flalign*}
        due to \eqref{ineq:Delta-xi}, \eqref{ineq:Delta-x} and Lemma~\ref{lem:Y-p-bound}. Finally, Grönwall's inequality yields
        \begin{equation*}
            \l\| \D \xi X^{t,x+y,\xi}\eta - \D \xi X^{t,x,\xi}\eta - D^{t,x,\xi,y,\eta} \r\|_{\H^2}
            \lesssim \l\|y\r\|^2 \l\| \eta \r\|_{\L^2}
        \end{equation*}
        which implies the desired result.
    \end{proof}

    \begin{lem}\label{lem:interchange-diff}
        If Assumption~\ref{ass:1-lipschitz}, \ref{ass:2-first-derivative} and \ref{ass:3-second-derivative} are satisfied, then the following holds for all components $f=b_k,h_{kij},g_{ki}$, $1\leq i,j\leq n$, $1\leq k\leq d$,
        \begin{equation*}
            \D x \l[ \D \xi f\!\l(s,x,\xi\r) \eta \r] y = \D \xi \l[ \D x  f\!\l(s,x,\xi\r) y \r] \eta
        \end{equation*}
        for all $0\leq s\leq T$, $x,y\in\Rbb^d$, $\xi,\eta\in\L^{2,d}$ and $\omega\in\Omega$.
    \end{lem}
    \begin{proof}
        Let $0\leq s\leq T$, $x,y\in\Rbb^d$, $\xi,\eta\in\L^{2,d}$ and $\omega\in\Omega$.
        We have
        \begin{align*}
            I&:=f\!\l(s,x+y,\xi+\eta,\omega\r) - f\!\l(s,x+y,\xi,\omega\r) - f\!\l(s,x,\xi+\eta,\omega\r) + f\!\l(s,x,\xi,\omega\r)
            % continuous differentiability
            \\&=
            \int_0^1 \D \xi f\!\l(s,x+y,\xi+\lambda_1 \eta,\omega\r) \eta - \D \xi f\!\l(s,x,\xi+\lambda_1 \eta,\omega\r) \eta \d \lambda_1
            % continuous differentiability
            \\&=
            \int_0^1 \int_0^1 \D x \D \xi f\!\l(s,x+\lambda_2 y,\xi+\lambda_1 \eta,\omega\r) \eta \, y \d \lambda_2 \d \lambda_1
            \\&=
            \D x \D \xi f\!\l(s,x,\xi,\omega\r) \eta \, y + R_1
        \end{align*}
        with 
        \begin{flalign*}
            \quad
            R_1 & := \int_0^1 \int_0^1 \D x \D \xi f\!\l(s,x+\lambda_2 y,\xi+\lambda_1 \eta,\omega\r) \eta \, y - \D x \D \xi f\!\l(s,x,\xi,\omega\r) \eta \, y \d \lambda_2 \d \lambda_1 
            && 
            \\&\leq 
            \int_0^1 \int_0^1 \l| \D x \D \xi f\!\l(s,x+\lambda_2 y,\xi+\lambda_1 \eta,\omega\r) \eta \, y - \D x \D \xi f\!\l(s,x,\xi,\omega\r) \eta \, y \r|
            \d \lambda_2 \d \lambda_1 
            % Lipschitz DxDxi f
            \\&\leq
            \int_0^1 \int_0^1 \alpha_2\!\l(s\r) \l\|y\r\| \l\|\eta\r\|_{\L^2} \l( \lambda_2 \l\|y\r\| + \lambda_1 \l\|\eta\r\|_{\L^2} \r)  \d \lambda_2 \d \lambda_1
            % lambda bounds
            \\&\leq
            \l\|y\r\| \l\|\eta\r\|_{\L^2} \l( \l\|y\r\| + \l\|\eta\r\|_{\L^2} \r).
        \end{flalign*}
        Analogously, we have
        \begin{align*}
            \quad
            I&=f\!\l(s,x+y,\xi+\eta,\omega\r) - f\!\l(s,x,\xi+\eta,\omega\r) - f\!\l(s,x+y,\xi,\omega\r) + f\!\l(s,x,\xi,\omega\r)
            &&
            % continuous differentiability
            \\&=
            \int_0^1 \D x f\!\l(s,x+\lambda_1 y,\xi+\eta,\omega\r) y - \D x f\!\l(s,x+\lambda_1 y,\xi,\omega\r) y \d \lambda_1
            % continuous differentiability
            \\&=
            \int_0^1 \int_0^1 \D \xi \D x f\!\l(s,x+\lambda_1 y,\xi+\lambda_2 \eta,\omega\r) y \, \eta \d \lambda_2 \d \lambda_1
            \\&=
            \D x \D \xi f\!\l(s,x,\xi,\omega\r) \eta \, y + R_2
        \end{align*}
        with
        \begin{align*}
            R_2 
            & := \int_0^1 \int_0^1 \D \xi \D x f\!\l(s,x+\lambda_1 y,\xi+\lambda_2 \eta,\omega\r) \eta \, y - \D x \D \xi f\!\l(s,x,\xi,\omega\r) \eta \, y \d \lambda_2 \d \lambda_1 
            \\&\leq 
            \l\|y\r\| \l\|\eta\r\|_{\L^2} \l( \l\|y\r\| + \l\|\eta\r\|_{\L^2} \r).
        \end{align*}
        Thus, we get
        \begin{equation*}
            \frac{\l| \D x \D \xi f\!\l(s,x,\xi,\omega\r) \eta \, y - \D x \D \xi f\!\l(s,x,\xi,\omega\r) \eta \, y\r| }{\l\|y\r\| \l\|\eta\r\|_{\L^2}} \lesssim \l\|y\r\| + \l\|\eta\r\|_{\L^2}
        \end{equation*}
        for all $0\leq s\leq T$, $\omega\in \Omega$, $x,y\in \Rbb^d$ and $\xi,\eta\in\L^{2,d}(t)$.
        By letting $\|\eta\|_{\L^2}$ and $\|y\|$ tend to zero, we conclude the desired result.
    \end{proof}

    \begin{prp}
        Let $0\leq t\leq T$, $x\in\Rbb^d$ and $\xi\in\L^{2,d}(t)$. 
        If Assumption~\ref{ass:1-lipschitz}, \ref{ass:2-first-derivative} and \ref{ass:3-second-derivative} are satisfied with $q_0\geq 6$ and $q_1\geq 3$, then the map
        \begin{equation*}
            \L^{2,d}(t) \rightarrow \H^{2,d}(t,T),\qquad \xi \mapsto \D x X^{t,x,\xi} y
        \end{equation*}
        is Fréchet differentiable with Fréchet derivative
        \begin{equation*}
            \D\xi \D x X^{t,x,\xi} y:\,
            \L^{2,d}(t) \rightarrow \H^{2,d}(t,T), \qquad \eta \mapsto \D\xi \D x X^{t,x,\xi} y\,\eta := D^{t,x,\xi,y,\eta}
        \end{equation*}
        at $\xi\in\L^{2,d}(t)$.
    \end{prp}
    \begin{proof}
        By Lemmas~\ref{lem:D-SDE} and \ref{lem:D-bound}, the map $\eta \mapsto D^{t,x,\xi,y,\eta}$ is linear and continuous.

        For all components $f=b_k,h_{ijk}, g_{ik}$, $1\leq k\leq d$, $1\leq i,j\leq n$, we have
        \begin{align*}
            \D x \D \xi f\!\l(s,\omega,x,\xi\r) \eta \, y &= \D \xi \D x  f\!\l(s,\omega,x,\xi\r) y \, \eta, 
            \\
            \D x^2 f\!\l(s,\omega,x,\xi\r)\!\l(y,z\r) & = \D x^2 f\!\l(s,\omega,x,\xi\r)\!\l(z,y\r)
        \end{align*}
        for all $0\leq s\leq T$, $\omega\in\Omega$, $x,y,z\in\Rbb^d$ and $\xi,\eta\in\L^{2,d}$ due to Lemma~\ref{lem:interchange-diff} and the symmetry of the second order Fréchet derivative.

        Set 
        \begin{align*}
            \Delta&:= X^{t,\xi+\eta} - X^{t,\xi} \\
            \Delta^\xi &:= X^{t,x,\xi+\eta} - X^{t,x,\xi} \\
            \Delta^{x,\xi} &:= \D x X^{t,x,\xi+\eta} y - \D x X^{t,x,\xi} y
        \end{align*}
        From Lemmas~\ref{lem:xi-eta-2-bound}, \ref{lem:xxi-yeta-p-bound} and \ref{lem:Dx-Dx}, we obtain
        \begin{align*}
            \E\l[ \sup_{t\leq w\leq T} \l\|\Delta_w\r\|^2 \r]
            % xi-eta-2-bound
            &\lesssim \l\| \eta \r\|_{\L^2}^2 
            \\
            \E\l[ \sup_{t\leq w\leq T} \l\| \Delta^\xi_w \r\|^6 \r] 
            % xxi-yeta-p-bound
            & \lesssim \l\|\eta\r\|_{\L^2}^6, 
            \\
            \E\l[ \sup_{t\leq w\leq T} \l\| \Delta^{x,\xi}_w \r\|^3 \r] 
            % Dx-Dx-p-bound
            & \lesssim \l\|y\r\|^3 \l\|\eta\r\|_{\L^2}^3.
        \end{align*}
        Moreover, Lemma~\ref{lem:Y-Y-p-bound} yields
        \begin{align*}
            \E\l[ \sup_{t\leq w\leq T} \l\| \Delta^\xi_w - \D \xi X^{t,x,\xi}_w\eta \r\|^3\r]
            % continuously differentiable
            &\leq 
            \int_0^1 \E\l[ \sup_{t\leq w\leq T} \l\| \D \xi X^{t,x,\xi+\lambda\eta }_w \eta - \D \xi X^{t,x,\xi}_w \eta \r\|^3 \r] \d \lambda
            % Y-Y-p-bound
            \\&\lesssim
            \l\|\eta\r\|_{\L^2}^6,
        \end{align*}
        and we have
        \begin{flalign*}
            %\quad &
            \E\l[ \sup_{t\leq w\leq T}\l\| \Delta_w - \D \xi X^{t,\xi}_w\eta \r\| \r]
            %&&\\
            % continuously differentiable
            &\leq 
            \int_0^1 \l\| \D\xi X^{t,\xi+\lambda\eta }\eta - \D \xi X^{t,\xi}\eta \r\|_{\H^1} \d \lambda 
            \\&\leq 
            \int_0^1 \l\| \D x X^{t,\xi+\lambda \eta,\xi+\lambda \eta}\eta - \D x X^{t,\xi,\xi}\eta \r\|_{\H^1} \d \lambda 
            \\&\quad 
            + \int_0^1 \l\| Y^{t,\xi+\lambda \eta,\eta} - Y^{t,\xi,\eta} \r\|_{\H^1} \d \lambda 
            \\&\lesssim
            \l\|\eta\r\|_{\L^2}^2
        \end{flalign*}
        due to Corollaries~\ref{cor:Dx-Dx-L1} and \ref{cor:Y-Y-2-bound}.
        
        By Lemma~\ref{lem:conditional}, we have for all $t\leq s\leq T$
        \begin{flalign*}
            \quad 
            \E&\l[ \sup_{t\leq w \leq s}\l\| \Delta^{x,\xi}_w - D^{t,x,\xi,y,\eta}_w \r\|^2\r]
            &&
            % lem:conditional
            \\&\lesssim
            \sum_{f\in F} \int_t^s \E\bigg[ \Big| \D x f\!\l(u,X^{t,x,\xi+\eta}_u,X^{t,\xi+\eta}_u\r) \D x X^{t,x,\xi+\eta}_u y - \D x f\!\l(u,X^{t,x,\xi}_u,X^{t,\xi}_u\r) \D x X^{t,x,\xi}_u y
            \\&\quad 
            - \D x f\!\l(u,X^{t,x,\xi}_u,X^{t,\xi}_u\r) D^{t,x,\xi,y,\eta}_u 
            - \D x^2 f\!\l(u,X^{t,x,\xi}_u,X^{t,\xi}_u\r)\!\l(\!\D \xi X^{t,x,\xi}_u\eta, \D x X^{t,x,\xi}_u y \r) 
            \\&\quad 
            - \D \xi \D x f\!\l(u,X^{t,x,\xi}_u,X^{t,\xi}_u\r) \D x X^{t,x,\xi}_u y \, \D \xi X^{t,\xi}_u \eta 
            \Big|^2 \bigg] \d u
            % Jensen
            \\&\lesssim
            \int_t^s \alpha_1\!\l(u\r)^2 \E\l[ \l\|\Delta^\xi_u\r\|^4 \l\| \D x X^{t,x,\xi}_uy \r\|^2 \r] \d u
            \\&\quad +
            \int_t^s \alpha_2\!\l(u\r)^2 \l\| \Delta_u\r\|_{\L^2}^2 \l( \l\|\Delta_u\r\|_{\L^2}^2 \E\l[ \l\| \D x X^{t,x,\xi}_u y \r\|^2 \r]  + \E\l[ \l\| \D x X^{t,x,\xi}_u y \r\|^2 \l\| \Delta^\xi_u\r\|^2 \r] \r) \d u
            \\&\quad +
            \int_t^s \alpha_1 \!\l(u\r)^2 \E\l[ \l\|\Delta^\xi_u - \D \xi X^{t,x,\xi}_u\eta\r\|^2 \l\| \D x X^{t,x,\xi}_u y \r\|^2 \r] \d u
            \\&\quad + 
            \int_t^s \alpha_1\!\l(u\r)^2 \l\|\Delta_u - \D \xi X^{t,\xi}_u \eta\r\|_{\L^1}^2 \E\l[ \l\| \D x X^{t,x,\xi}_u y \r\|^2 \r] \d u
            \\&\quad +
            \int_t^s \alpha_1\!\l(u\r)^2 \l( \E\l[ \l\|\Delta^{x,\xi}_u\r\|^2 \l\| \Delta^\xi_u \r\|^2  \r] + \l\| \Delta_u\r\|_{\L^2}^2 \E\l[ \l\| \Delta^{x,\xi}_u \r\|^2 \r]\r) \d u
            \\&\quad +
            \int_t^s \alpha_0\!\l(u\r)^2 \E\l[ \l\| \Delta^{x,\xi}_u - D^{t,x,\xi,y,\eta}_u \r\|^2 \r] \d u
            \\&\lesssim
            \l\|\eta\r\|_{\L^2}^4 \l\|y\r\|^2
            +
            \int_t^s \alpha_0\!\l(u\r)^2 \E\l[ \l\| \Delta^{x,\xi}_u - D^{t,x,\xi,y,\eta}_u \r\|^2 \r] \d u.
        \end{flalign*}
        Finally, Grönwall's inequality yields the desired result.
    \end{proof}

    \section{Discussion on Space of Sublinear Distributions}\label{sec:comparison}
    In \cite{sun_distribution_2023}, the authors consider coefficients that depend on the sublinear distribution of the solution process,  where the sublinear distribution of a random variable $\xi$ is defined as the mapping $\phi \mapsto \E\l[ \phi(\xi)\r]$.    
    More precisely, they introduce the set $\mathcal D$ consisting of all functionals $F:\,\textnormal{Lip}(\Rbb^d)\rightarrow \Rbb$ which satisfy the following properties.
    \begin{enumerate}
        \item \emph{Constant-Preservation}: For all $\phi \in \textnormal{Lip}(\Rbb^d)$ with $\phi\equiv c\in \Rbb$, we have $F(\phi)=c$.
        \item \emph{Monotonicity}: For all $\phi,\psi \in \textnormal{Lip}(\Rbb^d)$ with $\phi \geq \psi$ everywhere, we have $F(\phi) \geq F(\psi)$.
        \item \emph{Positive Homogeneity}: For all $c \geq 0$ and $\phi \in \textnormal{Lip}(\Rbb^d)$, we have $F(c \phi)=c F(\phi)$.
        \item \emph{Subadditivity}: For all $\phi,\psi \in \textnormal{Lip}(\Rbb^d)$, we have $F(\phi + \psi) \leq F(\phi) + F(\psi)$.
        \item \emph{Boundedness}: We have
        \begin{equation*}
            \sup_{\phi \in \textnormal{Lip}_1(\Rbb^d)} \l| F(\phi) - \phi(0) \r| < \infty.
        \end{equation*}
    \end{enumerate}
    Here, $\textnormal{Lip}(\Rbb^d)$ denotes the space of all Lipschitz functions $\phi:\, \Rbb^n \rightarrow \Rbb$ and $\textnormal{Lip}_1(\Rbb^d)\subseteq \textnormal{Lip}(\Rbb^d)$ the subspace of functions with Lipschitz constant smaller or equal to 1.
    Further, the authors define the metric
    \begin{equation*}
        d:\, \mathcal D\times \mathcal D \rightarrow \Rbb, \qquad 
        (F,G)\mapsto d(F,G):=\sup_{\phi \in \textnormal{Lip}_1(\Rbb^d)} \l| F(\phi) - G(\phi) \r|
    \end{equation*}
    and consider a $G$-SDE of the form
    \begin{align*}
        \d X_t &=  b\!\l(t,X_t,F_{X_t}\r) \d t + h\!\l(t,X_t,F_{X_t}\r) \d \l<B\r>_t +  g\!\l(t,X_t,F_{X_t}\r) \d B_t, & 0\leq t\leq T, \\
        X_0 &= x, \numberthis\label{eq:sun-sde}
    \end{align*}
    where $x\in \Rbb^d$ and the coefficients $ b$, $ g$ and $ h$ are defined on $[0,T] \times \Rbb^d \times \mathcal D$ and, for $\xi \in \L^{1,d}$, the functional $F_\xi:\, \textnormal{Lip}(\Rbb^d) \rightarrow \Rbb$ is defined by $\phi \mapsto \E\l[ \phi(\xi)\r]$. 
    Clearly, for any $X$ that satisfies \eqref{eq:sun-sde}, we have $X\in \H^{1,d}(t,T)$ and, in particular, $F_{X_t}\in \mathcal D$ for all $0\leq t\leq T$, cf. also Remark~3.2 in \cite{sun_distribution_2023}.

    The authors show that \eqref{eq:sun-sde} admits a unique solution $X\in \M^{2,d}(0,T)$ for any initial value $x\in \Rbb^d$ when the coefficients satisfy the following assumption, cf. Theorem~4.1 in \cite{sun_distribution_2023}.
    \begin{ass}\label{ass:sun}
        Let $ b:\, [0,T]\times\Rbb^d\times \mathcal D\rightarrow\Rbb^d$, $ h:\, [0,T]\times\Rbb^d\times \mathcal D\rightarrow\Rbb^{d\times n\times n}$, and $ g:\, [0,T]\times\Rbb^d\times \mathcal D\rightarrow\Rbb^{d\times n}$ be such that the following holds for all components $f=\tilde b_k, h_{kij},\tilde g_{ki}$, $1\leq i,j\leq n$, $1\leq k\leq d$.

        \begin{enumerate}
            \item We have $f\!\l(\cdot, x, F\r)\in \M^2(0,T)$ for all $x\in \Rbb^d$ and $F \in \mathcal D$.
            \item There exist a constant $K>0$ such that
            \begin{equation*}
                \l| f\!\l(t,x,F\r) - f\!\l(t,y,G\r)\r| \leq K \l( \l\|x-y\r\|^2 + d\!\l(F,G\r) \r).
            \end{equation*}
        \end{enumerate}
    \end{ass}
    We can embed the formulation from \cite{sun_distribution_2023} into our setting by defining coefficient $\hat b$, $\hat g$ and $\hat h$ on $[0,T]\times \Omega \times \Rbb^d \times \L^{2,d}$ componentwise by
    \begin{align*}
        \hat b_k\!\l(s,\omega,x,\xi\r)&:=  b_k\!\l(s,x,F_{\xi}\r), & 
        \hat h_{kij}\!\l(s,\omega,x,\xi\r)&:=  h_{kij}\!\l(s,x,F_{\xi}\r), & 
        \hat g_{ki}\!\l(s,\omega,x,\xi\r)&:= g_{ki}\!\l(s,x,F_{\xi}\r).
    \end{align*}
    Note that in contrast to the general formulation in \cite{bollweg_mean-field_2025}, the coefficients $\hat b$, $\hat h$ and $\hat g$ are deterministic. 
    Moreover, for the components $\hat f=\hat b_k,\hat h_{kij},\hat g_{ki}$, $1\leq i,j\leq n$, $1\leq k\leq d$, Assumption~\ref{ass:sun} yields
    \begin{align*}
        \l| \hat f(t,\omega, x,\xi)-\hat f(t,\omega,y,\eta) \r|
        &\leq 
        K\l(\l\|x-y\r\| + d(F_\xi,F_\eta)\r)
        \\ &\leq 
        K\l(\l\|x-y\r\| + \l\|\xi-\eta\r\|_{\L^2} \r)
    \end{align*}
    for all $\omega\in\Omega$, $0\leq s\leq T$, $x,y\in\Rbb^d$ and $\xi,\eta\in\L^{2,d}$
    since
    \begin{align*}
        d\l(F_\xi,F_\eta)\r) &= \sup_{\phi \in \textnormal{Lip}_1(\Rbb^d)} \l| \E\l[ \phi(\xi) \r] - \E\l[ \phi(\eta)\r] \r| \leq \E\l[ \l\| \xi-\eta\r\| \r] = \l\|\xi-\eta\r\|_{\L^1} \leq \l\|\xi-\eta\r\|_{\L^2}.
    \end{align*}
    Further, we have $\hat f\!\l(\cdot,x,\xi\r) \ind_{[s,T]} \in \M^2(t,T)$ for all $x\in \Rbb^d$ and $\xi\in \F^{d}(\CF_s)$, $0\leq s\leq T$.
    That is, if the coefficients $b$, $h$ and $g$ satisfy Assumption~\ref{ass:sun}, then the coefficients $\hat b$, $\hat h$ and $\hat g$ satisfy Assumption~\ref{ass:1-lipschitz}. 
    In particular, Theorem~3.12 in \cite{bollweg_mean-field_2025} implies Theorem~4.1 in \cite{sun_distribution_2023}.
    
    Note that $\mathcal D$ is not a vector space and, thus, we need to consider a different notion of differentiability for functions defined on $\mathcal D$.    
    In classical mean-field theory, we encounter a similar issue when considering functions defined on the space of square-integrable distributions $\mathfrak P_2(\Rbb^d)$. 
    By lifting a function $f:\,\mathfrak P_2(\Rbb^d)\rightarrow \Rbb$ to a function $\hat f:\, L^2(\Rbb^d,\Omega,\CF,P)\rightarrow \Rbb$ and considering the Fréchet derivative of the lifted function $\hat f$, Lions developed a useful notion of derivative which is commonly referred to as Lions derivative, see e.g. \cite{cardaliaguet_mean_2018} for more details.
    In the same manner, we might want to lift a function $f:\,\mathcal D\rightarrow \Rbb$ to a function $\hat f:\,\L^{2,d}\rightarrow \Rbb$ such that $\hat f(\xi)=f(F_\xi)$ for all $\xi \in\L^{2,d}$, but it is not immediately clear whether the space $\L^{2,d}$ is rich enough in the sense that
    \begin{equation*}
        \mathcal D =\l\{ F_\xi:\, \textnormal{Lip}(\Rbb^d)\rightarrow \Rbb,\; \phi\mapsto \E\l[\phi(\xi)\r]\,:\,\xi \in \L^{2,d}\r\}=:\mathcal D_0.
    \end{equation*}    
    However, it is sufficient to consider the restriction of the coefficients $b$, $h$ and $g$ in \eqref{eq:sun-sde} to $[0,T]\times \Rbb^d\times \mathcal D_0$ so that $\hat b$, $\hat h$ and $\hat g$ are the respective liftings defined on $[0,T]\times \Rbb^d\times \L^{2,d}$ so that we can define a notion of differentiability for $b$, $h$ and $g$ in terms of the Gateaux or Fréchet derivatives of $\hat b$, $\hat h$ and $\hat g$ respectively.

    In the following, we develop a notion of differentiability for a map $f:\,\mathcal D_0\rightarrow \Rbb$ in terms of the derivative of its lifting $\hat f$. In particular, we need to ensure that the derivative $\partial f$ is such that $\partial f(F_\xi)=\partial f(F_\eta)$ for all $\xi,\eta\in\L^{2,d}$ with $F_\xi=F_\eta$.

    \begin{lem}\label{lem:derivative-0}
        Let $f:\,\mathcal D_0\rightarrow \Rbb$ be such that its lifting $\hat f:\,\L^{2,d}\rightarrow \Rbb$ is Gateaux differentiable at $\xi\in\L^{2,d}$. If $\eta \in \L^{2,d}$ is such that $F_\xi=F_\eta$, then $\hat f$ is Gateaux differentiable at $\eta$ and 
        \begin{align*}
            \partial \hat f\!\l(\xi;\zeta\r) = \partial \hat f\!\l(\eta;\zeta\r)
        \end{align*}
        for all $\zeta\in\L^{2,d}$ such that $\xi$ and $\eta$ are independent of $\zeta$, where $\partial \hat f(\xi;\zeta)$ denotes the Gateaux derivative of $\hat f$ at $\xi$ in the direction $\zeta$.
    \end{lem}
    \begin{proof}
        Since $F_\xi=F_\eta$, we have
        \begin{align*}
            \E\l[ \phi(\xi) \r] = \E\l[ \phi(\eta)\r]
        \end{align*}
        for all $\phi\in \Lip(\Rbb^d)$. 
        Let $\phi\in\Lip(\Rbb^d)$, then $y\mapsto \phi(y+x)$ is Lipschitz for all $x\in \Rbb^d$. Since $\xi$ and $\eta$ are independent of $\zeta$, we have
        \begin{align*}
            \E\l[ \phi(\xi + \lambda \zeta) \r] 
            % independence
            &= \E\l[ \E\l[ \phi(\xi + x)\r] \Big|_{x=\lambda \zeta} \r] 
            = \E\l[ \E\l[ \phi(\eta + x)\r] \Big|_{x=\lambda \zeta} \r] = \E\l[ \phi(\eta+\lambda \zeta) \r].
        \end{align*}
        Since this holds for all $\phi\in\Lip(\Rbb^d)$, we obtain $F_{\xi+\lambda \zeta}=F_{\eta+\lambda \zeta}$ for all $\lambda> 0$. By the Gateaux differentiability of $\hat f$, we have
        \begin{align*}
            0 
            = 
            \lim_{\lambda \rightarrow 0} \frac{\l| \hat f(\xi+\lambda \zeta)-\hat f(\xi) - \lambda \partial \hat f\!\l(\xi;\zeta\r) \r|}{\lambda}
            &= 
            %\lim_{\lambda \rightarrow 0} \frac{\l|  f(F_{\xi+\lambda \zeta})- f(F_\xi) - \lambda \partial \hat f\!\l(\xi;\zeta \r) \r|}{\lambda}
            %\\&= 
            %\lim_{\lambda \rightarrow 0} \frac{\l|  f(F_{\eta+\lambda \zeta})- f(F_\eta) - \lambda \partial \hat f\!\l(\xi;\zeta\r)  \r|}{\lambda }
            %\\&=
            \lim_{\lambda \rightarrow 0} \frac{\l| \hat f(\eta+\lambda \zeta)-\hat f(\eta) - \lambda \partial \hat f\!\l(\xi; \zeta\r) \r|}{\lambda}
            .
        \end{align*}
        Thus, $\hat f$ is Gateaux differentiable at $\eta$ and we conclude $\partial \hat f\!\l(\xi;\zeta \r) = \partial \hat f\!\l(\eta;\zeta\r)$ from the uniqueness of the Gateaux derivative.
    \end{proof}

    From Lemma~\ref{lem:derivative-0}, we immediately obtain the following result. 

    \begin{cor}\label{cor:derivative-0}
        Let $f:\,\mathcal D_0\rightarrow \Rbb$ be such that its lifting $\hat f:\,\L^{2,d}\rightarrow \Rbb$ is Gateaux differentiable at $\xi\in\L^{2,d}$. If $\eta \in \L^{2,d}$ is such that $F_\xi=F_\eta$, then $\hat f$ is Gateaux differentiable at $\eta$ and 
        \begin{align*}
            \partial \hat f\!\l(\xi;x\r) = \partial \hat f\!\l(\eta;x\r)
        \end{align*}
        for all $x\in\Rbb^d$.
    \end{cor}

    For defining the derivative of $f$, we will require its lifting $\hat f$ to be Gateaux differentiable in any deterministic direction with Lipschitz Gateaux derivative. Clearly, this condition is weaker than requiring the lifting $\hat f$ to be Fréchet differentiable everywhere.

    \begin{dfn}
        Let $f:\,\mathcal D_0 \rightarrow \Rbb$. We say that $f$ is differentiable if its lifting $\hat f$ is Gateaux differentiable at $\xi$ in the direction $x$ with Lipschitz Gateaux derivative for any $\xi\in\L^{2,d}$ and $x\in \Rbb^d$, and its derivative $\partial f:\,\mathcal D_0\times \mathcal D_0 \rightarrow \Rbb$ is given by
        \begin{align*}
            \partial f(F_\xi,F_\eta) = F_\eta(x\mapsto \partial f(\xi;x) )=\E\l[ \partial \hat f\!\l( \xi;x\r)\Big|_{x=\eta }\r]
        \end{align*}
    \end{dfn}

    Since the map $x\mapsto \partial f(\xi;x)$ is Lipschitz for any $\xi$, we can apply $F\in \mathcal D_0$ to it. In particular, this ensures that $\partial f(F_\xi, F_\eta)=\partial f(F_\xi,F_\zeta)$ for any $\xi,\eta,\zeta\in \L^{2,d}$ with $F_\eta=F_\zeta$. Moreover, Lemma~\ref{cor:derivative-0} ensures that $\partial f(F_\xi,F_\zeta)=\partial f(F_\eta,F_\zeta)$ for all $\xi,\eta,\zeta \in \L^{2,d}$ with $F_\xi=F_\eta$. Thus, the derivative $\partial f$ is well-defined.
    
    \appendix

    \section{Conditional Sublinear Expectation}\label{sec:aux}
    
    \begin{lem}\label{lem:conditional-dt}
        Let $0\leq t\leq T$ and $X\in\S(0,T)$.
        Then
        \begin{equation*}
            \E\l[ \int_t^T  X_s \d s \,\bigg|\,\CF_t\r] \leq \int_t^T \E\l[ X_s \,|\,\CF_t\r] \d s.
        \end{equation*}
    \end{lem}
    \begin{proof}
        Since $X\in\S(0,T)$, there exist $m\in\mathbb{N}$, $t=t_0<\ldots<t_m=T$, and $\xi_k\in\F(\CF_{t_k})$, $0\leq k\leq m-1$ such that
        \begin{equation*}
            X\ind{[t,T]}=\sum_{k=0}^{m-1} \xi_k\ind{\l[t_k,t_{k+1}\r)},
        \end{equation*}
        and
        \begin{equation*}
            \int_t^T X_s \d s = \sum_{k=0}^{m-1} \xi_k \l( t_{k+1}-t_k\r).
        \end{equation*}
        Due to the sublinearity of the conditional expectation, we obtain
        \begin{align*}
            \E\l[ \int_t^T X_s \d s \,\bigg|\,\CF_t\r] 
            % representation as sum
            &= \E\l[ \sum_{k=0}^{m-1} \xi_k \l( t_{k+1}-t_k\r) \,\bigg|\,\CF_t\r]
            % sublinearity
            \\&\leq \sum_{k=0}^{m-1} \E\l[  \xi_k  \,|\,\CF_t\r] \l( t_{k+1}-t_k\r)
            % positive homogeneity
            \\&= \int_t^T \E\l[ X_s \,|\,\CF_t\r] \d s.
        \end{align*}
    \end{proof}

    \begin{cor}\label{cor:conditional-dt}
        Let $p\geq 1$, $0\leq t\leq T$ and $X\in\M^p(0,T)$.
        Then
        \begin{equation*}
            \E\l[ \l| \int_t^T  X_s \d s \r|^p \,\bigg|\,\CF_t\r] \leq \l(T-t\r)^{p-1} \int_t^T \E\l[ \l| X_s\r|^p \,|\,\CF_t\r] \d s.
        \end{equation*}
    \end{cor}
    \begin{proof}
        This follows immediately from the construction of $\M^p(0,T)$ and Jensen's inequality. % conditional expectation continuous w.r.t. L1-norm
    \end{proof}
    
    \begin{lem}\label{lem:conditional-dQV}
        Let $a\in\Rbb^n$, $p\geq 1$, $0\leq t\leq T$ and $X\in\M^p(0,T)$.
        Then
        \begin{equation*}
            \E\l[ \l| \int_t^T X_s \d \l<B^a\r>_s\r|^p \, \bigg|\,\CF_t\r] \leq \l( T-t \r)^{p-1} \overline{\sigma}_{aa}^{2p} \int_t^T \E\l[ \l| X_s\r|^p \,|\,\CF_t\r] \d s .
        \end{equation*}
    \end{lem}
    \begin{proof}
        By Corollary 3.5.5 in \cite{peng_nonlinear_2019}, we have
        \begin{equation*}
            \l| \l<B^a\r>_{t_{k+1}}-\l<B^a\r>_{t_{k}} \r|
            \leq
            \overline{\sigma}_{aa}^2 \l(t_{k+1}-t_k\r).
        \end{equation*}
        Thus, Jensen's inequality yields
        \begin{align*}
            \l| \int_t^T X_s \d \l<B^a\r>_s\r|^p
            % Corollary 3.5.5
            &\leq   
            \l| \int_t^T \l|X_s\r| \overline{\sigma}_{aa}^2 \d s \r|^p 
            % Jensen's inequality
            \\&\leq
            \l(T-t\r)^{p-1} \overline{\sigma}_{aa}^{2p} \int_t^T \l| X_s\r|^p \d s.
        \end{align*}
        Finally, Corollary~\ref{cor:conditional-dt} yields the desired result.
    \end{proof}

    \begin{lem}\label{lem:conditional-dB}
        Let $a\in\Rbb^n$, $p\geq 2$, $0\leq t\leq T$ and $X\in\M^p(0,T)$.
        Then
        \begin{equation*}
            \E\l[ \sup_{t\leq w\leq T}\l| \int_t^w X_s \d B^a_s\r|^p \,\bigg|\,\CF_t\r] 
            \leq 
            \l( T-t\r)^{\frac{p-2}{2}} \overline{\sigma}_{aa}^p \int_t^T \E\l[ \l|X_s\r|^p \,|\, \CF_t\r] \d s.
        \end{equation*}
    \end{lem}
    \begin{proof}
        The Burkholder-Davis-Gundy inequality yields
        \begin{align*}
            \E\l[ \sup_{t\leq w\leq T}\l| \int_t^w X_s \d B^a_s\r|^p \,\bigg|\,\CF_t\r] 
            % BDG-inequality
            &\leq 
            C_p
            \E\l[ \l| \int_t^T X_s^2 \d \l<B^a\r>_s\r|^{\frac{p}{2}} \,\bigg|\,\CF_t\r]
            % lem:conditional-dQV
            \\&\leq
            C_p\l(T-t\r)^{\frac{p-2}{2}} \overline{\sigma}_{aa}^{p} \int_t^T
            \E\l[ \l| X_s \r|^p \,|\,\CF_t\r] \d s,
        \end{align*}
        where the last step follows from Lemma~\ref{lem:conditional-dQV}.
    \end{proof}

    \begin{lem}\label{lem:conditional}
        Let $p\geq 2$, $0\leq t\leq T$, $\xi\in\L^{p,d}(t)$ and $b_{k},h_{kij}, g_{ki}\in \M^p(0,T)$ for $1\leq k\leq d$, $1\leq i,j\leq n$. 
        Let $X$ satisfy
        \begin{align*}
            \d X_s &= b(s) \d s + h(s) \d \l<B\r>_s + g(s) \d B_s, \qquad t\leq s\leq T
            \\ X_t &= \xi.
        \end{align*}
        Then
        \begin{flalign*}
            \quad 
            \E&\l[ \sup_{t\leq s\leq w} \l\|X_s \r\|^p \,\Big| \, \CF_t \r]
            &&
            \\&
            \lesssim \l\|\xi\r\|^p + \sum_{k=1}^d \sum_{i,j=1}^n \int_t^w \E\l[ \l| b_k(s) \r|^p \,|\,\CF_t\r] + \E\l[ \l|h_{kij}(s)\r|^p \,|\,\CF_t\r] + \E\l[ \l|g_{ki}(s)\r|^p \,|\,\CF_t\r] \d s .
        \end{flalign*}
    \end{lem}
    \begin{proof}
        Follows from Corollary \ref{cor:conditional-dt} and Lemmas \ref{lem:conditional-dQV} and \ref{lem:conditional-dB}.
    \end{proof}

    \printbibliography
\end{document}

%% file: alphabets.tex
\newcommand{\Ebb}{\mathbb{E}}
\newcommand{\Fbb}{\mathbb{F}}

\newcommand{\Rbb}{\mathbb{R}}

\newcommand{\CA}{\mathcal{A}}

\newcommand{\CF}{\mathcal{F}}

\newcommand{\CP}{\mathcal{P}}

%% file: environments.tex
\newtheorem{prp}{Proposition}[section]

\newtheorem{lem}[prp]{Lemma}
\newtheorem{ass}[prp]{Assumption}
\newtheorem{cor}[prp]{Corollary}
\theoremstyle{definition}
\newtheorem{dfn}[prp]{Definition}
\newtheorem{rmk}[prp]{Remark}

\newtheorem{notation}[prp]{Notation}
\numberwithin{equation}{section}

%% file: references.bib
@article{li_doubly_2025,
	title = {Doubly reflected backward {SDEs} driven by {G}-{Brownian} motions and fully nonlinear {PDEs} with double obstacles},
	copyright = {2025 The Author(s)},
	issn = {2194-041X},
	url = {https://link.springer.com/article/10.1007/s40072-025-00354-3},
	doi = {10.1007/s40072-025-00354-3},
	language = {en},
	urldate = {2025-07-24},
	journal = {Stochastics and Partial Differential Equations: Analysis and Computations},
	author = {Li, Hanwu and Ning, Ning},
	month = mar,
	year = {2025},
	note = {Company: Springer
Distributor: Springer
Institution: Springer
Label: Springer
Publisher: Springer US},
	pages = {1--40},
}

@misc{hu_bsdes_2022,
	title = {{BSDEs} driven by {G}-{Brownian} motion under degenerate case and its application to the regularity of fully nonlinear {PDEs}},
	url = {http://arxiv.org/abs/2205.09164},
	doi = {10.48550/arXiv.2205.09164},
	urldate = {2025-07-24},
	publisher = {arXiv},
	author = {Hu, Mingshang and Ji, Shaolin and Li, Xiaojuan},
	month = may,
	year = {2022},
	note = {arXiv:2205.09164 [math]},
}

@misc{peng_g-expectation_2006,
	title = {G-{Expectation}, {G}-{Brownian} {Motion} and {Related} {Stochastic} {Calculus} of {Ito}'s type},
	url = {http://arxiv.org/abs/math/0601035},
	doi = {10.48550/arXiv.math/0601035},
	urldate = {2025-07-21},
	publisher = {arXiv},
	author = {Peng, Shige},
	month = dec,
	year = {2006},
	note = {arXiv:math/0601035
version: 2},
}

@article{sun_mean-field_2020,
	title = {Mean-field backward stochastic differential equations driven by {G}-{Brownian} motion and related partial differential equations},
	volume = {43},
	issn = {1099-1476},
	doi = {10.1002/mma.6573},
	number = {12},
	urldate = {2022-06-07},
	journal = {Mathematical Methods in the Applied Sciences},
	author = {Sun, Shengqiu},
	year = {2020},
	pages = {7484--7505},
}

@article{mckean_class_1966,
	title = {A {Class} of {Markov} {Processes} {Associated} with {Nonlinear} {Parabolic} {Equations}},
	volume = {56},
	issn = {0027-8424, 1091-6490},
	url = {https://www.pnas.org/content/56/6/1907},
	doi = {10.1073/pnas.56.6.1907},
	number = {6},
	urldate = {2021-12-14},
	journal = {Proceedings of the National Academy of Sciences},
	author = {McKean, Henry P.},
	month = dec,
	year = {1966},
	pages = {1907--1911},
}

@article{bollweg_mean-field_2025,
	title = {Mean-field stochastic differential equations driven by {G}-{Brownian} motion},
	copyright = {http://creativecommons.org/licenses/by/3.0/},
	issn = {2095-9672},
	url = {https://www.aimsciences.org/en/article/doi/10.3934/puqr.2025011},
	doi = {10.3934/puqr.2025011},
	urldate = {2025-05-08},
	journal = {Probability, Uncertainty and Quantitative Risk},
	author = {Bollweg, Karl-Wilhelm Georg and Meyer-Brandis, Thilo},
	month = apr,
	year = {2025},
	note = {Publisher: Probability, Uncertainty and Quantitative Risk},
	pages = {1--24},
}

@book{carmona_probabilistic_2018,
	series = {Probability theory and stochastic modelling},
	title = {Probabilistic theory of mean field games with applications {I}: {Mean} field {FBSDEs}, control, and games},
	volume = {1st},
	isbn = {978-3-319-56437-1 978-3-319-58920-6},
	shorttitle = {Probabilistic theory of mean field games with applications {II}},
	number = {volume 83},
	author = {Carmona, René and Delarue, Franc̦ois},
	year = {2018},
}

@book{peng_nonlinear_2019,
	address = {Berlin [Heidelberg]},
	series = {Probability theory and stochastic modelling},
	title = {Nonlinear expectations and stochastic calculus under uncertainty: with {Robust} {CLT} and {G}-{Brownian} motion},
	isbn = {978-3-662-59902-0},
	shorttitle = {Nonlinear expectations and stochastic calculus under uncertainty},
	number = {volume 95},
	publisher = {Springer},
	author = {Peng, Shige},
	year = {2019},
	doi = {10.1007/978-3662-59903-7},
}

@inproceedings{sznitman_topics_1991,
	address = {Berlin, Heidelberg},
	series = {Lecture {Notes} in {Mathematics}},
	title = {Topics in propagation of chaos},
	isbn = {978-3-540-46319-1},
	doi = {10.1007/BFb0085169},
	booktitle = {Ecole d'{Eté} de {Probabilités} de {Saint}-{Flour} {XIX} — 1989},
	publisher = {Springer},
	author = {Sznitman, Alain-Sol},
	editor = {Burkholder, Donald L. and Pardoux, Etienne and Sznitman, Alain-Sol and Hennequin, Paul-Louis},
	year = {1991},
	pages = {165--251},
}

@misc{sun_distribution_2023,
	title = {On distribution dependent stochastic differential equations driven by {G}-{Brownian} motion},
	url = {http://arxiv.org/abs/2302.12539},
	urldate = {2023-04-24},
	publisher = {arXiv},
	author = {Sun, De and Wu, Jiang-Lun and Wu, Panyu},
	month = feb,
	year = {2023},
	note = {arXiv:2302.12539 [math]},
}

@article{biagini_non-linear_2023,
	title = {Non-linear affine processes with jumps},
	volume = {8},
	issn = {2095-9672},
	url = {https://www.aimsciences.org/en/article/doi/10.3934/puqr.2023010},
	doi = {10.3934/puqr.2023010},
	language = {en},
	number = {2},
	urldate = {2023-11-11},
	journal = {Probability, Uncertainty and Quantitative Risk},
	author = {Biagini, Francesca and Bollweg, Georg and Oberpriller, Katharina},
	month = mar,
	year = {2023},
	pages = {235--266},
}

@article{peng_new_2008,
	title = {A {New} {Central} {Limit} {Theorem} under {Sublinear} {Expectations}},
	url = {http://arxiv.org/abs/0803.2656},
	urldate = {2021-05-13},
	author = {Peng, Shige},
	month = mar,
	year = {2008},
	note = {arXiv: 0803.2656},
}

@article{peng_filtration_2004,
	title = {Filtration {Consistent} {Nonlinear} {Expectations} and {Evaluations} of {Contingent} {Claims}},
	volume = {20},
	issn = {1618-3932},
	url = {https://doi.org/10.1007/s10255-004-0161-3},
	doi = {10.1007/s10255-004-0161-3},
	language = {en},
	number = {2},
	urldate = {2023-12-18},
	journal = {Acta Mathematicae Applicatae Sinica, English Series},
	author = {Peng, Shige},
	month = jun,
	year = {2004},
	pages = {191--214},
}

@article{hu_g-levy_2021,
	title = {G-{Lévy} processes under sublinear expectations},
	volume = {6},
	copyright = {http://creativecommons.org/licenses/by/3.0/},
	issn = {2095-9672},
	url = {https://www.aimsciences.org/en/article/doi/10.3934/puqr.2021001},
	doi = {10.3934/puqr.2021001},
	language = {en},
	number = {1},
	urldate = {2023-11-18},
	journal = {Probability, Uncertainty and Quantitative Risk},
	author = {Hu, Mingshang and Peng, Shige},
	month = mar,
	year = {2021},
	note = {Publisher: Probability, Uncertainty and Quantitative Risk},
	pages = {1--22},
}

@article{buckdahn_mean-field_2014,
	title = {Mean-field stochastic differential equations and associated {PDEs}},
	url = {http://arxiv.org/abs/1407.1215},
	urldate = {2021-05-12},
	author = {Buckdahn, Rainer and Li, Juan and Peng, Shige and Rainer, Catherine},
	month = jul,
	year = {2014},
	note = {arXiv: 1407.1215},
}

@article{lasry_jeux_2006,
	title = {Jeux à champ moyen. {I} – {Le} cas stationnaire},
	volume = {343},
	issn = {1631-073X},
	url = {https://www.sciencedirect.com/science/article/pii/S1631073X06003682},
	doi = {10.1016/j.crma.2006.09.019},
	language = {fr},
	number = {9},
	urldate = {2022-06-08},
	journal = {Comptes Rendus Mathematique},
	author = {Lasry, Jean-Michel and Lions, Pierre-Louis},
	month = nov,
	year = {2006},
	pages = {619--625},
}

@article{lasry_jeux_2006-1,
	title = {Jeux à champ moyen. {II} – {Horizon} fini et contrôle optimal},
	volume = {343},
	issn = {1631-073X},
	url = {https://www.sciencedirect.com/science/article/pii/S1631073X06003670},
	doi = {10.1016/j.crma.2006.09.018},
	language = {fr},
	number = {10},
	urldate = {2022-06-08},
	journal = {Comptes Rendus Mathematique},
	author = {Lasry, Jean-Michel and Lions, Pierre-Louis},
	month = nov,
	year = {2006},
	pages = {679--684},
}

@article{peng_bsde_2016,
	title = {{BSDE}, path-dependent {PDE} and nonlinear {Feynman}-{Kac} formula},
	volume = {59},
	issn = {1869-1862},
	url = {https://doi.org/10.1007/s11425-015-5086-1},
	doi = {10.1007/s11425-015-5086-1},
	language = {en},
	number = {1},
	urldate = {2021-12-30},
	journal = {Science China Mathematics},
	author = {Peng, ShiGe and Wang, FaLei},
	month = jan,
	year = {2016},
	pages = {19--36},
}

@article{hu_backward_2014,
	title = {Backward stochastic differential equations driven by {G}-{Brownian} motion},
	volume = {124},
	issn = {0304-4149},
	url = {https://www.sciencedirect.com/science/article/pii/S0304414913002470},
	doi = {10.1016/j.spa.2013.09.010},
	language = {en},
	number = {1},
	urldate = {2021-12-25},
	journal = {Stochastic Processes and their Applications},
	author = {Hu, Mingshang and Ji, Shaolin and Peng, Shige and Song, Yongsheng},
	month = jan,
	year = {2014},
	pages = {759--784},
}

@article{kac_foundations_1956,
	title = {Foundations of {Kinetic} {Theory}},
	volume = {3.3},
	url = {https://projecteuclid.org/ebooks/berkeley-symposium-on-mathematical-statistics-and-probability/Proceedings-of-the-Third-Berkeley-Symposium-on-Mathematical-Statistics-and/chapter/Foundations-of-Kinetic-Theory/bsmsp/1200502194},
	urldate = {2021-12-07},
	journal = {Proceedings of the Third Berkeley Symposium on Mathematical Statistics and Probability, Volume 3: Contributions to Astronomy and Physics},
	author = {Kac, Mark},
	month = jan,
	year = {1956},
	note = {Publisher: University of California Press},
	pages = {171--198},
}

@article{cohen_quasi-sure_2012,
	title = {Quasi-sure analysis, aggregation and dual representations of sublinear expectations in general spaces},
	volume = {17},
	issn = {1083-6489},
	url = {https://projecteuclid.org/journals/electronic-journal-of-probability/volume-17/issue-none/Quasi-sure-analysis-aggregation-and-dual-representations-of-sublinear-expectations/10.1214/EJP.v17-2224.full},
	doi = {10.1214/EJP.v17-2224},
	number = {none},
	urldate = {2021-10-01},
	journal = {Electronic Journal of Probability},
	author = {Cohen, Samuel},
	month = jan,
	year = {2012},
}

@article{cardaliaguet_mean_2018,
	title = {Mean field game of controls and an application to trade crowding},
	volume = {12},
	issn = {1862-9679, 1862-9660},
	url = {http://link.springer.com/10.1007/s11579-017-0206-z},
	doi = {10.1007/s11579-017-0206-z},
	language = {en},
	number = {3},
	urldate = {2021-09-27},
	journal = {Mathematics and Financial Economics},
	author = {Cardaliaguet, Pierre and Lehalle, Charles-Albert},
	month = jun,
	year = {2018},
	pages = {335--363},
}

@book{carmona_probabilistic_2017,
	address = {New York, NY},
	edition = {1st edition},
	title = {Probabilistic theory of mean field games with applications {II}: mean field games with common noise and master equations},
	volume = {2nd},
	isbn = {978-3-319-56435-7},
	shorttitle = {Probabilistic theory of mean field games with applications {II}},
	publisher = {Springer Science+Business Media},
	author = {Carmona, René and Delarue, Franc̦ois},
	year = {2017},
}

@incollection{benth_g-expectation_2007,
	address = {Berlin, Heidelberg},
	title = {G-{Expectation}, {G}-{Brownian} {Motion} and {Related} {Stochastic} {Calculus} of {Itô} {Type}},
	isbn = {978-3-540-70846-9},
	url = {http://link.springer.com/10.1007/978-3-540-70847-6_25},
	language = {en},
	urldate = {2020-12-29},
	booktitle = {Stochastic {Analysis} and {Applications}},
	publisher = {Springer Berlin Heidelberg},
	author = {Peng, Shige},
	editor = {Benth, Fred Espen and Di Nunno, Giulia and Lindstrøm, Tom and Øksendal, Bernt and Zhang, Tusheng},
	year = {2007},
	doi = {10.1007/978-3-540-70847-6_25},
	pages = {541--567},
}

@article{neufeld_nonlinear_2016,
	title = {Nonlinear {Lévy} {Processes} and their {Characteristics}},
	volume = {369},
	issn = {1088-6850},
	url = {http://arxiv.org/abs/1401.7253},
	doi = {10.1090/tran/6656},
	language = {en},
	number = {1},
	urldate = {2020-12-29},
	journal = {Transactions of the American Mathematical Society},
	author = {Neufeld, Ariel and Nutz, Marcel},
	month = mar,
	year = {2016},
	note = {arXiv: 1401.7253},
	pages = {69--95},
}

@book{hollender_levy-type_2016,
	address = {Dresden},
	title = {Lévy-type processes under uncertainty and related nonlocal equations},
	isbn = {978-1-5355-5384-1},
	language = {eng},
	publisher = {Technische Universität Dresden},
	author = {Hollender, Julian},
	collaborator = {Technische Universität Dresden},
	year = {2016},
	note = {OCLC: 962061637},
}

@article{nutz_constructing_2013,
	title = {Constructing sublinear expectations on path space},
	volume = {123},
	issn = {03044149},
	url = {https://linkinghub.elsevier.com/retrieve/pii/S0304414913001063},
	doi = {10.1016/j.spa.2013.03.022},
	language = {en},
	number = {8},
	urldate = {2021-01-03},
	journal = {Stochastic Processes and their Applications},
	author = {Nutz, Marcel and van Handel, Ramon},
	month = aug,
	year = {2013},
	pages = {3100--3121},
}

@article{peng_multi-dimensional_2008,
	title = {Multi-dimensional {G}-{Brownian} motion and related stochastic calculus under {G}-expectation},
	volume = {118},
	issn = {03044149},
	url = {https://linkinghub.elsevier.com/retrieve/pii/S0304414907002128},
	doi = {10.1016/j.spa.2007.10.015},
	language = {en},
	number = {12},
	urldate = {2020-12-29},
	journal = {Stochastic Processes and their Applications},
	author = {Peng, Shige},
	month = dec,
	year = {2008},
	pages = {2223--2253},
}

@article{fadina_affine_2019,
	title = {Affine processes under parameter uncertainty},
	volume = {4},
	issn = {2367-0126},
	url = {https://probability-risk.springeropen.com/articles/10.1186/s41546-019-0039-1},
	doi = {10.1186/s41546-019-0039-1},
	language = {en},
	number = {1},
	urldate = {2020-12-29},
	journal = {Probability, Uncertainty and Quantitative Risk},
	author = {Fadina, Tolulope and Neufeld, Ariel and Schmidt, Thorsten},
	month = dec,
	year = {2019},
	pages = {5},
}
